\documentclass[sn-mathphys-num]{sn-jnl}

\usepackage{graphicx}%
\usepackage{mathtools}
\usepackage{multirow,longtable}%
\usepackage{amsmath,amssymb,amsfonts}%
\usepackage{amsthm}%
\usepackage{mathrsfs}%
\usepackage[title]{appendix}%
\usepackage{xcolor}%
\usepackage{textcomp}%
\usepackage{manyfoot}%
\usepackage{booktabs}%
\usepackage{algorithm}%
\usepackage{algorithmicx}%
\usepackage{algpseudocode}%
\usepackage{listings}%
\usepackage[justification=centering]{caption}

\theoremstyle{thmstyleone}%
\newtheorem{theorem}{Theorem}
\newtheorem{lemma}{Lemma}%
\theoremstyle{thmstyletwo}%
\newtheorem{remark}{Remark}%

\theoremstyle{thmstylethree}%

\begin{document}

\title[An Adaptive Cubic Regularisation Algorithm Based on Interior-Point Methods for Optimization with General Inequality Constraints]{An Adaptive Cubic Regularisation Algorithm Based on Interior-Point Methods for Optimization with General Inequality Constraints}


\author[1]{\fnm{Yonggang} \sur{Pei}}\email{peiyg@htu.edu.cn}

\author[1]{\fnm{Jingyi} \sur{Guo}}\email{gggggjy104@163.com}

\author[2]{\fnm{Detong} \sur{Zhu}}\email{dtzhu@shnu.edu.cn}


\affil[1]{\orgdiv{School of Mathematics and Statistics}, \orgname{Henan Normal University}, \city{Xinxiang},\orgaddress{\street{} \postcode{453007},~\state{}\country{China}}}

\affil[2]{\orgdiv{Mathematics and Science College}, \orgname{Shanghai Normal University},\orgaddress{\street{} \city{Shanghai}, \postcode{200234},~\state{}\country{China}}}


\abstract{Nonlinear constrained optimization has a wide range of practical applications. The interior-point method is considered to be one of the most powerful algorithms for solving nonlinear inequality constrained optimization. In this paper, we consider optimization with general inequality constraints and propose an Adaptive Regularisation algorithm using Cubics Based on Interior-Point methods (ARCBIP). For solving the barrier problem, we construct ARC subproblem with linearized constraints and the well-known fraction to the boundary rule that prevents slack variables from approaching their lower bounds prematurely. 
We employ a composite-step approach and reduced Hessian methods to deal with linearized constraints, where the trial step is decomposed into a normal step and a tangential step. They are obtained by solving two ARC subproblems approximately with the fraction to the boundary rule. Requirements on normal steps and tangential steps are given to ensure global convergence. To determine whether the trial step is accepted, we use exact penalty function as the merit function in ARC framework. The updating of the barrier parameter is implemented by adaptive strategies. Global convergence is analyzed under mild assumptions. Preliminary numerical experiments and some comparison results are reported.
}


\keywords{Nonlinear constrained optimization, Interior-point method, Adaptive regularisation using cubics, Global convergence}



\maketitle

\section{Introduction}\label{sec1}
Nonlinear constrained optimization problems are ubiquitous in many important fields such as economics, management, and engineering \cite{Nocedal1999}. In this paper, we focus on the inequality constrained optimization problems of the form 
\begin{subequations}\label{ori}
\begin{align}
\underset{x}{\text{minimize}} \quad &~ f(x) \\
\text{subject to} \quad &~ g(x)\leq0,
\end{align}
\end{subequations}
where $ f: \mathbf{R}^n \rightarrow \mathbf{R} $ and $ g: \mathbf{R}^n \rightarrow \mathbf{R}^m $ are smooth functions. 

Since the 1950s, various relatively effective methods have been proposed to solve constrained optimization problems. 
Interior-point methods are one of the most powerful algorithms for solving inequality constrained optimization problems \cite{Forsgren2002,Nemirovski2008,Nocedal1999,Ye2011}. For example, Lee et al. \cite{lee2014interior} proposed interior-point algorithms for $P_*(\kappa)$-linear complementarity problem based on a new class of kernel functions. Ill{\'e}s et al. \cite{illes2010polynomial} presented a polynomial path-following interior point algorithm for general linear complementarity problems. Lin et al. \cite{Lin2021} and Deng et al. \cite{Deng2024} presented alternating direction method of multipliers (ADMM) based interior-point methods for solving large-scale linear and conic optimization. Zhang et al. \cite{zhang2023iprqp} proposed a primal-dual interior-point relaxation algorithm for convex quadratic programming (IPRQP). He et al. \cite{LZS2024COAP} proposed a Newton-CG based barrier-augmented Lagrangian method for general nonconvex conic optimization, which can be extended to solve inequality constrained optimization. Vanderbei and Shanno \cite{Vanderbei1999} described an interior-point algorithm for nonconvex nonlinear programming which is a direct extension of interior-point methods for linear and quadratic programming. Combining the augmented Lagrangian and the interior-point technique, Liu et al. introduced a class of effective methods for nonlinear programs \cite{Liu2020,Liu2022,Liu2023}. Another important category of methods is the combination of trust region methods and interior point methods \cite{Byrd2000,Byrd1999,Coleman1996,Dennis1998,Tseng2002,Villalobos2005,Waltz2006}. In these methods, one has to handle the incompatibility of the intersection of linearized constraints with trust-region bounds.

The adaptive regularisation algorithm using cubics (ARC) was first proposed to solve unconstrained optimization problems by Cartis, Gould and Toint \cite{Cartis2011A} which has roots in earlier algorithmic proposals by Griewank \cite{Griewank1981} and Nesterov and Polyak \cite{Neaterov2006}. The idea of ARC is that in each iteration, as long as the Hessian of the objective function is locally Lipschitz continuous, it is replaced by an approximate cubic model, where the adaptive parameters in the cubic model are dynamically adjusted to ensure optimal performance during the iteration process. ARC has good convergence properties, worst-case function- and derivative-evaluation complexity, and promising numerical experiments  \cite{Cartis2011A2}. There have been many methods (\cite{Bellavia2021,Benson2018,Bergou2018,Bianconcini2016,Cartis2020,Dehghan2021,Dussault2018,Lu2012,Martinez2017,Park2020,Zhao2021}) on using ARC (or its variants) for unconstrained optimization. And some methods (\cite{Bergou2017,Gould2012,Hsia2017,Jia2022,Jiang2021,Lieder2020}) are designed to solve or use ARC subproblems more efficiently. For constrained optimization problems, many effective methods have been proposed  based on the ARC framework. Lubkoll et al. \cite{Lubkoll2017} proposed a composite-step method specifically designed for solving equality constrained optimization problems involving partial differential equations. Cartis et al. \cite{Cartis2019} discovered a general class of adaptive regularization methods that use first- or higher-order local Taylor models of the objective regularized by a(ny) power of the step. They also generalized the approach in \cite{Grapiglia2017} to optimization problems with convex constraints and inexact subproblem solutions. Agarwal \cite{Agarwal2021} extended the ARC method to optimization problems defined on Riemannian manifolds.
Benson \cite{Benson2014} demonstrated that the Levenberg-Marquardt (LM) perturbation is equivalent to replacing the Newton step with a cubic regularisation step, where the regularisation parameter is chosen appropriately. Pei et al. \cite{pei2023} presented a sequential ARC algorithm for solving equality constrained optimization problem.

In this paper, we aim to propose an implementable ARC algorithm Based on Interior-Point methods (ARCBIP) for solving nonlinear inequality constrained optimization problems. Following the strategy of interior-point methods, we first focus on solving barrier problems accompanying with the original problem \eqref{ori}. 
In each iteration, by analogy with sequential quadratic programming (SQP) methods, we construct an ARC subproblem with linearized constraints and the well-known fraction to the boundary rule \cite{Nocedal1999} that prevents slack variables from approaching their lower bounds of zero too early. 
Instead of solving ARC subproblem with linearized constraints and the fraction to the boundary rule directly, we employ composite-step approach \cite{Conn2000,Lubkoll2017OMS} and reduced Hessian methods to deal with linearized constraints, where the trial step is decomposed into a normal step and a tangential step. The normal step aims to reduce constraint violations, while the tangential step ensures sufficient model reduction. They are computed by solving two standard ARC subproblems (approximately) with the fraction to the boundary rule that can be handled by the backtrack strategy \cite{Nocedal1999}. And we present conditions for normal steps and tangential steps to be satisfied to ensure the convergence. After the trial step has been obtained, the exact penalty function is used as the merit function in ARC framework, and the ratio of reduction of the penalty function to reduction of the model is calculated to determine whether the trial point is accepted according to ARC mechanism.
 The updating of the barrier parameter is implemented by adaptive strategies \cite{Nocedal1999}, which is based on the deviation of the smallest complementary condition.
 Under mild assumptions, we establish the global convergence of our algorithm. Preliminary numerical experiments and some comparison results are reported.
 
 Although we draw inspiration from the SQP method when solving the barrier problem, there are substantial differences between the entire algorithm ARCBIP we proposed and the SQP method framework for solving inequality-constrained optimization problems.  From our point of view, the key differences between the ARCBIP and SQP can be summarized in three aspects. First, structurally, they differ in their framework: SQP features a single-loop structure, whereas ARCBIP adopts a double-loop structure. Second, their primary computational costs differ: for SQP, the main computational effort lies in solving quadratic programming (QP) subproblems with linear inequality constraints, which are typically addressed using active-set methods or interior-point methods \cite{Nocedal1999}. In contrast, ARCBIP expends most of its computational resources on solving the barrier problem within the inner loop. Since the barrier problem is a nonlinear optimization problem with equality constraints, the algorithm is constructed by analogy with the SQP idea, where the subproblem uses an ARC model instead of a QP model. Third, their convergence analyses and proofs are distinct, which stems from the differences in the ideas and logics underlying the two algorithms.

The remainder of this paper is organized as follows. In Section 2, we describe the development of ARCBIP. In Section 3, global convergence properties of the first-order critical points is presented. Preliminary numerical results are reported in Section 4, and the conclusion is provided in Section 5.

Throughout the paper, $\|\cdot\|$ denotes the Euclidean norm. The inner product of vectors $a,b\in\mathbb{R}^n$ is denoted by $a^T b$. The symbol $y^{(i)}$ represents the $i$-th component of vector $y$. The symbol $ A^{(i)} $ denotes the $i$-th column of $ A $. For a scalar $ \alpha $, we define $ \alpha^+ := \max \{0,\alpha\} $, while for a vector $u$, $ u^+ $ is defined component-wise by $ (u^+)^{(i)} = (u^{(i)})^+ $.

\section{ARCBIP for Nonlinear Inequality Constrained Optimization}\label{sec2}

In this section, we describe the development of ARCBIP in detail. 
Based on the concept of the interior-point method, we relate problem \eqref{ori} to the following barrier problem with $x$ and $y$ as variables 
\begin{subequations}\label{barr}
\begin{align}
\underset{x,y}{\text{minimize}} \quad &~ f(x)-\mu \sum^m_{i=1}\ln y^{(i)} \\
\text{subject to} \quad &~ g(x)+y=0,
\end{align}
\end{subequations}
where the barrier parameter $\mu $ is strictly positive and where the vector of slack variables $y=(y^{(1)},\cdots, y^{(m)})^T $ is implicitly assumed to be positive. By letting $\mu $ converge to zero, the sequence of solutions to \eqref{barr} should normally converge to a stationary point of the original nonlinear program \eqref{ori}.

The Lagrangian of problem \eqref{barr} is
\begin{equation}\label{lag}
L(x,y,\lambda) :=f(x)-\mu \sum^m_{i=1}\ln y^{(i)}+\lambda^T (g(x)+y),
\end{equation}
where $\lambda \in \mathbf{R}^m $ is the Lagrange multiplier. At an optimal solution $
\begin{pmatrix}
  x \\
  y 
\end{pmatrix}$ of problem \eqref{barr}, from \eqref{lag} we have that
\begin{subequations}\label{Lagrad}
\begin{align}
\nabla_x L(x,y,\lambda)&=\nabla f(x)+A(x)\lambda=0, \\
\nabla_y L(x,y,\lambda)&=-\mu Y^{-1}e+\lambda=0,
\end{align}
\end{subequations}
where
\begin{equation}\label{Ax}
A(x)= \left( \nabla g^{(1)} (x),\cdots, \nabla g^{(m)} (x) \right) \nonumber
\end{equation}
represents the matrix of constraint gradients, and
\begin{equation}
Y=
\begin{pmatrix}
  y^{(1)} &   &   \\
    & \ddots &   \\
    &   &  y^{(m)} 
\end{pmatrix},
e=
\begin{pmatrix}
  1 \\
  \vdots \\
  1 
\end{pmatrix}. \nonumber
\end{equation}
We define the error function 
\begin{equation}\label{E}
  E(x,y; \mu) := \max \left\{ \| \nabla f(x) + A(x) \lambda \|, \| Y \lambda - \mu e \|, \| g(x)+y \| \right\},  \nonumber
\end{equation}
which is based on the perturbed KKT system of barrier problem \eqref{barr}.
We find a solution 
$\begin{pmatrix}
  x \\
  y 
\end{pmatrix}$ 
that satisfies
$E(x,y; \mu)<a \mu$, where $a$ is a constant. 

We define
\begin{equation}\label{zvarphi}
\varphi(z):=f(x)-\mu \sum^m_{i=1}\ln y^{(i)}, c(z):=g(x)+y
\end{equation}
with 
\begin{equation}
z:=
\begin{pmatrix}
  x \\
  y 
\end{pmatrix}. \nonumber
\end{equation}
The barrier problem (\ref{barr}) can be rewritten as 
\begin{subequations}\label{rebarr}
\begin{align}
\underset{z}{\text{minimize}} \quad &~ \varphi (z) \\
\text{subject to} \quad &~  c(z)=0.
\end{align}
\end{subequations}

The section is organized as follows. In Subsection 2.1, we focus on solving barrier problem. In Subsection 2.2, we supplement the algorithm details and provide a complete description of ARCBIP.

\subsection{Solving Barrier Problems}\label{Solving barrier problems}
At current iterate 
$z_k= \begin{pmatrix}
  x_k \\
  y_k 
\end{pmatrix}$, 
we need to generate a trial step
$d_k=
\begin{pmatrix}
  d_x \\
  d_y
\end{pmatrix} $
with the same blocks as $z_k$. To this end, analogous to SQP methods, we construct the following ARC subproblem with linearized constraints
\begin{subequations}\label{rebarrsqp}
\begin{align}
\underset{d}{\text{minimize}} \quad &~ \nabla \varphi( z_k )^T d +\frac{1}{2}d^T W_k d +\frac{1}{3}\sigma_k \|d \|^3\\
\text{subject to} \quad &~ \hat{A}(z_k)^T d + c(z_k)=0,
\end{align} 
\end{subequations}
where $ W_k $ represents the Hessian of the Lagrangian with respect to $ z_k $ for the barrier problem \eqref{rebarr}, adaptive parameter $\sigma_k >0$ is updated at every iteration and 
\begin{align}\label{Ahat}
\hat{A}(z_k)^T&:=
\begin{pmatrix}
  A(x_k)^T & I 
\end{pmatrix}
\end{align}
is the Jacobian of $c(z)$ at $z_k$.

From \eqref{Lagrad} we derive that
\begin{equation}\label{W}
W_k = \nabla^2_{zz} L(x_k,y_k,\lambda_k)=
\begin{pmatrix}
  \nabla^2_{xx} L(x_k,y_k,\lambda_k) & 0 \\
  0 & \mu Y_k^{-2} 
\end{pmatrix}.
\end{equation}



Considering that the slack variable should not approach zero prematurely, we introduce a scaling matrix $Y_k^{-1}$ that penalizes step $d_y$ near the boundary of the feasible region, that is, the cubic term of objective function of this subproblem can be written as
 $\frac{1}{3}\sigma_k \left\|
  \begin{pmatrix}
    d_x \\
    Y_k^{-1} d_y
  \end{pmatrix}
  \right\|^3 = \frac{1}{3} \sigma_k \|D_k d \|^3,$
  where 
\begin{equation}\label{D}
D_k :=
\begin{pmatrix}
  I & 0 \\
  0 & Y_k^{-1}
\end{pmatrix} 
\end{equation}
denotes coefficient matrix of $d$.

 The adaptive parameter in problem \eqref{rebarrsqp} does not prevent the new slack variable values $y_k + d_y$ from becoming their lower bounds of zero too quickly. 
 Hence, we impose the well-known fraction to the boundary rule
\begin{equation}
y_k + d_y \geq (1-\tau) y_k,\nonumber
\end{equation}
where the parameter $\tau \in(0, 1)$ is selected to be close to 1. This results in the subproblem
\begin{subequations}\label{cubicsubp}
\begin{align}
\underset{d}{\text{minimize}} \quad &~ \nabla \varphi(z_k)^T d + \cfrac{1}{2}d^T W_k d  + \cfrac{1}{3}\sigma_k \| D_k d \|^3\\
\text{subject to} \quad  &~ \hat{A}(z_k)^T d + c(z_k)=0,\\
  &~ d_y\geq-\tau y_k.
\end{align} 
\end{subequations}

We do not address subproblem \eqref{cubicsubp} directly. We employ the composite-step method to decompose the trial step $d_k$ into a normal step $n_k$ and a tangential step $t_k$, i.e.,
\begin{equation}\label{dcomposite}
d_k =   n_k + t_k, 
\end{equation}
where 
\begin{equation}\label{ntcompo}
n_k = 
\begin{pmatrix}
    n_x \\
    n_y
\end{pmatrix},
t_k = 
\begin{pmatrix}
    t_x \\
    t_y
\end{pmatrix} 
\end{equation}
both have the same blocks as $z_k$.
The main role of the normal step is to improve the feasibility of iteration $z_{k+1}$ and the tangential step is used to provide sufficient descent. 

To produce a sufficient decrease, a contraction parameter $0 <\xi\leq 1$ is introduced to restrict the norm of $n_k$ by adjusting the weight of cubic regularization item in subproblem \eqref{cubicsubp}. Then, we (approximately) solve the following problem
\begin{subequations}\label{norm1}
  \begin{align}
\underset{n}{\text{minimize}} \quad &~ \frac{1}{2} \|\hat{A}(z_k)^T n + c(z_k)\|^2 +\frac{\sigma_k}{3\xi^3}
\|D_k n\|^3 \\
\text{subject to} \quad &~ n_y\geq-\xi\tau y_k
  \end{align}
\end{subequations}
to get a normal step 
$n_k= \begin{pmatrix}
  n_x \\
  n_y 
\end{pmatrix}$.

Then the tangential step $t_k$ is computed by 
solving the following modification of problem \eqref{cubicsubp}.
\begin{subequations}\label{tang1}
\begin{align}
\underset{t}{\text{minimize}} \quad &~ \nabla \varphi(z_k)^T (n_k+t)+ \cfrac{1}{2}(n_k+t)^T W_k (n_k+t) + \cfrac{1}{3}\sigma_k \| D_k t \|^3 \\
\text{subject to} \quad &~ \hat{A}(z_k)^T (n_k+t)  = \hat{A}(z_k)^T n_k ,\\
  &~ d_y\geq-\tau y_k.
\end{align}
\end{subequations}
We shall discuss how to obtain $n_k$ and $t_k$ in detail in the following part.

\subsubsection{Computation of Normal Step}\label{211}
\noindent We first rewrite the normal problem \eqref{norm1} by \eqref{zvarphi}, \eqref{Ahat} and \eqref{ntcompo} as follows. 
\begin{subequations}\label{norm3}
  \begin{align}
\underset{n}{\text{minimize}} \quad &~ \frac{1}{2} \|g_k+ y_k + A_k^T n_x + n_y\|^2 +\frac{1}{3} \widetilde{\sigma}_k \left\|
\begin{pmatrix}
  n_x \\
  Y_k^{-1} n_y
\end{pmatrix}
\right\|^3\\
\text{subject to} \quad &~ n_y\geq-\xi\tau y_k,
  \end{align}
\end{subequations}
where $ A_k := A(x_k) $, $g_k := g(x_k) $ and
\begin{equation}\label{sigmahat}
\widetilde{\sigma}_k:=\frac{\sigma_k}{\xi^3}.
\end{equation}

Then, we propose two conditions the (approximate) solution $n_k$ of problem \eqref{norm3} must satisfy. 

Introduce variables 
\begin{equation}\label{u}
  u_x:=n_x,\quad u_y:=Y_k^{-1} n_y,
\end{equation}
and omit the constant term so that problem \eqref{norm3} becomes
\begin{subequations}\label{normu}
  \begin{align}
\underset{u}{\text{minimize}} \quad &~
  \left( g_k+ y_k \right)^T  \left( A_k^T \quad Y_k \right) u +  \frac{1}{2} u^T 
  \begin{pmatrix}
    A_k A_k^T & A_k Y_k \\
    Y_k A_k^T & Y_k^2 
  \end{pmatrix}
  u +\frac{1}{3} \widetilde{\sigma}_k \left\| u \right\|^3 \label{normua}\\
\text{subject to} \quad &~ u_y\geq- \xi \tau e, \label{normub}
  \end{align}
\end{subequations}
 where $ u := 
\begin{pmatrix}
  u_x \\
  u_y
\end{pmatrix}
 $ with the same blocks as $z_k$.
 
If the lower bound constraint \eqref{normub} is neglected,  the solution $u_\ast$ is to lie in the range space of 
\begin{equation}\label{range}
\binom{A_k}{Y_k}.
\end{equation}
In fact, according to Theorem 3.1 in \cite{Cartis2011A}, the necessary and sufficient conditions for $u_\ast$ to be a globally optimal solution of problem \eqref{normua} are 
\begin{equation}\label{optcon1}
\left[
\begin{pmatrix}
  A_k A_k^T & A_k Y_k \\
  Y_k A_k^T & Y_k^2 
\end{pmatrix}
+ \tilde{\sigma}_k \|u_\ast\|I \right]u_\ast= -
\begin{pmatrix}
  A_k  \\
  Y_k
\end{pmatrix}
(g_k+y_k) ,
\end{equation}
and 
$\begin{pmatrix}
  A_k A_k^T & A_k Y_k \\
  Y_k A_k^T & Y_k^2 
\end{pmatrix}
+ \tilde{\sigma}_k \|u_\ast\|I  \succeq 0$ 
for some $ \tilde{\sigma}_k \|u_\ast\| \geq 0 $. 

If $ \tilde{\sigma}_k \|u_\ast\| = 0 $, then 
 $u_\ast = -
 \begin{pmatrix}
  A_k \\
  Y_k
 \end{pmatrix}
 \left[
  \begin{pmatrix}
  A_k^T & Y_k
 \end{pmatrix}
  \begin{pmatrix}
  A_k \\
  Y_k
 \end{pmatrix}
 \right]^{-1} (g_k+y_k)$
is a solution, 
and it is obviously within the range space of (\ref{range}). 

Let $ N_k $ be a matrix whose columns form an orthogonal basis of the null space of $
\begin{pmatrix}   
A_k^T & I_k
\end{pmatrix}$. 
Then $\begin{pmatrix}   
A_k^T & I_k
\end{pmatrix} N_k = 0 $ 
and hence
\begin{equation}\label{nullN}
\begin{pmatrix}   
A_k^T & Y_k
\end{pmatrix} D_k N_k = 0. 
\end{equation}

 If $ \tilde{\sigma}_k \|u_\ast\| > 0 $, then premultiply both sides of (\ref{optcon1}) by $ (D_k N_k)^T $ to obtain $ (D_k N_k)^T u_\ast= 0$, showing that $u_\ast$ is in the range of (\ref{range}).

Hence, we require that the (approximate) solution should be in the range of (\ref{range}) even with the presence of a lower bound constraint in problem \eqref{norm3}, which can limit the magnitude of normal step $n_k$ and hence the tangential step can have a better chance to produce sufficient descent. 
In the implementation, if the (approximate) solution of problem \eqref{normu} violates the bounds \eqref{normub}, we can backtrack so that these bounds are satisfied.

For the analysis of this paper, it is sufficient to impose the following milder condition.

(a) Range Space Condition.

 The (approximate) solution $n_k$ of the normal problem \eqref{norm3} must follow the given form
\begin{equation}\label{ranspacon}
 n_k =
 \begin{pmatrix}
  A_k \\
  Y_k^2
 \end{pmatrix}\omega_k
\end{equation}
for some vector $\omega_k \in \mathbf{R}^m $, whenever problem \eqref{norm3} possesses an optimal solution of that form.
\\[0.5pt] 

The other condition on the normal step is related with sufficient reduction in the objective of problem \eqref{norm3}, which is a classical condition. We require that the reduction be comparable to the reduction obtained by minimizing along the steepest descent direction $u_k^c$ in $u$. 
\begin{equation}\label{ukc}
 u_k^c := -
 \begin{pmatrix}
  A_k \\
  Y_k
 \end{pmatrix}(g_k+y_k)
\end{equation}
is the gradient of the objective function at $u = 0$ as stated in problem \eqref{normu}.
 
By substituting the original variables, we obtain the vector
\begin{equation}
 n_k^c := -
 \begin{pmatrix}
  A_k \\
  Y_k^2
 \end{pmatrix}(g_k+y_k).\nonumber
\end{equation}
We consider the reduction in the objective of problem \eqref{norm3} by a step $n_k = 
\begin{pmatrix}
  n_x \\
  n_y
\end{pmatrix}$ as the \emph{normal predicted reduction}:
\begin{equation}\label{npred}
 \text{npred}_k(n_k):= \left\| g_k+y_k \right\| - \left\| g_k+y_k+A_k^T n_x +n_y \right\|  -\frac{1}{3} \widetilde{\sigma}_k \left\|
\begin{pmatrix}
  n_x \\
  Y_k^{-1} n_y
\end{pmatrix}
\right\|^3
\end{equation}
and we require such a reduction to meet the following condition.
\\[0.5pt] 

(b) Normal Cauchy Decrease Condition. 

The (approximate) solution $n_k$ of the normal problem \eqref{norm3} must satisfy
\begin{equation}\label{NCauchycon}
  \text{npred}_k(n_k) \geq \gamma_n \text{npred}_k(\alpha_k^c n_k^c)
\end{equation}
for some constant $\gamma_n > 0$, 
where $\alpha_k^c$ solves the problem
\begin{subequations}\label{NCauchypoi}
\begin{align}
\underset{\alpha \geq 0}{\text{minimize}} \quad &~ \cfrac{1}{2} \| g_k+ y_k + \alpha (A_k^T n_x + n_y)\|^2 +\frac{1}{3} \widetilde{\sigma}_k \left\|\alpha
\begin{pmatrix}
  n_x \\
  Y_k^{-1} n_y
\end{pmatrix}
\right\|^3\\
\text{subject to} \quad &~ \alpha n_y^c \geq -\xi \tau y_k ,
\end{align}
\end{subequations}
where $n_y^c$ is one block of $n_k^c := \begin{pmatrix}
                 n_x^c \\
                 n_y^c 
               \end{pmatrix}$ with the same blocks as $z_k$.
\\[0.5pt] 

Note that, since $ \alpha = 0 $ is a feasible solution to problem \eqref{NCauchypoi}, it is evident from \eqref{NCauchycon} that
\begin{equation}\label{npred0}
 \text{npred}_k (n_k) \geq 0.
\end{equation}

\subsubsection{Computation of Tangential Step}
\noindent At an iterate $
\begin{pmatrix}
  x_k \\
  y_k
\end{pmatrix}$, from the definitions \eqref{zvarphi}, \eqref{Ahat}-\eqref{D}, the tangential problem \eqref{tang1} can be expressed as
\begin{subequations}\label{tang3}
\begin{align}
\underset{d}{\text{minimize}} \quad &~\nabla f_k^T d_x - \mu e^T Y_k^{-1} d_y +\frac{1}{2}d_x^T B_k d_x +\frac{1}{2} \mu d_y^T Y_k^{-2} d_y \\ \nonumber
&~ +\frac{1}{3}\sigma_k \left\|
\begin{pmatrix}
  d_x - n_x \\
  Y_k^{-1} (d_y - n_y)
\end{pmatrix}
\right\|^3\\
\text{subject to} \quad &~A_k^T d_x +d_y =A_k^T n_x +n_y,\\
  &~ d_y\geq-\tau y_k,
\end{align}
\end{subequations}
where $B_k$ denotes $\nabla^2_{xx} L(x,y,\lambda)$ or its approximation, $\nabla f_k := \nabla f(x_k)$ and $n_k$ is the (approximate) solution to problem \eqref{norm3}.

Set $ d = n_k + t $ with $t = 
\begin{pmatrix}
  t_x \\
  t_y
\end{pmatrix}$, where $t$ is a tangential step to obtain sufficient reduction of the objective function's model. 
Then, problem \eqref{tang3} can be written as follows. 
\begin{subequations}\label{tang4}
\begin{align}
\underset{t}{\text{minimize}} \quad &~\left(\nabla f_k + B_k n_x\right)^T t_x  +\frac{1}{2} t_x^T B_k t_x \nonumber \\
&~- \mu \left(e^T Y^{-1}_k t_y - n_y^T Y_k^{-2} t_y -\frac{1}{2} t_y^T Y^{-2}_k t_y \right)   +\frac{1}{3} \sigma_k \left\|
\begin{pmatrix}
  t_x \\
  Y_k^{-1} t_y
\end{pmatrix}
\right\|^3 \nonumber \\ 
\text{subject to} \quad &~A_k^T t_x +t_y =0,\\
  &~ Y^{-1}_k (n_y+t_y)\geq-\tau e .
\end{align}
\end{subequations}

We now discuss the decrease condition of the tangential step that must be followed for the (approximate) solution of problem \eqref{tang4}. 
We define
\begin{equation}\label{hatt}
  \hat{t} := 
  \begin{pmatrix}
    t_x \\
    \hat{t}_y 
  \end{pmatrix}
  =
  \begin{pmatrix}
    t_x \\
    Y_k^{-1} t_y 
  \end{pmatrix} = D_k t,
\end{equation}
\begin{equation}\label{defBkN1}
  \nabla f_k^N :=
  \begin{pmatrix}
    \nabla f_k + B_k n_x \\
    -\mu ( e - Y_k^{-1} n_y )  
  \end{pmatrix}^T, 
\end{equation}
and
\begin{equation}\label{defBkN2}
  B_k^N :=
  \begin{pmatrix}
    B_k & 0 \\
    0 & \mu I 
  \end{pmatrix}.
\end{equation}
From \eqref{hatt}, \eqref{defBkN1} and \eqref{defBkN2}, problem \eqref{tang4} can be expressed as 
\begin{subequations}\label{tanghatt}
\begin{align}
\underset{t}{\text{minimize}} \quad &~
  \nabla f_k^N \hat{t} + \frac{1}{2} \hat{t}^T   B_k^N \hat{t} +\frac{1}{3} \sigma_k \left\| \hat{t} \right\|^3\\
\text{subject to} \quad &~A_k^T t_x + Y_k \hat{t}_y =0,  \label{eqcon} \\
  &~ \hat{t}_y \geq-\tau e - Y^{-1}_k n_y .
\end{align}
\end{subequations}


From the equality constraints \eqref{eqcon}, $ \hat{t} $ lies in the null space of 
$\begin{pmatrix}
  A_k^T & Y_k
\end{pmatrix}$.
Thus, by defining $\hat{N}_k := D_k N_k$, we can rewrite
\begin{equation}\label{tN}
  \hat{t} = \hat{N}_k p = 
  \begin{pmatrix}
    N_x \\
    \hat{N}_y
  \end{pmatrix} p =
  \begin{pmatrix}
    N_x \\
    Y_k^{-1} N_y 
  \end{pmatrix} 
  p ,
\end{equation}
where $p \in \mathbf{R}^n$, and 
$N_k =
\begin{pmatrix}
    N_x \\
    N_y 
  \end{pmatrix}$ with $N_x \in \mathbf{R}^{n \times n}$ and $N_y \in \mathbf{R}^{m \times n}$.

From \eqref{nullN}, $ D_k N_k $ is a null space basis matrix of 
$\begin{pmatrix}   
A_k^T & Y_k
\end{pmatrix}$ and then
$\begin{pmatrix}   
A_k^T & Y_k
\end{pmatrix} \hat{N}_k = 0$.
Thus, the tangential subproblem (\ref{tanghatt}) is equivalent to the following problem. 
\begin{subequations}\label{tang5}
\begin{align}
\underset{p}{\text{minimize}} \quad &~ 
\nabla f_k^N \hat{N}_k p + \cfrac{1}{2} p^T \hat{N}_k^T B_k^N \hat{N}_k p + \frac{1}{3} \sigma_k \left\| p \right\|^3  \\
\text{subject to} \quad &~  \hat{N}_y p \geq-\tau e -Y^{-1}_k n_y .
\end{align}
\end{subequations}
This has the form of an ARC subproblem for unconstrained optimization with bounds away from zero (in the scaled variables). 

It is necessary for the step $\hat{t}_k$ to produce the comparable reduction in the objective of problem \eqref{tang5} as a steepest descent step. The steepest descent direction for the objective function of problem \eqref{tang5} at $p = 0$ is identified by
\begin{equation}\label{pkc}
  p_k^c 
  = -N_x^T \left(\nabla f_k + B_k n_x\right) + \mu N_y^T \left( Y^{-1}_k e - Y_k^{-2} n_y \right). 
\end{equation}

We define the \emph{tangential predicted reduction} 
as the change in the objective function of problem \eqref{tang4} from $0$ to $t_k$,
\begin{eqnarray}\label{tpred}
 &\text{tpred}_k(t_k) := &
  - \left( \nabla f_k + B_k n_x\right)^T t_x  -\frac{1}{2} t_x^T B_k t_x \nonumber\\
&&+ \mu \left(e^T Y^{-1}_k t_y - n_y^T Y_k^{-2} t_y -\frac{1}{2} t_y^T Y^{-2}_k t_y \right)   - \frac{1}{3} \sigma_k \left\|
\begin{pmatrix}
  t_x \\
  Y_k^{-1} t_y
\end{pmatrix}
\right\|^3
\end{eqnarray}
and we require such a reduction to satisfy the following condition.
\\[0.5pt] 

(c) Tangential Cauchy Decrease Condition. 

The (approximate) solution $t_k$ of the tangential problem \eqref{tang4} must satisfy
\begin{equation}\label{TCauchycon}
  \text{tpred}_k(t_k)\geq  \gamma_t \text{tpred}_k(\beta_k^c N_k p_k^c)
\end{equation}
for some constant $\gamma_t > 0$, where $\beta_k^c$ solves the problem
\begin{subequations}\label{TCauchypoi}
\begin{align}
\underset{\beta \geq 0}{\text{minimize}} \quad &~ -\text{tpred}_k(\beta N_k p_k^c)  \\
\text{subject to} \quad &~ n_y +\beta N_k p_k^c \geq -\tau y_k .
\end{align}
\end{subequations}

Obviously, the optimal solution of problem \eqref{tang5} satisfies the tangential Cauchy decrease condition. It should also be noted that 
\begin{equation}\label{tpredgeq0}
  \text{tpred}_k(t_k)\geq 0
\end{equation}
since $\beta = 0$ is a feasible solution to problem \eqref{TCauchypoi}.

\subsubsection{Merit Function}

The role of the merit function is to determine whether a step should be accepted. We use an exact merit function 
\begin{equation}\label{mer1}
  \phi(z_k;\nu) := \varphi(z_k)+\nu \|c(z_k)\| , \nonumber
\end{equation}
where $\nu > 0$ is penalty parameter. It can also be represented as
\begin{equation}\label{mer2}
  \phi(x_k, y_k; \nu)= f(x_k)- \mu \sum^m_{i=1}\ln y_k^{(i)} + \nu \|g(x_k)+y_k\|. \nonumber
\end{equation}
A model $m_k$ of $\phi(\cdot, \cdot; \nu)$ is constructed around an iterate $ 
\begin{pmatrix}
x_k\\
y_k
\end{pmatrix}$ with
\begin{equation}\label{mmodel}
  \begin{aligned}
    m_k(d_k) := &~ f_k+ \nabla f_k^T d_x + \frac{1}{2}d_x^T B_k d_x + \frac{1}{3}\sigma_k \left\|
\begin{pmatrix}
  d_x \\
  Y_k^{-1} d_y
\end{pmatrix}
\right\|^3 \\
    &~ - \mu \left( \sum^m_{i=1}\ln y_k^{(i)}+ e^T Y_k^{-1} d_y -\frac{1}{2} d_y^T Y_k^{-2} d_y  \right) \\
    &~ + \nu_k\| g_k+ y_k+ A_k^T d_x + d_y \|.
  \end{aligned}\nonumber
\end{equation}
We define the \emph{predicted reduction} in the merit function $\phi$ to represent the change in model $m_k$ generated by step $d_k$
\begin{equation}\label{pred}
 \begin{aligned}
  \text{pred}_k(d_k)&:=m_k(0)-m_k(d_k)\\
  &=- \nabla f_k^T d_x - \frac{1}{2}d_x^T B_k d_x - \frac{1}{3}\sigma_k \left\|
\begin{pmatrix}
  d_x \\
  Y_k^{-1} d_y
\end{pmatrix}
\right\|^3 \\
  &~ + \mu \left( e^T Y_k^{-1} d_y -\frac{1}{2} d_y^T Y_k^{-2} d_y  \right)\\
  &~ +\nu_k \left(\| g_k+ y_k \|- \| g_k+ y_k+ A_k^T d_x + d_y \| \right).
 \end{aligned}
\end{equation}


We define the actual reduction of the merit function $ \phi $ from $
\begin{pmatrix}
  x_k \\
  y_k
\end{pmatrix}$ to $
\begin{pmatrix}
  x_k + d_x \\
  y_k + d_y
\end{pmatrix}$ as
\begin{equation}\label{ared}
  \text{ared}_k (d_k) := \phi(x_k, y_k; \nu_k)-\phi(x_k+d_x, y_k+d_y; \nu_k).
\end{equation}

The predicted reduction is used to determine whether the trial step is accepted and update the regularisation parameter. 
Now we define the descent ratio as 
\begin{equation}\label{rho}
\rho_k:=\frac{\phi(x_k, y_k; \nu_k)-\phi(x_k+d_x, y_k+d_y; \nu_k)}{m_k(0)-m_k(d_k)} =\frac{\text{ared}_k (d_k)}{\text{pred}_k (d_k)}.
\end{equation}
The descent ratio is related to the updating of the regularisation parameters, which is described in detail in our algorithm.

\subsection{Description of the Algorithm}
\noindent Before providing a complete description of the algorithm, let us first describe the update rules for the penalty parameter of the merit function, Lagrange multiplier, and barrier parameter of barrier function.

From $ d_x = n_x + t_x $ and $ d_y = n_y + t_y $, the total predicted reduction (\ref{pred}) can be expressed as follows. 
\begin{equation}\label{repred1}
  \begin{aligned}
  \text{pred}_k (d_k) = 
  &~ - \nabla f_k^T n_x - \frac{1}{2}n_x^T B_k n_x - \left( \nabla f_k +B_k n_x \right)^T t_x  \\
  &~ - \frac{1}{2}t_x^T B_k t_x - \frac{1}{3}\sigma_k \| D_k d_k \|^3 + \mu \left( e^T Y_k^{-1} n_y -\frac{1}{2} n_y^T Y_k^{-2} n_y  \right) \\
  &~ + \mu \left( e^T Y_k^{-1} t_y - n_y^T Y_k^{-2} t_y -\frac{1}{2} t_y^T Y_k^{-2} t_y  \right)\\
  &~ +\nu_k \left(\| g_k+ y_k \| - \| g_k+ y_k+ A_k^T n_x + n_y \|  \right).
  \end{aligned}\nonumber
\end{equation}
Then, using
the definitions (\ref{npred}) and (\ref{tpred}) of the normal and tangential predicted reductions, respectively, we obtain that
\begin{equation}\label{repred2}
  \begin{aligned}
    \text{pred}_k (d_k) &~ = \text{tpred}_k (t_k)   + \nu_k  \text{npred}_k (n_k) + \chi_k \\
   &~ + \cfrac{1}{3} \left( \sigma_k \|D_k t_k\|^3 + \nu_k \tilde{\sigma}_k \|D_k n_k\|^3 - \sigma_k \|D_k d_k\|^3 \right),
  \end{aligned}
\end{equation}
where 
\begin{equation}\label{chi}
\chi_k= - \nabla f_k^T n_x - \frac{1}{2}n_x^T B_k n_x + \mu \left( e^T Y_k^{-1} n_y -\frac{1}{2} n_y^T Y_k^{-2} n_y  \right).
\end{equation}
In order to achieve a satisfactory decrease, we require that
\begin{equation}\label{penapara}
  \text{pred}_k(d_k) \geq \delta \nu_k \text{npred}_k (n_k)
\end{equation}
for some $ \delta \in (0,1) $ and $ \nu_k $ is sufficiently large. 

According to (\ref{repred2}), if $ \text{npred}_k (n_k) = 0 $, then by (\ref{npred}) we obtain $ n_k = 0 $ and nonnegative property of $ \text{tpred}_k (t_k) $, which clearly establishes (\ref{penapara}). 
Conversely, if $ \text{npred}_k (n_k) > 0 $, (\ref{penapara}) holds when
\begin{equation}
  \begin{aligned}
    \nu_k \geq
    -\frac{\text{tpred}_k (t_k) + \frac{1}{3} \sigma_k \left( \|D_k t_k\|^3 - \|D_k d_k\|^3 \right) + \chi_k }{( 1 - \delta ) \; \text{npred}_k (n_k) + \frac{1}{3} \tilde{\sigma}_k \|D_k n_k\|^3 }.
  \end{aligned}\nonumber
\end{equation}
We define
\begin{equation}\label{nu}
\begin{aligned}
  \tilde{\nu}_k := \max & \left\{ -\frac{\text{tpred}_k (t_k) + \frac{1}{3} \sigma_k \left( \|D_k t_k\|^3 - \|D_k d_k\|^3 \right) + \chi_k }{( 1 - \delta ) \; \text{npred}_k (n_k) + \frac{1}{3} \tilde{\sigma}_k \|D_k n_k\|^3 } \right. , \\ 
  &\left. \quad -\frac{ \frac{1}{3} \sigma_k \left( \|D_k t_k\|^3 - \|D_k d_k\|^3 \right) }{\frac{1}{2}\; \text{npred}_k (n_k) + \frac{1}{3} \tilde{\sigma}_k \|D_k n_k\|^3 } \right\}
  \end{aligned}
\end{equation}
as the smallest value to satisfy the \eqref{penapara}, and the purpose of the second term in the right hand in \eqref{nu} is to support the proof of Lemma \ref{nuknubar}. 
Then $ \nu_k $ is updated by
\begin{equation}\label{nugengxin}
  \nu_k =\left\{
  \begin{array}{rcl}
  &\nu_{k-1},                         & {\text{if} \; \tilde{\nu}_k \leq \nu_{k-1}},\\
  &\text{max}\{\tilde{\nu}_k, 1.5\nu_{k-1}\},& {\text{otherwise}.}
\end{array} 
\right.
\end{equation}

The new multiplier vector
\begin{equation}\label{lambda}
\lambda_{k+1} := -\left[  
\begin{pmatrix}
 A(x_k)^T & Y_k
\end{pmatrix}
\begin{pmatrix}
 A(x_k) \\
 Y_k
\end{pmatrix}
  \right]^{-1} 
\begin{pmatrix}
 A(x_k)^T & Y_k
\end{pmatrix} 
\begin{pmatrix}
 \nabla f(x_k) \\
 -\mu e
\end{pmatrix}
\end{equation} 
is least squares multiplier estimate which can also be derived by minimizing the gradient of the Lagrange function, where the gradient consists of \eqref{Lagrad}. The multiplier estimates $\lambda_{k+1}$ obtained in this manner may not always be positive. If $i$-th component of $\lambda_{k+1}$ is not positive, we redefine them as
\begin{equation}\label{lambposi}
  \lambda_{k+1}^{(i)} = \min\left\{ 10^{-3}, \mu_k^{(i)}/y_k^{(i)} \right\}.
\end{equation}


In difficult situations, adaptive strategies for updating barrier parameter have strong robustness. At each iteration, these strategies change $\mu$ based on the progress of the algorithm. Updating barrier parameters using adaptive strategies can take the following form
\begin{equation}\label{muupdate}
  \mu_{k+1} = \vartheta_k \frac{y_k^T \lambda_k}{m}.
\end{equation}
$\vartheta_k$ is chosen as 
\begin{equation}\label{vartheta}
  \vartheta_k = 0.1\min \left\{ 0.05 \frac{1-\varpi_k}{\varpi_k},2 \right\}, \text{ where } \varpi_k = \frac{\min_i y_k^{(i)} \lambda_k^{(i)} }{(y_k)^T \lambda_k /m}.
\end{equation}

Based on the aforementioned discussion, we present a complete description of ARCBIP (see Algorithm \ref{Alg2}).
\begin{algorithm}[H]
\caption{Adaptive cubic regularisation algorithm based on interior-point methods (ARCBIP).}\label{Alg2}
\renewcommand{\algorithmicrequire}{\textbf{Input:}}
\renewcommand{\algorithmicensure}{\textbf{Output:}}
\begin{algorithmic}[1]
\Require Initial iterates $x_0,y_0$, Lagrange multiplier $ \lambda_0 $, 
penalty parameter $\nu_1 > 0$, regularisation parameter $\sigma_0 > 0$, barrier parameter $ \mu_0 > 0$ and stopping tolerance $e_t\geq 0$. The constants satisfy the following condition  
\begin{equation}\label{paracon}
\begin{aligned}
  & 0 < \eta_1 \leq \eta_2 <1, 1 < \gamma_1 < \gamma_2,0 < \gamma_3 < 1, 0 < \xi < 1, \\
  & 0 < \hat{\sigma}_{\min} \leq \sigma_0, 0 < \delta < 1,0 < \tau < 1, 0 < b < 1, a>0.
   \end{aligned}\nonumber
\end{equation} 
Set $ k = 0 $ and $ j = 0 $.
\While{$ E(x_k,y_k; 0) > e_t $}
        \While{$ E(x_k,y_k; \mu_j) > a\mu_j $}
        \State {\bf Step 1.}  Compute the normal step $ n_k =
        \begin{pmatrix}
          n_x \\
          n_y
        \end{pmatrix}$ by solving problem \eqref{normu} and (\ref{u}).
        \State {\bf Step 2.}  Compute the tangential step $ t_k =
        \begin{pmatrix}
         t_x \\
         t_y
        \end{pmatrix} $ by solving problem \eqref{tang5}, \eqref{hatt} and (\ref{tN}), and set $ d_k = n_k + t_k $.
        \State {\bf Step 3.}  Compute $ \tilde{\nu}_k $ by (\ref{nu}) and update the penalty parameter $ \nu_k $ from (\ref{nugengxin}). 
        \State {\bf Step 4.}  Compute the ratio $ \rho_k $ by (\ref{rho}).
                              \If{$ \rho_k \geq \eta_1 $}
                                   \State  $ x_{k+1} = x_k + d_x $, $ y_{k+1} =  y_k + d_y $,
                                   \State compute a new multiplier $ \lambda_{k+1} $ by \eqref{lambda} and \eqref{lambposi}, update $ B_{k+1} $.
                              \Else
                                   \State $ x_{k+1} = x_k $, $ y_{k+1} = y_k $.
                              \EndIf
        \State {\bf Step 5.}  Update the regularisation parameter 
        \begin{equation}\label{sigma}
        \sigma_{k+1}\in
        \begin{cases}
          [\max\{\hat{\sigma}_{\min},\gamma_3 \sigma_k\},\sigma_k] & \text{if $\rho_k\geq\eta_2$, [very successful iteration],}\\
          [\sigma_k,\gamma_1\sigma_k] & \text{if $\rho_k\in[\eta_1,\eta_2)$, [successful iteration]},\\
          [\gamma_1\sigma_k,\gamma_2 \sigma_k] & \text{if $\rho_k<\eta_1$, [unsuccessful iteration]}.
        \end{cases}
        \end{equation}
        \State {\bf Step 6.}  Set $ k = k+1 $;
        \EndWhile
        \State Update the barrier parameter $\mu_j$ by \eqref{muupdate} and \eqref{vartheta}. Let $j = j+1$.
\EndWhile
\Ensure $x_k$.
\end{algorithmic}
\end{algorithm}

\section{Global Convergence of ARCBIP}
\label{sec4}
In this section, we analyze the global behavior of 
ARCBIP when applied to the barrier problem \eqref{barr} for a fixed value of $ \mu $. We present the assumptions regarding the problem and the iterative generation by ARCBIP which are essential for establishing the primary result of this section.
\\[11pt] 
\noindent \(\mathbf{(AS1)}\) \quad The functions $f(x)$ and $g(x)$ are differentiable on an open convex set $\mathcal{X}$ containing all the iterates, and $\nabla f(x) $, $ g(x) $, and $A(x)$ are Lipschitz continuous on $\mathcal{X}$.
\\[9pt] 
\noindent \(\mathbf{(AS2)}\) \quad  The iterate $ {x_k} $ forms an infinite sequence, the sequence $ \{ f_k \} $ is uniformly bounded below, and the sequences $ \{\nabla f_k\} $, $ \{g_k\} $, $ \{A_k\} $ and $ \{B_k\} $ are uniformly bounded.
\\[9pt] 
\noindent \(\mathbf{(AS3)}\) \quad  There exists a positive constant  $\gamma_N$ such that
\begin{equation}\label{N_kbou}
  \| N_k \| \leq \gamma_N \quad \text{and } \quad \sigma_{\min} (N_k) \geq \gamma_N^{-1} \quad \text{for all } k,
\end{equation}
where $N_k$ is a matrix whose columns form an orthogonal basis of the null space of $
\begin{pmatrix}   
A_k^T & I_k
\end{pmatrix}$, $\gamma_N$ is a positive constant, and $ \sigma_{\min} (N_k) $ denotes the smallest singular value of $ N_k $.


 
Some preliminary results are given in Subsection \ref{subs3.1}. The feasibility and optimality are demonstrated in Subsection \ref{subs3.2} and Subsection \ref{subs3.1} respectively. The main convergence results are in Theorem \ref{convergence} and Theorem \ref{overallconvergence}.

\subsection{Preliminary Results}\label{subs3.1}
For ease of explanation, we first define the multiple of the merit function $ \phi $ as 
\begin{equation}
  \tilde{\phi}( x , y ; \nu ) := \cfrac{1}{\nu} \phi( x , y ; \nu )= \cfrac{1}{\nu} \left( f(x)- \mu \sum^m_{i=1}\ln y^{(i)} \right) + \|g(x)+y\| \quad (y>0) . \nonumber
\end{equation}
Since $ \phi $ is reduced sufficiently at every new iterate in the Step 4 of ARCBIP, 
we have that
\begin{equation}
 \tilde{\phi}( x_k , y_k ; \nu_{k-1} ) \leq \tilde{\phi}( x_{k-1} , y_{k-1} ; \nu_{k-1} ) -\cfrac{\eta_1 \text{pred}_{k-1}}{\nu_{k-1}}.  \nonumber
\end{equation}
Therefore,
\begin{equation}\label{tildephi}
  \begin{aligned}
    \tilde{\phi}( x_k , y_k ; \nu_k ) \leq &~ \tilde{\phi}( x_{k-1} , y_{k-1} ; \nu_{k-1} ) \\
    &~ +\left( \frac{1}{\nu_k} - \frac{1}{\nu_{k-1}} \right) \left( f_k - \mu \sum^m_{i=1}\ln y_k^{(i)} \right) -\cfrac{\eta_1 \text{pred}_{k-1}}{\nu_{k-1}}  .
  \end{aligned}
\end{equation}

\begin{lemma}\label{ykbound}
  Suppose that assumption (AS2) is true. Then the sequence $ \{y_k\} $ is bounded, which implies that $ \{\phi( x_k , y_k ; \nu_k )\} $ is bounded below.
\end{lemma}

The proof of the above lemma is given in Appendix \ref{le1}.

Based on the above conclusion, we conclude that $\{D_k\}$ is bounded so that 
\begin{equation}\label{Dup}
  \sigma_{\min} (D_k) \geq \gamma_D^{-1} \quad \text{for all } k,
\end{equation}
where 
$ \sigma_{\min} (D_k) $ denotes the smallest singular value of $ D_k $.

Note that $ \{ y_k \} $ is bounded above and $f_k$ is bounded below. Hence, it is evident that adding a constant to the objective function $f$ makes
\begin{equation}
  f_k- \mu \sum^m_{i=1}\ln y_k^{(i)} \geq 0 
\nonumber
\end{equation}
at all iterations. 
Then, using \eqref{tildephi} and the fact that $ \nu_k $ is nondecreasing imply that
\begin{equation}\label{tildephi2}
 \tilde{\phi}( x_k , y_k ; \nu_k ) \leq \tilde{\phi}( x_{k-1} , y_{k-1} ; \nu_{k-1} ) -\cfrac{\eta_1 \text{pred}_{k-1}}{\nu_{k-1}}  
\end{equation}
for all $k$.

In the following lemma, we provide a stronger bound on the normal predicted reduction $\text{npred}_k (n_k)$ of a solution, satisfying the normal Cauchy decrease condition.

\begin{lemma}\label{npredlow}
  Suppose that $ y_k > 0 $ and that $ n_k =
   \begin{pmatrix}
  n_x \\
  n_y
\end{pmatrix}$ is a solution of problem \eqref{norm3} satisfying the normal Cauchy decrease condition (\ref{npred}). Then
  \begin{equation}\label{npredboun}
  \begin{aligned}
    & \| g_k + y_k \| \mathrm{npred}_k (n_k) \\
    \geq & \frac{\gamma_n}{12} \left\| \binom{A_k}{Y_k} ( g_k + y_k ) \right\|    \min \left\{ \sqrt{\cfrac{ \left\|
     \begin{pmatrix}
       A_k \\
       Y_k 
     \end{pmatrix}
     ( g_k + y_k )\right\|}{\tilde{\sigma}_k \| g_k + y_k \| }} \; , \; \cfrac{ \left\|
     \begin{pmatrix}
       A_k \\
       Y_k 
     \end{pmatrix}
     ( g_k + y_k )\right\|}{\left\|\left(A_k^T \quad Y_k\right)\right\|^2} \; , \; \xi \tau \right\},
    \end{aligned}
  \end{equation}
  where $ \gamma_n $ is used in (\ref{NCauchycon}).
\end{lemma}

The proof of the above lemma is given in Appendix \ref{le2}.

The following result establishes a lower bound on the tangential predicted reduction $\text{tpred}_k(t_k)$ of a solution, satisfying the tangential Cauchy decrease condition.

\begin{lemma}\label{treduplemma}
  Suppose that $ y_k > 0 $ and $t_k = 
  \begin{pmatrix}
  t_x \\
  t_y
\end{pmatrix}$ satisfies the tangential Cauchy decrease condition (\ref{NCauchycon}). Assume that assumptions (AS2)-(AS3) hold. Then for all $ k \geq 0 $,
  \begin{equation}\label{tpredboun}
  \begin{aligned}
     &~ \mathrm{tpred}_k(t_k)
     \geq  \cfrac{\gamma_t}{6\sqrt{2}} \|p_k^c\| \\
     &~ \min \left\{  \cfrac{\|p_k^c\|}{\| N_x^T B_k N_x + \mu N_y^T Y_k^{-2} N_y \|} , \cfrac{1}{2} \sqrt{\cfrac{\|p_k^c\|}{\sigma_k
     \gamma_N^3 \|D_k\|^3}}, \cfrac{(1-\xi)\tau}{\| N_x^T N_x + N_y^T Y_k^{-2} N_y \|^{\frac{1}{2}}}  \right\},
    \end{aligned}
  \end{equation}
  where $ \gamma_N $ and $ \gamma_t $ are constants and $ p_k^c $ is defined by (\ref{pkc}).
\end{lemma}
The proof of the above lemma is given in Appendix \ref{le3}.

The following two lemmas give useful bounds for the normal step $ n_k $ and the tangential step $ t_k $.
\begin{lemma}\label{nkuplemma}
  Suppose that the step $ n_k $ satisfies (\ref{npred}) and assumption (AS2) holds. Then  
  \begin{equation}\label{nkup}
    \|n_k\| \leq \frac{\sqrt{3} }{\|D_k\|}
    \sqrt{ \cfrac{\left\| 
    \begin{pmatrix}
      A_k^T & Y_k 
    \end{pmatrix}
    \right\|}{\tilde{\sigma}_k} }  .
  \end{equation}
\end{lemma}
The proof of the above lemma is given in Appendix \ref{le4}.

\begin{lemma}\label{tkuplemma}
Suppose that assumptions (AS2)-(AS3) hold.  Then the tangential step 
\begin{equation}\label{tkbou}
  \| t_k \| \leq \cfrac{3 \gamma_N^{\frac{5}{2}} }{\sigma_k} \max \left\{ \sqrt{ \frac{\sigma_k \|p_k^c\|}{\|D_k\|^3}} , \cfrac{3 \gamma_N^{\frac{3}{2}} \gamma_W}{4\|D_k\|^3} \right\},
\end{equation}
where $\gamma_W$ is a constant. 
\end{lemma}
The proof of the above lemma is given in Appendix \ref{le5}.

Next, we provide the boundary for the step $d_k$.
\begin{lemma}\label{dkuplemma}
Suppose that assumptions (AS2)-(AS3) hold. 
Then for all $k>0$
\begin{equation}\label{dkup}
      \| d_k \| \leq\cfrac{\gamma_d^{\prime \prime}}{\sqrt{ \sigma_k }\|D_k\|} = \cfrac{\gamma_d}{\sqrt{\tilde{\sigma}_k}\|D_k\|},  
\end{equation}
where $\gamma_d^{\prime \prime}$ is a positive constant and $ \gamma_d = \gamma_d^{\prime \prime} / \xi^{\frac{3}{2}} $.
\end{lemma}
The proof of the above lemma is given in Appendix \ref{le6}.

In the next lemma, we prove that $m_k$ is an exact local model of the merit function $\phi$.

\begin{lemma}\label{p-al}
Suppose that assumptions (AS1)-(AS2) hold.
Then there exists a positive constant $ \gamma_L $ such that for any iterate $
\begin{pmatrix}
  x_k \\
  y_k
\end{pmatrix}$ and any step $
\begin{pmatrix}
  d_x \\
  d_y
\end{pmatrix}$ such that the segment $[x_k, x_k + d_x]$ is in $\mathcal{X}$ and $d_y \geq -\tau y_k$,
\begin{equation}\label{p-a}
  | \mathrm{pred}_k (d) - \mathrm{ared}_k (d) | \leq \gamma_L \left[ ( 1+ \nu_k ) \| d_x \|^2 + \| Y_k^{-1} d_y \|^2 \right].
\end{equation}
\end{lemma}
The proof of the above lemma is given in Appendix \ref{le7}.


In the next lemma, we prove that ARCBIP 
determines an acceptable step with $ \sigma_k $, i.e., that there cannot be an infinite loop between Step 1 and Step 6 of ARCBIP. For this, we must ensure that, by increasing $ \sigma_k $, we are able to make the displacement in $ y $ arbitrarily small.
\begin{lemma}\label{feasib}
  Suppose that assumptions (AS1)-(AS3) hold. Suppose that $ y_k > 0 $ and that $ (x_k,y_k) $ is not a stationary point of the barrier problem \eqref{barr}. Then there exists $\sigma_k^0 > 0$, such that if $ \sigma_k \geq \sigma_k^0 $, the inequality $ \rho_k \geq \eta_2 $ holds.
\end{lemma}
The proof of the above lemma is given in Appendix \ref{le8}.

The rule in Step 4 of ARCBIP is to determine the new slack variable. We demonstrate that the new slack variable ensures that the steps between two successive iterations are still governed by $ \sigma_k $. 
\begin{lemma}\label{lemma9}
 Supposing that the assumptions (AS2)-(AS3) hold. Then, for all $k >0$
   \begin{equation}\label{sucitera}
    \left\|
    \begin{pmatrix}
      x_{k+1} \\
      y_{k+1}
    \end{pmatrix}-
    \begin{pmatrix}
      x_k \\
      y_k
    \end{pmatrix}
\right\| \leq \cfrac{\gamma_{xy}}{\sqrt{\sigma_k}}, 
   \end{equation}
 where $ \gamma_{xy} $ is a positive constant. 
\end{lemma}

The proof of the above lemma is given in Appendix \ref{le9}.

\subsection{Feasibility of ARCBIP}\label{subs3.2}

The index set of successful iterations of the algorithm is recorded as 
\begin{equation}\label{S}
  \mathcal{S} := \{ k \geq 0 : k \; \text{successful or very successful in the sense of (\ref{sigma})} \} . \nonumber
\end{equation}

In the subsequent results, we initiate the discussion on convergence. It deals with function $ 
    \begin{pmatrix}
      x \\
      y
    \end{pmatrix} \in \mathbf{R}^n \times \mathbf{R}^m_+ \mapsto \| g(x) + y \|^2 $, which is another measure of infeasibility for the original problem \eqref{barr}. Notably, if the slack variable is scaled by matrix $ Y_k^{-1} $, the gradient of this function with respect to the scaled variable becomes $2 \binom{A_k}{Y_k} ( g_k + y_k ). $
    
We now show that, for the infeasibility function $ \| g(x) + y \|^2 $, the iterative process generated by the algorithm approaches stationarity.
\begin{theorem}\label{AYgy0}
  Suppose that assumptions (AS1)-(AS3) hold.
  Then
  \begin{equation}
    \lim_{k \rightarrow \infty} \binom{A_k}{Y_k} ( g_k + y_k ) = 0. \nonumber
  \end{equation}
\end{theorem}
\begin{proof}
  By the assumptions on $A(x)$ and $g(x)$, and since Lemma \ref{ykbound} implies that $\{y_k\}$ is contained in a bounded open set $\mathcal{Y}$, we have that the function
  \begin{equation}
    \kappa (x,y) := \left\| \binom{A(x)}{Y} ( g(x) + y ) \right\| \nonumber
  \end{equation}
  is Lipschitz continuous on the open set $ \mathcal{X} \times \mathcal{Y} $ containing all the iterates $
    \begin{pmatrix}
      x_k \\
      y_k
    \end{pmatrix}
$, i.e., there is a constant $ \gamma_L^{\prime} > 0 $ such that
  \begin{equation}\label{kaplip}
    \left| \kappa (x,y) - \kappa (x_j,y_j)\right| \leq \gamma_L^{\prime} \left\|
    \begin{pmatrix}
      x \\
      y
    \end{pmatrix}-
    \begin{pmatrix}
      x_j \\
      y_j
    \end{pmatrix}
\right\|
  \end{equation}
  for any two points  
  $\begin{pmatrix}
      x \\
      y
    \end{pmatrix} $ and 
     $\begin{pmatrix}
      x_j \\
      y_j
    \end{pmatrix}$
     in $ \mathcal{X} \times \mathcal{Y} $. 
     
     Now consider an arbitrary iterate 
     $\begin{pmatrix}
      x_j \\
      y_j
    \end{pmatrix}$ such that $ \kappa_j := \kappa (x_j,y_j) \neq 0 .$
  We aim to demonstrate that in a neighborhood of this iterate, all sufficiently small steps are accepted by ARCBIP.
  Thus, 
  we define an open ball
  \begin{equation}\label{ball}
    \mathcal{O}_j := \left\{ \begin{pmatrix}
      x \\
      y
    \end{pmatrix} : \left\|
    \begin{pmatrix}
      x \\
      y
    \end{pmatrix}-
    \begin{pmatrix}
      x_j \\
      y_j
    \end{pmatrix}
\right\| < \kappa_j / (2 \gamma_L^{\prime} ) \right\} . \nonumber
  \end{equation}
  Then by \eqref{kaplip}, for any $\begin{pmatrix}
      x \\
      y
    \end{pmatrix} \in \mathcal{O}_j $, we have that $    \kappa (x,y)     \geq \cfrac{1}{2} \kappa_j,$  which implies that $ g(x) + y \neq 0 $. 
  
  Using \eqref{penapara}, \eqref{npredboun} and the boundedness assumptions on $\{A_k\}$ and $\{g_k + y_k\}$,
   we see that there exists a constant $ \gamma_1^{\prime} $ (independent of $k$ and $j$) such that for any such iteration $\begin{pmatrix}
      x_j \\
      y_j
    \end{pmatrix}$ and any iteration $ \begin{pmatrix}
      x_k \\
      y_k
    \end{pmatrix} \in \mathcal{O}_j $
  \begin{equation}\label{preddeta}
    \text{pred}_k(d_k) \geq \delta \nu_k \text{npred}_k (n_k) \geq \nu_k \gamma_1^{\prime} \kappa_j \min \left\{ \sqrt{\frac{\kappa_j}{2 \tilde{\sigma}_k }} , \kappa_j, \xi\tau \right\}.
  \end{equation}
  If $ \tilde{\sigma}_k $ is sufficiently large, it can be inferred that
  \begin{equation}
   \text{pred}_k(d_k) \geq \nu_k \gamma_1^{\prime} \kappa_j \sqrt{\frac{\kappa_j}{2 \tilde{\sigma}_k }}. \nonumber
  \end{equation}
  Then using (\ref{p-a}) and the fact that $ \widetilde{\sigma}_k=\cfrac{\sigma_k}{\xi^3} $, we can obtain  that
  \begin{equation}
  \begin{aligned}
     &~ \cfrac{| \text{ared}_k (d_k) - \text{pred}_k (d_k) |}{ \text{pred}_k (d_k)} \\
     \leq &~ \cfrac{ \sqrt{2} \gamma_L \left[ ( 1+ \nu_k ) \| d_x \|^2 + \| Y_k^{-1} d_y \|^2 \right] \sqrt{\widetilde{\sigma}_k} }{ \nu_k \gamma_1^{\prime} \kappa_j^{\frac{3}{2}} } 
     \leq  \cfrac{ \sqrt{2} \gamma_L  ( 1+ \nu_k ) \| D_k d_k\|^2  \sqrt{\widetilde{\sigma}_k} }{ \nu_k \gamma_1^{\prime} \kappa_j^{\frac{3}{2}} } \\
     \leq &~ \cfrac{ \sqrt{2} \gamma_L \gamma_d^2  ( 1+ \nu_k ) \sqrt{\widetilde{\sigma}_k} }{ \nu_k \gamma_1^{\prime} \kappa_j^{\frac{3}{2}} \widetilde{\sigma}_k}   =\cfrac{ \sqrt{2} \xi^{\frac{3}{2}} \gamma_L \gamma_d^2 ( 1+ \nu_k ) }{ \nu_k \gamma_1^{\prime} \kappa_j^{\frac{3}{2}} \sqrt{\sigma_k}}. \nonumber
    \end{aligned}
  \end{equation}
  By making $ \sigma_k $ sufficiently large, we can ensure that the last term is less than or equal to $ 1 - \eta_2 $, and therefore a very successful iteration will occur from any $ 
  \begin{pmatrix}
  x_k \\
  y_k
\end{pmatrix} \in \mathcal{O}_j $.
  
  Next, we demonstrate that the remaining portion of the iteration $ \{ x_k \}_{k>j} $ cannot remain within $ \mathcal{O}_j $. We establish this by contradiction, assuming that for all $ k>j $, $ x_k \in \mathcal{O}_j $. Consequently, $ \rho_k \geq \eta_1 $ holds for sufficiently large $ \sigma_k $; this implies the existence of $ \sigma_0 > 0 $ such that $ \sigma_k \leq \sigma_0 $ for all successful iteration $ k \in \mathcal{S} $ and $ k>j $. Then from (\ref{tildephi2}) and (\ref{preddeta}), it can be deduced that
  \begin{equation}
    \tilde{\phi}_{k+1} \leq \tilde{\phi}_k - \cfrac{\eta_1}{\nu_k} \text{pred}_k (d_k) \leq \tilde{\phi}_k - \eta_1 \gamma_1^{\prime} \kappa_j \min\left\{ \sqrt{\frac{\xi^3 \kappa_j}{2 \sigma_0 }} , \kappa_j, \xi\tau \right\}, \nonumber
  \end{equation}
  where $ \tilde{\phi}_k := \tilde{\phi}( x_k , y_k ; \nu_k ) $.
  The last term on the right side of the above formula is a constant, so $ \tilde{\phi}_k $ tends towards negative infinity.
  This contradicts 
  conclusion of Lemma \ref{ykbound}. Consequently, the iterative sequence must exit $ \mathcal{O}_j $ when $ k>j $.
  
  Let $
  \begin{pmatrix}
  x_{k+1} \\
  y_{k+1}
\end{pmatrix}$ denote the first iteration after $\begin{pmatrix}
  x_j \\
  y_j
\end{pmatrix} $ that is not included in $ \mathcal{O}_j $. 
  If there exists $ i \in [ j , k ] $ such that $ \sqrt{\cfrac{\kappa_j}{2 \tilde{\sigma}_k }} > \min \left\{ \kappa_j, \xi\tau \right\} $, then it can be inferred from (\ref{tildephi2}) and (\ref{preddeta}) that
  \begin{equation}\label{tildephi11}
    \begin{aligned}
      \tilde{\phi}_{k+1} &~ \leq \tilde{\phi}_{i+1}  \leq \tilde{\phi}_i - \cfrac{\eta_1}{\nu_i} \text{pred}_i (d_i)  \leq \tilde{\phi}_i - \eta_1 \gamma_1^{\prime} \kappa_j \min \left\{ \kappa_j, \xi\tau \right\} \\
      &~ \leq \tilde{\phi}_j - \eta_1 \gamma_1^{\prime} \kappa_j \min \left\{ \kappa_j, \xi\tau \right\}.
    \end{aligned}
  \end{equation}
  When $ \sqrt{\cfrac{\kappa_j}{2 \tilde{\sigma}_k }} \leq \min \left\{ \kappa_j, \xi\tau \right\}$ holds for all $i \in [j, k]$,  it can be inferred from (\ref{tildephi2}) and (\ref{preddeta}) that
  \begin{equation}\label{tildephi12}
    \begin{aligned}
      \tilde{\phi}_{k+1} &~ \leq \tilde{\phi}_j - \sum_{i=j}^k \cfrac{\eta_1}{\nu_i} \text{pred}_i (d_i)  \leq \tilde{\phi}_j - \cfrac{\eta_1 \gamma_1^{\prime} \kappa_j^{ \frac{3}{2} } \xi^{ \frac{3}{2} }}{\sqrt{2}} \sum_{i=j}^k \frac{1}{\sqrt{\sigma_i}},
    \end{aligned}
  \end{equation}
  where the first inequality is obtained by summing the first $j$ of these successful iterations.
  
  Subsequently, based on the \eqref{sucitera} and the displacement of $ 
   \begin{pmatrix}
      x_{k+1} \\
      y_{k+1}
    \end{pmatrix}
    $ from the ball $ \mathcal{O}_j $, which has a radius of $ \kappa_j / (2 \gamma_L^{\prime} ) $, it is derived that
  \begin{equation}
    \sum_{i=j}^k \cfrac{1}{\sqrt{\sigma_i}} \geq \cfrac{1}{\gamma_{xy}} \left\| 
    \begin{pmatrix}
      x_{k+1} \\
      y_{k+1}
    \end{pmatrix}-
    \begin{pmatrix}
      x_j \\
      y_j
    \end{pmatrix}
     \right\| \geq \cfrac{\kappa_j}{2 \gamma_{xy} \gamma_L^{\prime}}.\nonumber
  \end{equation}
  Substituting in (\ref{tildephi12}) yields
  \begin{equation}\label{tildephi13}
    \tilde{\phi}_{k+1} \leq \tilde{\phi}_j - \frac{\eta_1 \gamma_1^{\prime} \kappa_j^{ \frac{5}{2} } \xi^{ \frac{3}{2} }}{ 2\sqrt{2} \gamma_{xy} \gamma_L^{\prime} }.
  \end{equation}
  
 Since $ \{ \tilde{\phi}_k \} $ is decreasing and bounded below, we can see that $ \tilde{\phi}_j \rightarrow \tilde{\phi}_{\ast} $ for some minimum value of $ \tilde{\phi}_{\ast} $. Due to the arbitrariness of $ j $, it follows that either (\ref{tildephi11}) or (\ref{tildephi13}) must remain constant at $
  \begin{pmatrix}
  x_j \\
  y_j
\end{pmatrix}$, which implies $ \kappa_j \rightarrow 0 $.
\end{proof}

The result of Theorem 1 indicates that $ A_k ( g_k + y_k) \rightarrow 0 $, and $ Y_k ( g_k + y_k) \rightarrow 0 $. Of course, when $ g_k + y_k \rightarrow 0 $, the conclusion clearly holds. 

\subsection{Global Analysis of Barrier Problem in ARCBIP}\label{subs3.3}
We now analyze the global behavior of ARCBIP when applied to the barrier problem \eqref{barr} for a fixed value of $ \mu $.

In this section, we only focus on the case where the matrix $ \begin{pmatrix}   
A_k^T & Y_k
\end{pmatrix} $ has full rank, i.e. $ \sigma_{\min} (\begin{pmatrix}   
A_k^T & Y_k
\end{pmatrix}) \geq \gamma_{\rm{inf}} > 0 $, which implies that $ g_k + y_k \rightarrow 0 $. Initially, we utilize this condition to constrain the relationship between the normal step $ n_k $
and $ \text{npred}_k $. 

\begin{lemma}\label{nandnpred}
  Suppose that assumptions (AS1)-(AS3) hold and that for some $ \gamma_{\rm{inf}} > 0 $,
  \begin{equation}\label{sigmamin}
    \sigma_{\min} (\begin{pmatrix}   
A_k^T & Y_k
\end{pmatrix}) \geq \gamma_{\rm{inf}} > 0
  \end{equation}
  for all $k$. Then, there exist positive constants $ \gamma_{Dn}^{\prime}$ and $ \gamma_{gy}^{\prime}$ such that if $ \| g_k + y_k \| \leq \gamma_{gy}^{\prime} $
  \begin{equation}\label{nnpred}
    \left\|
    \begin{pmatrix}
      n_x \\
      Y_k^{-1} n_y
      \end{pmatrix}
\right\| \leq \gamma_{Dn}^{\prime \prime} \mathrm{npred}_k (n_k).
  \end{equation}
\end{lemma}
The proof of the above lemma is given in Appendix \ref{le10}.

Subsequently, we can utilize relationship of Lemma \ref{nandnpred} to establish that the sequence of penalty parameters $ \nu_k $ is bounded. For the upper bound of the normal step (\ref{nnpred}), under the condition that $  g_k + y_k \rightarrow 0 $, we can demonstrate that the parameter $ \nu_k $ ultimately remains fixed.

\begin{lemma}\label{nuknubar}
  Suppose that assumptions (AS1)-(AS3) are satisfied, and that (\ref{nnpred}) holds when $k$ is sufficiently large. Then the sequence of penalty parameters $ \{ \nu_k \} $ is bounded. Furthermore, there exists an index $ k_1 $ and positive scalars $ \bar{\nu} $ and $\gamma_{pt}^{\prime}$ such that for all $ k \geq k_1 $, $\nu_k= \bar{\nu}$
  and 
  \begin{equation}\label{predgapttpred}
    \mathrm{pred}_k (d_k) \geq \gamma_{pt}^{\prime} \mathrm{tpred}_k (t_k).
  \end{equation}
\end{lemma}
The proof of the above lemma is given in Appendix \ref{le11}.

\begin{lemma}\label{yawfr0}
Suppose that assumptions (AS1)-(AS3) hold. Then $ \{ y_k \} $ is bounded away from zero and $ g_k $ is negative for all large $ k $.
\end{lemma}
The proof of the above lemma is given in Appendix \ref{le12}.

We present the main convergence result for the barrier problem.

\begin{theorem}\label{convergence}
  Suppose that assumptions (AS1)-(AS3) are valid and the singular values of the matrix $ 
  \begin{pmatrix}
      A_k^T & Y_k
\end{pmatrix}$ are bounded away from zero. Then
\begin{equation}
 \nabla f_k + A_k \mu Y_k^{-1} e \rightarrow 0 .\nonumber
 \end{equation}
\end{theorem}

\begin{proof}

  From Lemma \ref{treduplemma}, $t_k$ satisfies \eqref{tpredboun}, where the null space basis matrix $ N_k $
   is assumed to possesses singular values, which are bounded above and bounded away from zero. Since Lemma \ref{yawfr0} and assumption (AS2),
  inequality (\ref{tpredboun}) can be expressed as
  \begin{equation}\label{tpredboun22}
    \text{tpred}_k  (t_k) \geq \gamma_1^{\prime} \| p_k^c \| \min \left\{ 1, \sqrt{\cfrac{ \| p_k^c \| }{ \sigma_k }} , \| p_k^c \| \right\}
  \end{equation}
  with some positive constant $ \gamma_1^{\prime} $. 
  
  To demonstrate that $ \nabla f_k + \mu A_k Y_k^{-1} e \rightarrow 0 $, we associate this quantity with $ p_k^c $. Noting that the matrix $ \left( I \; -A_kY_k^{-1} \right)^T $ also represents a null space basis of the equality constraints \eqref{eqcon}, i.e. 
  \begin{equation}
  \begin{pmatrix}   
    A_k^T & Y_k
  \end{pmatrix} 
  \left( I \; -A_kY_k^{-1} \right)^T = 0 . \nonumber
\end{equation}
Utilizing the equivalence of the null space basis, we obtain that
  \begin{equation}\label{hk}
    \begin{aligned}
      h_k &:= \nabla f_k + \mu A_k Y_k^{-1} e = \left( I \; -A_kY_k^{-1} \right) \binom{\nabla f(x_k)}{-\mu e} \\
      &~ = \left( I \; -A_kY_k^{-1} \right)
\hat{N}_k ( \hat{N}_k^T \hat{N}_k )^{-1} \hat{N}_k^T  \binom{\nabla f(x_k)}{-\mu  e} 
    \end{aligned}
  \end{equation}
  for the selected null space basis $ \hat{N}_k $. Based on the boundedness of $ A_k $ and of the singular values of $ N_k $, it can be inferred from (\ref{hk}) that $ \{ h_k \} $ is bounded by a constant multiple of $ \| N_x^T \nabla f(x_k) - \mu N_y^T Y_k^{-1} e \| $. Therefore, from (\ref{pkc}), it exist some positive constants $ \gamma_2^{\prime} $ and $ \gamma_3^{\prime} $, for all $k$
  \begin{equation}
    \| p_k^c \| \geq \gamma_2^{\prime} \| h_k \| - \gamma_3^{\prime} \| n_k \|. \nonumber
  \end{equation}
  
  We provide a contradiction to the hypothesis that $ \varsigma = \frac{1}{4} \limsup_{k \rightarrow \infty} \| h_k \| $ is nonzero. Due to $ n_k \rightarrow 0 $, there exists an iteration $ \begin{pmatrix}
      x_j \\
      y_j
    \end{pmatrix} $ with arbitrarily large $ j \in \mathcal{S} $ such that $ \| h_j \| > 3 \varsigma $, and $ \gamma_3^{\prime} \| n_k \| < \gamma_2^{\prime} \varsigma $ for all $ k \geq j $. 
    
For convinience, we define $ \bar{\gamma}_L $ as the Lipschitz constant for $ h(x,y) = \nabla f(x) - \mu A(x) Y^{-1} e $. Then $h_k = h(x_k, y_k)=\nabla f_k + \mu A_k Y_k^{-1} e$. And within the ball 
    \begin{equation}
  \mathcal{O}_j = \left\{ \begin{pmatrix}
      x \\
      y
    \end{pmatrix} : \left\|
    \begin{pmatrix}
      x \\
      y
    \end{pmatrix}-
    \begin{pmatrix}
      x_j \\
      y_j
    \end{pmatrix}
\right\|
  < \varsigma / \bar{\gamma}_L \right\}, \nonumber
  \end{equation}
  any iteration $ \begin{pmatrix}
      x_k \\
      y_k
    \end{pmatrix} $ with $ \{ k \in \mathcal{S} : k \geq j \} $ satisfies
  \begin{equation}
    \| p_k^c \| \geq \gamma_2^{\prime} ( \| h_j \| - \| h_j - h_k \| ) - \gamma_3^{\prime} \| n_k \| \geq \gamma^{\prime}_2 ( 3 \varsigma - \varsigma - \varsigma ) = \gamma^{\prime}_2 \varsigma. \nonumber
  \end{equation}
  By Lemma \ref{nuknubar} and (\ref{tpredboun22}), we have that
  \begin{eqnarray}\label{predlow111}
      \text{pred}_k (d_k) & \geq & \gamma^{\prime}_{pt} \text{tpred}_k (t_k) \geq \gamma^{\prime}_{pt} \gamma^{\prime}_1 \| p_k^c \| \min \left\{ 1, \sqrt{\cfrac{ \| p_k^c \| }{ \sigma_k }}, \| p_k^c \| \right\} \nonumber\\
      &\geq& \gamma^{\prime}_{pt} \gamma^{\prime}_1 \gamma^{\prime}_2 \varsigma \min \left\{ 1, \sqrt{\cfrac{ \gamma^{\prime}_2 \varsigma }{ \sigma_k }} , \gamma^{\prime}_2 \varsigma \right\}.
  \end{eqnarray}
  Since $ \{ \phi_k \} $ is bounded below, $ \text{pred}_k \rightarrow 0 $. 
  Therefore, from \eqref{predlow111}, $ \cfrac{ 1}{ \sqrt{\sigma_k} } \rightarrow 0 $ for the subsequence of $ k $.
  Consequently, we can take $ j $ sufficiently large such that $ \sqrt{\cfrac{ \gamma^{\prime}_2 \varsigma }{\sigma_k }} \leq \min \left\{ 1 , \gamma^{\prime}_2 \varsigma \right\} $ for any $ k \in \{ k \in \mathcal{S} : k \geq j \} $ with $  
   \begin{pmatrix}
      x_k \\
      y_k
    \end{pmatrix}
     \in \mathcal{O}_j $. Then 
  \begin{equation}\label{pred112}
    \text{pred}_k (d_k) \geq \cfrac{ \gamma^{\prime}_{pt} \gamma^{\prime}_1 (\gamma^{\prime}_2)^{\frac{3}{2}} \varsigma^{\frac{3}{2}}}{\sqrt{ \sigma_k }}.
  \end{equation}
  Now by Lemma \ref{p-al} and (\ref{dkup}), if $ \begin{pmatrix}
      x_k \\
      y_k
    \end{pmatrix} \in \mathcal{O}_j, $
  \begin{equation}\label{a_pbp}
    \cfrac{| \text{ared}_k (d_k) - \text{pred}_k (d_k) |}{ \text{pred}_k (d_k)} \leq \cfrac{ 4 \gamma_L ( 1+ \bar{\nu} ) (\gamma^{\prime \prime}_n)^2 }{ \gamma^{\prime}_{pt} \gamma^{\prime}_1 (\gamma^{\prime}_2)^{\frac{3}{2}} \varsigma^{\frac{3}{2}} \sqrt{\sigma_k} } \leq 1- \eta_2 \nonumber
  \end{equation}
  for $ \cfrac{1 }{\sqrt{ \sigma_k }} $ sufficiently small, which yields the acceptance of the step. This implies that if $ \begin{pmatrix}
      x_k \\
      y_k
    \end{pmatrix} \in \mathcal{O}_j $ for all $ k \in \{ k \in \mathcal{S} : k \geq j \}  $, $ \sigma_k  $ would eventually stop decreasing. This contradicts our proof that $ \cfrac{ 1}{ \sqrt{\sigma_k} } \rightarrow 0 $. Consequently, the sequence must ultimately leave $ \mathcal{O}_j $ for some $ k > j $. 
  
  In that case, let $ \begin{pmatrix}
      x_{k+1} \\
      y_{k+1}
    \end{pmatrix}$ denote the first iteration following $ \begin{pmatrix}
      x_j \\
      y_j
    \end{pmatrix}$ that is not included in $ \mathcal{O}_j $. As can be inferred $ k+1 \in \mathcal{S} $ and from (\ref{pred112}),
  \begin{eqnarray}\label{phi113}
      \phi_{k+1} & \leq& \phi_j - \eta_1 \sum_{i=j}^k \text{pred}_i (d_i) \nonumber\\
      &  \leq& \phi_j - \eta_1 \gamma^{\prime}_{pt} \gamma^{\prime}_1 (\gamma^{\prime}_2)^{\frac{3}{2}} \varsigma^{\frac{3}{2}} \sum_{i=j}^k \cfrac{1 }{\sqrt{ \sigma_i }} \leq \phi_j - \cfrac{\eta_1 \gamma^{\prime}_{pt} \gamma^{\prime}_1 (\gamma^{\prime}_2)^{\frac{3}{2}} \varsigma^{\frac{5}{2}}}{\gamma_{xy} \bar{\gamma}_L}.
  \end{eqnarray}
  The last inequality follows from the fact that $ \begin{pmatrix}
      x_{k+1} \\
      y_{k+1}
    \end{pmatrix}$ has left the ball $ \mathcal{O}_j $ with a radius of $ \varsigma / \bar{\gamma}_L $. Therefore, at the conclusion of Theorem \ref{AYgy0}, 
  \begin{equation}
     \sum_{i=j}^k \cfrac{1}{\sqrt{\sigma_i}}  \geq \cfrac{\varsigma}{\gamma_{xy} \bar{\gamma}_L} .\nonumber
  \end{equation}
  
  Given that the sequence $ \{ \phi_k \} $ is decreasing and bounded below, it is convergent. This is in contrast to the fact that $ j $ in (\ref{phi113}) can take on arbitrarily large values, and the fact that $ \varsigma \neq 0 $. Then $ h_k $ approaches zero. The theorem is proved.
\end{proof}

\begin{remark}\label{remarkinfes}
~\\
We define that the sequence $ \{x_k\} $ is asymptotically feasible if $ g(x_k)^+ \rightarrow 0 $.
 
We call that the sequence 
$ \left\{ \begin{pmatrix}
      g_k \\
      A_k
    \end{pmatrix}\right\} $ has a limit point
     $ \begin{pmatrix}
      \bar{g}\\
      \bar{A}
    \end{pmatrix}$ failing the linear independence constraint qualification (LICQ) if the set $ \{ \bar{A}^{(i)} : \bar{g}^{(i)} = 0 \} $ is rank deficient.
 
The result of Theorem \ref{AYgy0} indicates that $ A_k ( g_k + y_k) \rightarrow 0 $ and $ Y_k ( g_k + y_k) \rightarrow 0 $. These results can yield two cases, the first case being $ g_k + y_k \rightarrow 0 $ and the second being $ g_k + y_k \not\rightarrow 0 $. 



From assumption (AS1)-(AS3), Theorem \ref{AYgy0} and Lemma \ref{ykbound}, we can derive the following result ({\romannumeral1}).

({\romannumeral1}) When the sequence $\{x_k\}$ is not asymptotically feasible, the iteration tends towards stationarity of the measure of infeasibility $ x \mapsto \| g(x)^+ \| $, that is, $ A_k g_k^+ \rightarrow 0 $, and the penalty parameter $ \nu_k\to \infty$.

When $ g_k + y_k \not\rightarrow 0 $ and the matrices $ A_k $ and $ Y_k $ approach rank deficiency, such the following case may also occur. 
From assumptions (AS2) and Lemma \ref{ykbound}, we have that the following result ({\romannumeral2}) holds.

({\romannumeral2}) When the sequence $\{x_k\}$ is asymptotically feasible and one limit point of sequence $ \left\{ \begin{pmatrix}
      g_k \\
      A_k
    \end{pmatrix}\right\} $ fails LICQ, the penalty parameter $ \nu_k\to \infty$.
    
    Of course, in the case that $\{x_k\}$ is asymptotically feasible ($ g_k + y_k \rightarrow 0 $) and all limit points of the sequence
$ \left\{ \begin{pmatrix}
      g_k \\
      A_k
    \end{pmatrix}\right\} $, satisfy the LICQ, Theorem \ref{convergence} can be proven through Lemma \ref{nandnpred}-\ref{yawfr0}.

\end{remark}

\subsection{Convergence Results of ARCBIP}

Now, we give the convergence results of the overall algorithm, where ARCBIP is operated to reduce the value of the barrier parameter $ \mu $ and stopping tolerance $e_t$ is set to 0. 

All iterations generated by the algorithm form a single sequence $ \left\{ \begin{pmatrix}
      x_k\\
      y_k
    \end{pmatrix}\right\}_{k \geq 0} $. The index $ k_{j-1} (j \geq 1) $ labels the starting point of the $j$-th external iteration, and the end point of the iteration is $ 
    \begin{pmatrix}
      x_{k_j} \\
      y_{k_j}
    \end{pmatrix}$.

\begin{theorem}\label{overallconvergence}
  Suppose that $ \left\{ \begin{pmatrix}
      x_k\\
      y_k
    \end{pmatrix} \right\} $ is generated by ARCBIP, and for each barrier problem, assumptions (AS1)-(AS3) hold. Then, one of the following two results may occur.\\
  (\uppercase\expandafter{\romannumeral1}) For a parameter $ \mu_j $, either inequality
   \begin{equation}\label{3.1}
     \| g_k + y_k \| \leq a\mu_j
   \end{equation}
  is never satisfied, in which case the stationarity condition for minimizing $ x \mapsto \| g(x)^+ \| $ is satisfied in the limit, i.e., $ A(x_k) g(x_k)^+ \rightarrow 0 $, or inequality
   \begin{equation}\label{3.2}
     \| \nabla f_k + A_k \lambda_k \| \leq a\mu_j
   \end{equation}
   is never satisfied, in which case the sequence $ \left\{ \begin{pmatrix}
      g_k \\
      A_k
    \end{pmatrix}\right\} $ has a limit point
     $ \begin{pmatrix}
      \bar{g}\\
      \bar{A}
    \end{pmatrix}$ failing LICQ.\\
   (\uppercase\expandafter{\romannumeral2}) In each external iteration $ j $ of ARCBIP, the internal algorithm successfully finds a pair $ 
    \begin{pmatrix}
      x_{k_j} \\
      y_{k_j}
    \end{pmatrix}$ that satisfies (\ref{3.1})-(\ref{3.2}). All limit points $\hat{x}$ of $ \{x_{k_j}\} $ are feasible. Furthermore, if any limit point $\hat{x}$ of $ \{ x_{k_j} \} $ satisfies LICQ, then the first-order optimality condition for the problem \eqref{ori} hold at $ \hat{x} $: there exists $ \hat{\lambda} \in \mathbf{R}^m $ such that
  \begin{equation}
   \nabla \hat{f} + \hat{A} \hat{\lambda} = 0,\; \hat{g} \leq 0,\; \hat{\lambda} \geq 0,\; \hat{g}^T \hat{\lambda} = 0. \nonumber
  \end{equation}
\end{theorem}
\begin{proof}
    Assume that ARCBIP cannot find a point that satisfies (\ref{3.1}) and (\ref{3.2}) for a certain value of $ \mu_j $. This implies that ARCBIP generates an infinite sequence with $\mu = \mu_j $ for the barrier problem \eqref{barr}. Then the result of Theorem \ref{convergence} does not occur. Since assumptions (AS1)-(AS3) hold, these imply that for a value of $\mu$, either result (i) or (ii) of Remark \ref{remarkinfes} occurs. We obtain the conclusion (I).
    
    The other possibility is that ARCBIP satisfies (\ref{3.1}) - (\ref{3.2}) for all $ j \geq 1 $. Let $ \mathcal{J} $ be a subsequence of index $ j $, such that $ x_{k_j} \rightarrow \hat{x} $ when $ j \rightarrow \infty $ in $ \mathcal{J} $. Since $ 0 \leq g_{k_j}^+ \leq g_{k_j} + y_{k_j} $ and $ g_{k_j} + y_{k_j} \rightarrow 0 $, then $ \hat{g} = g(\hat{x}) \leq 0 $ ($ \hat{x} $ is feasible) and  $ y_{k_j} \rightarrow \hat{y} = - \hat{g} $ when $ j\rightarrow \infty $ in $ \mathcal{J} $. 
    Suppose that LICQ holds at $ \hat{x} $ and consider a set of indices $ \mathcal{L} = \{ l : \hat{g}^{(l)} = 0 \} $. For $ l \notin \mathcal{L} $, $ \hat{g}^{(l)} < 0 $ and $ \hat{y}^{(l)} > 0 $, so that $ \lambda_{k_j}^{(l)} \rightarrow 0 $ when $ j \rightarrow \infty $ in $\mathcal{J}$. From this and $ \nabla f_{k_j} + A_{k_j} \lambda_{k_j} \rightarrow 0 $, we can deduce that
    \begin{equation}\label{5.3}
      \nabla f_{k_j} + \sum_{l \in \mathcal{L} } \lambda_{k_j}^{(l)} \nabla g_{k_j}^{(l)} \rightarrow 0.
    \end{equation}
    According to the LICQ, the vector $ \{ \nabla \hat{g}^{(l)} : l \in \mathcal{L} \} $ is linearly independent. By (\ref{5.3}), the positive sequence $ \{ \lambda_{k_j} \}_{j \in \mathcal{J} } $ converges to some value $ \hat{\lambda} \geq 0 $. Now, when $ j \rightarrow \infty $ in $\mathcal{J}$, 
    we have that $\nabla \hat{f} + \hat{A} \hat{\lambda} = 0$
    and $ \hat{g}^T \hat{\lambda} =0 $. Therefore, conclusion (II) holds.
  \end{proof}

\section{\bf Numerical Results}
\label{sec5}
In this section, we present the numerical results of ARCBIP which have been performed on laptops with AMD Ryzen 7 5800H with Radeon Graphics 3.20 GHz. Numerical testing is implemented under MATLAB R2020b.


We use the standard initial point $x_0$ for all test problems. The initial parameters are selected as follows: $\nu = 1$, $\mu = 1$, $\xi = 0.8$, $\eta_1 = 10^{-8}$, $\eta_2 = 0.9$, $\delta = 0.0001$, $\tau = 0.995$, $b = 0.001$, $\sigma_0 = 1$. The parameter $a$ is relatively sensitive to our algorithm, so its selected range is relatively large, specifically an integer of $[0, 10]\cup10\ast[0, 10]$.
$y_0$ is an all-ones vector. Then $Y_0$ is the identity matrix and initial Lagrange multiplier is set as
\begin{equation}\label{lambda0}
\lambda_0 := -\left[  
\begin{pmatrix}
 A(x_0)^T & Y_0
\end{pmatrix}
\begin{pmatrix}
 A(x_0) \\
 Y_0
\end{pmatrix}
  \right]^{-1} 
\begin{pmatrix}
 A(x_0)^T & Y_0
\end{pmatrix} 
\begin{pmatrix}
 \nabla f(x_0) \\
 -\mu e
\end{pmatrix}. \nonumber
\end{equation}

The regularisation parameter is updated by the following rules
\begin{equation}\label{sigma2}
\sigma_{k+1} = 
\begin{cases*}
\max\{\hat{\sigma}_{\min}, \sigma_k /20\}
& $\rho_k \geq \eta_2$, \\
\sigma_k 
& $\rho_k \in [\eta_1, \eta_2)$, \\
2 \sigma_k 
& $\rho_k < \eta_1$,
\end{cases*}\nonumber
\end{equation}
where $\hat{\sigma}_{\min} = 10^{-16}$.

We solve ARC subproblems \eqref{normu} and \eqref{tang5} dropping constraints using the algorithm proposed by Lieder \cite{Lieder2020} (the methods in \cite{Cartis2011A,Jiang2021,Jia2022} are also excellent choices), where the ARC subproblem is reduced to a generalized eigenvalue problem which can be solved efficiently due to existing highly advanced eigenvalue solvers. And we use the backtracking strategy \cite{Nocedal1999} to handle the constraints of the subproblems for obtaining $n_k$ and $t_k$.

We have tested our algorithm on test problems from \cite{HStests,Liu2020}, whose characteristics are described in Table \ref{maininfor}. For each problem, we give the number of variables (denoted as $\textit{n}$) and the total number of constraints (denoted as $\textit{m}$), including equalities and general inequalities (but not bounds on the variables). We also describe what conditions (free, bounds) were applied to the variables. And we embody the types of constraints (equalities, inequalities, linear, nonlinear) and the characteristics of the objective function. For example, the objective of problem HS15 is nonlinear function, the constraints are nonlinear inequalities, some variables are free, and some variables contain bounds. 
These test problems cover most of the numerical difficulties encountered in practical applications, including badly scaled objective and constraint functions, badly scaled variables, ill-conditioned optimization problems, non-regular solutions at points where the constraint qualification is not satisfied, distinct local solutions, and infinitely many solutions. Therefore, we employ these test problems to assess the feasibility and robustness of Algorithm \ref{Alg2} (ARCBIP).

We first compared ARCBIP with fmincon in MATLAB with the same termination tolerance $$Res:=E(x_k,y_k;0) \leq e_t=1e-8$$ on the whole test problems set of \cite{HStests}. The fmincon algorithm \cite{MATLAB2020b} that comes with MATLAB is flexible and comprehensive. It can handle linear, nonlinear, equality and inequality constraints, and adapt to optimization problems of different scales and complexities. Moreover, its' built-in functions have undergone extensive testing and verification. They show good stability and reliability in various complex mathematical models and practical application scenarios. 
For comparison, fmincon is set to use interior-point trust-region based algorithms, MaxIterations=10000, and MaxFunctionEvaluations=50000. 
The comparison results are reported in Table \ref{numercal table}, where the columns $\textit{n}$ and $\textit{m}$ are the numbers of variables and constraints, respectively. $NO$ and $NI$ are the numbers of computation of the outer iteration and the inner iteration, respectively. $NIF$ and $NIG$ are the numbers of computation of the function and gradient evaluations, respectively. “-” denotes that the algorithm failed to solve the corresponding problem. For these test problems, ARCBIP failed on 11 cases with a success rate of 90.4\%, while fmincon failed on 13 cases with a success rate of 88.6\%. ARCBIP yields a slightly higher success rate. On the other hand, for the 93 test problems where both algorithms achieved successful runs, the total number and average number of iterations are summarized in the last two rows of Table \ref{numercal table}. The results indicate that fmincon has a slight advantage. Fig \ref{Figure1} provides a visual illustration of the numerical performance of ARCBIP and fmincon for these 93 test problems based on the logarithmic performance profiles \cite{Dolan2001} (in the performance profile, the higher the position of the curve, the greater the probability of solving the problem under the corresponding performance ratio). 
Based on the results in Table \ref{numercal table} and the performance profile in Fig \ref{Figure1}, ARCBIP is not strictly competitive with fmincon overall; however, for the 93 test problems successfully solved by both algorithms, the performance profile in Fig \ref{Figure1} (using a logarithmic scale for the performance ratio) shows that, for the NIF metric, ARCBIP is within a factor of 2 of fmincon for approximately 70\% of these problems. Similar to the NIF metric, the performance profiles for NI and NIG confirm that ARCBIP remains within a factor of 2 of fmincon on roughly 70\% of the common successful runs, demonstrating its robustness across different performance measures. This indicates that ARCBIP achieves comparable efficiency on a substantial portion of the test problems set of \cite{HStests}.

Then, we compared Algorithm \ref{Alg2} (ARCBIP) with Algorithm 3.6 in \cite{Liu2020}, IPOPT \cite{WäChter2006} (Version 3.0.0) and fmincon with the same termination tolerance $ e_t = 10^{-8} $ on all the test problems from \cite{Liu2020}. 
Algorithm 3.6 \cite{Liu2020} is a globally convergent primal-dual interior point relaxation method with nice convergence results and good numerical performance. It does not require any primal or dual iterates to be interior-points, which has similarity to some warmstarting interior-point methods and is different from most of the globally convergent interior-point methods in the literatures.
IPOPT \cite{WäChter2006} is based on the interior-point methods. It performs very well in solving various nonlinear programming problems. Due to its efficiency and open-source nature, IPOPT is often used as a benchmark solver for nonlinear optimization problems in both scientific research and industry.
We use all the test problems provided by \cite{Liu2020} for comparative analysis. The comparison results of 42 test problems are presented in Table \ref{table4}. “-” denotes that the algorithm failed to solve the corresponding problem or the number of iterations exceeded $2000$. 
For these test problems, both ARCBIP and Algorithm 3.6 in \cite{Liu2020} achieved successful solutions for all cases, while IPOPT and fmincon each failed on two cases. For the computational results of the 38 test problems solved successfully by all four algorithms, the total and average iteration counts are summarized in the last two rows of Table \ref{table4}. 
For ARCBIP, the total number of iterations for $NI$, $NIF$, and $NIG$ are $551$, $589$, and $578$, respectively, and their average number of iterations are $15$, $16$, and $16$, respectively. These core metric values of ARCBIP are significantly lower than those of other algorithms. 
By using the logarithmic performance profiles \cite{Dolan2001}, we display the performance based on the numerical results in Table \ref{table4} visually (see Fig \ref{Figure2}  ).
These results indicate that ARCBIP outperforms the other three algorithms in terms of numerical performance on these test problems.
\begin{center}
{
\setlength{\tabcolsep}{1pt}{
\begin{longtable}{lcccccc}
\caption{Test problems}\label{maininfor}
  \endfirsthead
  \multicolumn{7}{l}{Table \ref{maininfor} continued}\\
  \hline
\multirow{2}{*}{Problem}  & \multicolumn{2}{l}{Dimension}&\multirow{2}{*}{ }  &\multirow{2}{*}{Variable types}    &\multirow{2}{*}{Constraint types} &\multirow{2}{*}{Objective} \\
                          \cmidrule{2-3}
                          &$n$&$m$                       & &  &                                 \\
  \hline
  \endhead
  \hline
\multirow{2}{*}{Problem}  & \multicolumn{2}{l}{Dimension}&\multirow{2}{*}{ }    &\multirow{2}{*}{Variable types}    &\multirow{2}{*}{Constraint types}  &\multirow{2}{*}{Objective} \\
                          \cmidrule{2-3}  

                          &$n$&$m$                       & &     &                           \\
\hline
  CB2        &3   &3  &&free  &nonlin ineq &linear   \\ 
  CB3        &3   &3  &&free  &nonlin ineq &linear   \\
  CHACONN1   &3   &3  &&free  &nonlin ineq &linear \\ 
  CHACONN2   &3   &3  &&free  &nonlin ineq &linear  \\ 
  CONGIGMZ   &3   &5  &&free  &linear ineq, nonlin ineq &linear \\ 
  DEMYMALO   &3   &3  &&free  &linear ineq, nonlin ineq &linear   \\ 
  DIPIGRI   &7   &4   &&free  &nonlin ineq &nonlin  \\ 
  EXPFITA   &5   &22  &&free  &linear ineq &nonlin   \\ 
  GIGOMEZ1   &3   &3  &&free  &linear ineq, nonlin ineq &linear   \\ 
  GIGOMEZ2   &3   &3  &&free  &nonlin ineq &nonlin  \\ 
  GIGOMEZ3   &3   &3  &&free  &nonlin ineq &nonlin\\ 
  GOFFIN   &51   &50  &&free  &linear ineq &nonlin   \\ 
  HAIFAS   &13   &9   &&free  &nonlin ineq &linear  \\ 
  HALDMADS &6   &42   &&free  &nonlin ineq &linear  \\
  HS01    &2   &1     &&free, bounded &  &nonlin \\
  HS02    &2   &1     &&free, bounded &  &nonlin \\ 
  HS03    &2   &1     &&free, bounded &  &nonlin  \\
  HS04    &2   &2     &&bounded  &  &nonlin  \\
  HS05    &2   &4   &&bounded  &  &nonlin   \\
  HS06    &2   &2   &&free     &nonlin eq &nonlin   \\
  HS07    &2   &2   &&free     &nonlin eq &nonlin   \\
  HS08    &2   &4   &&free     &nonlin eq &nonlin \\ 
  HS09    &2   &2   &&free     &linear eq &linear  \\
  HS10    &2   &1   &&free     &nonlin ineq &nonlin \\ 
  HS11    &2   &1   &&free     &nonlin ineq &nonlin  \\ 
  HS12    &2   &1   &&free     &nonlin ineq &nonlin   \\ 
  HS13    &2   &3   &&bounded  &nonlin ineq &nonlin\\
  HS14    &2   &3   &&free     &linear eq, nonlin ineq &nonlin\\
  HS15    &2   &3   &&free, bounded &nonlin ineq &nonlin  \\
  HS16    &2   &4   &&bounded &nonlin ineq &nonlin   \\
  HS17    &2   &5   &&bounded &nonlin ineq &nonlin   \\
  HS18    &2   &6   &&bounded &nonlin ineq &nonlin   \\
  HS19    &2   &6   &&bounded &nonlin ineq &nonlin   \\
  HS20    &2   &5   &&free, bounded &nonlin ineq &nonlin  \\
  HS21    &2   &5   &&bounded &linear ineq  &nonlin  \\
  HS22    &2   &2   &&free  &linear ineq, nonlin ineq &nonlin   \\
  HS23    &2   &9   &&bounded &linear ineq, nonlin ineq &nonlin \\
  HS24    &2   &5   &&bounded & linear ineq  &nonlin \\
  HS25    &3   &6   &&bounded &   &nonlin  \\
  HS26    &3   &2   &&free &nonlin eq &nonlin \\
  HS27    &3   &2   &&free &nonlin eq  &nonlin\\ 
  HS28    &3   &3   &&free &linear eq  &nonlin   \\ 
  HS29    &3   &1   &&free &nonlin eq  &nonlin \\
  HS30    &3   &7   &&bounded &nonlin ineq &nonlin    \\   
  HS31    &3   &7   &&bounded &nonlin ineq &nonlin   \\   
  HS32    &3   &6   &&bounded &linear eq, nonlin ineq &nonlin   \\   
  HS33    &3   &6   &&bounded &nonlin ineq &nonlin   \\   
  HS34    &3   &8   &&bounded &nonlin ineq &linear   \\   
  HS35    &3   &4   &&bounded &linear ineq &nonlin   \\
  HS36    &3   &7   &&bounded &linear ineq  &nonlin   \\
  HS37    &3   &8   &&bounded &linear ineq  &nonlin   \\
  HS38    &4   &8   &&bounded &  &nonlin   \\
  HS39    &4   &4   &&free &nonlin eq  &linear   \\
  HS40    &4   &10   &&free &nonlin eq  &nonlin   \\
  HS41    &3   &8   &&bounded &linear eq  &nonlin   \\
  HS42    &4   &4   &&free &linear eq, nonlin eq  &nonlin   \\
  HS43    &4   &3   &&free &nonlin ineq &nonlin  \\
  HS44    &4   &10  &&bounded &linear ineq  &nonlin   \\
  HS45    &5   &10  &&bounded &  &nonlin    \\
  HS46    &3   &8   &&free &nonlin eq  &nonlin   \\
  HS47    &5   &6   &&free &nonlin eq &nonlin  \\
  HS48    &5   &4   &&free &linear eq  &nonlin  \\
  HS49    &5   &4   &&free &linear eq  &nonlin  \\
  HS50    &5   &6   &&free &linear eq  &nonlin   \\
  HS51    &5   &6   &&free &linear eq  &nonlin   \\
  HS52    &5   &6   &&free &linear eq  &nonlin   \\
  HS53    &5   &16  &&bounded &linear eq  &nonlin   \\
  HS54    &6   &14  &&bounded &linear eq  &nonlin   \\
  HS55    &6   &20  &&bounded &linear eq  &nonlin   \\
  HS56    &7   &8   &&free &nonlin eq  &nonlin   \\
  HS57    &2   &3   &&bounded &nonlin ineq &nonlin  \\
  HS59    &2   &7   &&bounded &nonlin ineq &nonlin  \\
  HS60    &3   &20  &&bounded &nonlin eq  &nonlin   \\
  HS61    &3   &4   &&free &nonlin eq  &nonlin   \\
  HS62    &3   &8   &&bounded &linear eq  &nonlin   \\
  HS63    &3   &10  &&bounded &linear eq, nonlin eq  &nonlin   \\
  HS64    &3   &4   &&bounded &nonlin ineq &nonlin   \\
  HS65    &3   &7   &&bounded &nonlin ineq  &nonlin   \\
  HS66    &3   &8   &&bounded &nonlin ineq  &linear   \\
  HS68,69    &4   &12  &&bounded &nonlin eq  &nonlin   \\
  HS70    &4   &9   &&bounded &nonlin ineq  &nonlin   \\
  HS71    &4   &11  &&bounded &nonlin eq, nonlin ineq  &nonlin   \\
  HS72    &4   &10  &&bounded &nonlin ineq  &linear   \\
  HS73    &4   &8   &&bounded &linear eq, linear ineq, nonlin ineq &linear \\
  HS74,75    &4   &16   &&bounded &nonlin eq, linear ineq &nonlin \\
  HS76    &4   &7   &&bounded &linear ineq  &nonlin    \\
  HS77    &6   &4   &&free &nonlin eq &nonlin \\
  HS78    &5   &6   &&free &nonlin eq &nonlin \\
  HS79    &5   &6   &&free &nonlin eq &nonlin \\
  HS80    &5   &16   &&bounded &nonlin eq &nonlin \\
  HS81    &5   &16   &&bounded &nonlin eq &nonlin \\
  HS83    &5   &16   &&bounded &nonlin ineq &nonlin \\
  HS84    &5   &16   &&bounded &nonlin ineq &nonlin \\
  HS85    &5   &48   &&bounded &linear ineq, nonlin ineq &nonlin \\
  HS86    &5   &15  &&bounded &linear ineq  &nonlin  \\
  HS87    &6   &20  &&bounded &nonlin eq  &nonlin    \\
  HS88    &2   &1   &&free &nonlin ineq &nonlin \\
  HS89    &3   &1   &&free &nonlin ineq &nonlin \\
  HS90    &4   &1   &&free &nonlin ineq &nonlin \\
  HS91    &5   &1   &&free &nonlin ineq &nonlin \\
  HS92    &6   &1   &&free &nonlin ineq &nonlin \\
  HS93    &6   &8   &&bounded & nonlin ineq &nonlin \\
  HS95-98    &6   &20   &&bounded &nonlin ineq &linear \\
  HS99    &7   &18   &&bounded &nonlin eq &nonlin \\
  HS100   &7   &4  &&free &nonlin ineq  &nonlin \\
  HS100MOD   &7  &4  &&free &nonlin ineq  &nonlin \\
  HS101-103    &7   &20   &&bounded &nonlin ineq &nonlin \\
  HS104    &7   &22   &&bounded &nonlin ineq &nonlin \\
  HS105    &8   &1  &&bounded &linear ineq &nonlin  \\
  HS106    &8   &22   &&bounded &linear ineq, nonlin ineq &linear \\
  HS107    &8   &20   &&free, bounded &nonlin eq &nonlin \\
  HS108    &9   &14   &&free, bounded & nonlin ineq &nonlin \\
  HS109    &9   &34   &&bounded &nonlin eq, linear ineq, nonlin ineq &nonlin \\
  HS110   &10  &20  &&bounded &  &nonlin  \\
  HS111    &10   &20   &&bounded &nonlin eq  &nonlin \\
  HS112    &10   &16  &&bounded &linear eq  &nonlin  \\
  HS113   &10  &8  &&free &linear ineq, nonlin ineq &nonlin   \\
  HS114   &10  &11  &&bounded &linear (in)eq, nonlin (in)eq &nonlin   \\
  HS116   &13  &15  &&free &linear ineq, nonlin ineq &linear   \\
  HS117   &15  &20 &&bounded &nonlin ineq &nonlin  \\
  HS118   &15  &42  &&bounded &linear ineq &nonlin   \\
  HS119   &16  &32  &&bounded &linear eq &nonlin   \\
  KIWCRESC   &3   &2  &&free &nonlin ineq &linear \\ 
  MADSEN  &3   &6  &&free &nonlin ineq &linear  \\ 
  MAKELA1 &3   &2  &&free &linear ineq, nonlin ineq &linear \\
  MAKELA2 &3   &4  &&free &nonlin ineq &linear \\
  MAKELA3 &21   &20  &&free &nonlin ineq &linear \\
  MAKELA4 &21   &40  &&free &linear ineq &linear  \\  
  MIFFLIN1 &3   &2   &&free &linear ineq, nonlin ineq &linear \\
  MIFFLIN2 &3   &2   &&free &nonlin ineq &linear   \\  
  PENTAGON &6   &15  &&free &linear ineq &nonlin \\
  POLAK1   &3   &2   &&free &nonlin ineq &linear \\ 
  POLAK3   &12  &10  &&free &nonlin ineq &linear  \\
  POLAK5   &3   &2   &&free &nonlin ineq &linear  \\ 
  ROSENMMX &5   &4   &&free &nonlin ineq &linear \\  
  S218  &2   &1  &&free &nonlin ineq &linear   \\
  S221  &2   &3  &&free &nonlin ineq &linear  \\
  S222  &2   &3  &&free &nonlin ineq &linear \\
  S225  &2   &5  &&free &linear ineq, nonlin ineq  &nonlin \\
  S227  &2   &2  &&free &nonlin ineq &nonlin \\
  S228  &2   &2  &&free &linear ineq, nonlin ineq  &nonlin\\
  S230  &2   &2  &&free &nonlin ineq  &linear \\
  S231  &2   &2  &&free &linear ineq &nonlin\\
  S233  &2   &1  &&free &nonlin ineq &nonlin \\
  S264  &4   &3  &&free &nonlin ineq &nonlin \\
  S268  &5   &6  &&free &linear ineq  &nonlin\\
  SPIRAL &3  &2  &&free &nonlin ineq &linear   \\  
  TFI1     &3   &101  &&free &nonlin ineq &nonlin   \\
  TFI3     &3   &101  &&free &linear ineq &nonlin \\
  WOMFLET &3   &3  &&free &nonlin ineq &linear   \\
  \hline
\end{longtable}
}}
\end{center}

\begin{center}
{
\setlength{\tabcolsep}{3pt}
\begin{longtable}{lcccccccccccc}
\caption{Comparison results I}\label{numercal table}
  \endfirsthead
  \multicolumn{13}{l}{Table \ref{numercal table} continued}\\
  \hline
  \multirow{2}{*}{Problem}  &&\multicolumn{5}{c}{ARCBIP}&&\multicolumn{5}{c}{fmincon} \\ 
           \cmidrule{2-7} 
		   \cmidrule{9-13}
&&$NO$  &$NI$ &$NIF$ &$NIG$ &$Res$ &&$NO$ &$NI$ &$NIF$ &$NIG$ &$Res$\\
  \hline
  \endhead
  \hline
\multirow{2}{*}{Problem} &&\multicolumn{5}{c}{ARCBIP}&&\multicolumn{5}{c}{fmincon} \\  
           \cmidrule{2-7} 
		   \cmidrule{9-13}
&&$NO$  &$NI$ &$NIF$ &$NIG$ &$Res$ &&$NO$ &$NI$ &$NIF$ &$NIG$ &$Res$\\
\hline
  HS01       &&13 &11 &12 &11 &4.10e$-$09   &&39 &46 &47 &40 &3.20e$-$10\\
  HS02       &&14 &9 &10 &10 &8.93e$-$09    &&15 &20 &21 &16 &4.40e$-$08 \\
  HS03       &&15 &10 &11 &11 &9.68e$-$10   &&15 &16 &17 &16 &4.00e$-$09\\
  HS04       &&4 &14 &15 &14 &1.52e$-$09    &&- &- &- &- &-\\
  HS05       &&4 &7 &8 &8 &4.89e$-$12       &&11 &13 &14 &12 &2.00e$-$10\\
  HS06       &&14 &15 &16 &14 &7.29e$-$10   &&9 &16 &17 &10 &4.42e$-$11\\
  HS07       &&12 &13 &14 &14 &6.74e$-$09   &&11 &19 &20 &12 &8.12e$-$12\\
  HS08       &&14 &25 &26 &25 &9.67e$-$10   &&7 &9 &10 &8 &3.89e$-$14\\
  HS09       &&4 &36 &37 &37 &2.53e$-$10    &&6 &6 &7 &7 &8.76e$-$10\\
  HS10       &&2  &10 &11 &11 &7.41e$-$09   &&16 &16 &17 &17 &2.00e$-$10\\
  HS11       &&8  &8  &9  &9  &6.69e$-$10   &&10 &10 &11 &11 &8.71e$-$10\\
  HS12       &&13 &9 &10 &10 &3.27e$-$09    &&9 &12 &13 &10 &2.50e$-$09\\
  HS13       &&13 &14 &15 &15 &2.57e$-$09   &&- &- &- &- &-\\
  HS14       &&13 &13 &14 &14 &8.95e$-$09   &&8 &9 &10 &9 &4.60e$-$09\\
  HS15       &&5 &22 &23 &23 &1.00e$-$10    &&11 &12 &13 &12 &7.53e$-$10\\
  HS16       &&13 &74 &75 &35 &4.11e$-$09   &&19 &22 &23 &20 &8.00e$-$10\\
  HS17       &&13 &35 &36 &26 &4.10e$-$09   &&13 &23 &24 &14 &8.65e$-$09\\
  HS18       &&14 &19 &20 &20 &4.96e$-$09   &&14 &14 &15 &15 &6.32e$-$11\\
  HS19       &&4 &45 &46 &46 &5.41e$-$09    &&18 &19 &20 &19 &6.13e$-$11\\
  HS20       &&4 &33 &34 &34 &5.03e$-$09    &&- &- &- &- &-\\
  HS21       &&13 &10 &11 &10 &8.84e$-$09   &&12 &15 &16 &13 &2.00e$-$10\\
  HS22       &&2 &5 &6 &6 &6.68e$-$09       &&10 &11 &12 &11 &1.00e$-$10\\
  HS23       &&13 &17 &18 &18 &4.10e$-$09   &&14 &14 &15 &15 &2.00e$-$09\\
  HS24       &&14 &10 &11 &11 &3.13e$-$09   &&15 &16 &17 &16 &2.00e$-$08 \\
  HS25       &&4 &67 &68 &22 &3.16e$-$09    &&53 &61 &62 &54 &2.34e$-$09\\
  HS26       &&4 &35 &36 &36 &3.70e$-$09    &&27 &136 &137 &128 &9.64e$-$09\\
  HS27       &&16 &55 &56 &55 &9.30e$-$09   &&28 &70 &71 &29 &6.26e$-$10\\
  HS28       &&14 &54 &55 &55 &1.37e$-$09   &&9 &10 &11 &10 &8.17e$-$09\\
  HS29       &&2 &8 &9 &9 &6.76e$-$09       &&14 &17 &18 &15 &3.54e$-$10\\
  HS30       &&12 &10 &11 &11 &7.07e$-$09   &&7 &7 &8 &8 &9.23e$-$10\\
  HS31       &&13 &15 &16 &14 &7.99e$-$09   &&9 &13 &14 &10 &1.92e$-$09\\
  HS32       &&13 &6 &7 &7 &8.23e$-$11      &&15 &16 &17 &16 &9.12e$-$09\\ 
  HS33       &&14 &61 &62 &41 &5.94e$-$09   &&- &- &- &- &-\\
  HS34       &&3 &18 &19 &16 &8.86e$-09$    &&17 &24 &25 &18 &8.00e$-$-10\\
  HS35       &&4 &10 &11 &11 &1.62e$-$09    &&12 &12 &13 &13 &4.00e$-$09\\
  HS36       &&4 &14 &15 &15 &2.76e$-$09    &&8 &8 &9 &9 &6.05e$-$10\\
  HS37       &&4 &11 &12 &12 &3.83e$-$10    &&7 &7 &8 &8 &4.61e$-$10\\
  HS38       &&3 &31 &32 &32 &2.35e$-$13    &&34 &40 &41 &35 &2.22e$-$09\\
  HS39       &&2 &19 &20 &20 &6.06e$-09$    &&13 &13 &14 &14 &4.28e$-$10\\
  HS40       &&3 &16 &17 &17 &5.47e$-09$    &&7 &7 &8 &8 &3.37e$-$11\\
  HS41       &&2 &21 &22 &21 &5.84e$-10$    &&16 &16 &17 &17 &4.01e$-$09\\
  HS42       &&3 &17 &18 &17 &4.28e$-$09    &&9 &11 &12 &10 &5.25e$-$10\\
  HS43       &&3 &11 &12 &12 &9.99e$-$09    &&13 &15 &16 &14 &5.25e$-$09 \\
  HS44       &&4 &17 &18 &18 &2.82e$-$09    &&- &- &- &- &-\\
  HS45       &&4 &19 &20 &18 &8.56e$-$09    &&20 &26 &27 &21 &2.00e$-$08\\ 
  HS47       &&4 &96 &97 &97 &8.50e$-$09    &&30 &36 &37 &31 &2.86e$-$09\\
  HS48       &&4 &60 &61 &61 &4.69e$-$09    &&10 &19 &20 &11 &1.62e$-$09\\ 
  HS49       &&4 &64 &65 &65 &7.95e$-$09    &&29 &32 &33 &30 &9.61e$-$09\\
  HS50       &&2 &16 &17 &14 &7.40e$-09$    &&14 &19 &20 &15 &7.68e$-$10\\
  HS51       &&2 &19 &20 &19 &7.39e$-$09    &&6 &8 &9 &7 &2.07e$-$10\\
  HS52       &&4 &18 &19 &19 &4.72e$-$09    &&9 &13 &14 &10 &7.64e$-$09\\
  HS53       &&2 &25 &26 &24 &3.69e$-$09    &&9 &10 &11 &10 &9.70e$-$10\\
  HS54       &&- &- &- &- &-    &&73 &88 &89 &74 &3.20e$-$09\\
  HS55       &&4 &61 &62 &62 &8.03e$-$09    &&13 &14 &15 &14 &5.00e$-$09\\
  HS56       &&2 &16 &17 &17 &9.80e$-$09    &&12 &16 &17 &13 &1.45e$-$10\\
  HS57       &&4 &17 &18 &18 &3.47e$-$10    &&24 &25 &26 &24 &1.54e$-$08 \\
  HS59       &&4 &45 &46 &44 &6.73e$-$09    &&20 &39 &40 &20 &4.09e$-$08 \\ 
  HS60       &&1 &14 &15 &15 &5.23e$-$09    &&11 &15 &16 &12 &2.36e$-$09\\
  HS61       &&4 &16 &17 &17 &4.08e$-$09    &&53 &378 &379 &54 &1.50e$-$09\\
  HS62       &&2 &32 &33 &24 &6.80e$-$09    &&11 &32 &33 &12 &2.32e$-$09\\
  HS63       &&3 &22 &23 &23 &7.62e$-$09    &&8 &8 &9 &9 &1.85e$-$09\\
  HS64       &&13 &21 &22 &22 &4.10e$-$09   &&53 &130 &131 &54 &2.48e$-$09\\
  HS65       &&3 &16 &17 &17 &9.89e$-$09    &&12 &13 &14 &13 &2.00e$-$10\\
  HS66       &&3 &13 &14 &13 &5.77e$-$09    &&15 &15 &16 &16 &4.00e$-$09\\
  HS68       &&3 &176 &177 &118 &9.02e$-$09    &&47 &144 &145 &48 &5.85e$-$09\\
  HS69       &&1 &32 &33 &24 &5.26e$-$09    &&20 &30 &31 &21 &8.99e$-$09\\
  HS70       &&3 &27 &28 &20 &7.02e$-$09    &&38 &40 &41 &39 &6.40e$-$10\\
  HS71       &&2 &88 &89 &76 &7.92e$-$09    &&10 &11 &12 &11 &1.37e$-$09\\
  HS72       &&- &- &- &- &-    &&33 &39 &40 &34 &6.54e$-$11\\
  HS73       &&4 &23 &24 &23 &2.94e$-$09    &&11 &11 &12 &12 &9.21e$-$10\\
  HS74       &&- &- &- &- &-    &&13 &20 &21 &14 &3.13e$-$11\\
  HS75       &&- &- &- &- &-    &&10 &10 &11 &11 &4.16e$-$09\\
  HS76       &&4 &17 &18 &17 &8.73e$-$09    &&15 &30 &31 &16 &7.63e$-$09\\
  HS77       &&2 &25 &26 &24 &9.80e$-$09    &&15 &17 &18 &16 &2.08e$-$09\\
  HS78       &&2 &20 &21 &21 &9.28e$-$09    &&9 &10 &11 &10 &2.76e$-$13\\
  HS79       &&2 &18 &19 &19 &9.79e$-$09    &&14 &14 &15 &15 &1.09e$-$10\\
  HS80       &&3 &26 &27 &27 &4.80e$-$09    &&- &- &- &- &-\\
  HS81       &&2 &21 &22 &21 &6.08e$-$09    &&21 &23 &24 &22 &1.60e$-$10\\
  HS83       &&9 &25 &26 &26 &5.44e$-$09    &&13 &13 &14 &14 &4.92e$-$10\\
  HS84       &&- &- &- &- &-    &&24 &27 &28 &25 &6.04e$-$09\\
  HS85       &&- &- &- &- &-    &&- &- &- &- &-\\
  HS86       &&4 &20 &21 &21 &9.90e$-$09    &&15 &17 &18 &16 &2.87e$-$10\\
  HS87       &&4 &37 &38 &33 &7.20e$-$09    &&- &- &- &- &-\\
  HS88       &&2 &66 &67 &42 &2.68e$-$09    &&29 &42 &43 &30 &3.88e$-$11\\
  HS89       &&2 &19 &20 &20 &3.44e$-$09    &&35 &57 &58 &36 &5.54e$-$09\\
  HS90       &&2  &52 &53 &23 &5.42e$-$10    &&35 &52 &53 &36 &4.47e$-$09\\
  HS91       &&13 &67 &68 &36 &8.42e$-$10    &&36 &58 &59 &37 &8.80e$-$09\\
  HS92       &&2 &78 &79 &48 &6.60e$-$09    &&39 &58 &59 &40 &4.86e$-$11\\
  HS93       &&4 &28 &29 &23 &1.29e$-$09    &&24 &26 &27 &25 &2.17e$-$09\\
  HS95       &&5 &45 &46 &36 &5.06e$-$09    &&14 &15 &16 &15 &5.84e$-$11\\
  HS96       &&8 &161 &162 &116 &6.20e$-$09    &&11 &12 &13 &12 &5.84e$-$11\\
  HS97       &&4 &46 &50 &37 &9.92e$-$09    &&23 &26 &27 &24 &1.94e$-$11\\
  HS98       &&7 &69 &70 &57 &9.49e$-$09    &&23 &26 &27 &24 &1.17e$-$11\\
  HS99       &&- &- &- &- &-    &&20 &121 &122 &21 &4.00e$-$10\\
  HS100      &&1 &15 &16 &16 &9.16e$-$09    &&17 &36 &37 &18 &3.20e$-$09\\ 
  HS100MOD   &&10 &15 &16 &16 &5.19e$-$09   &&18 &35 &36 &19 &9.07e$-$09\\
  HS101       &&4 &554 &555 &328 &7.02e$-$09    &&- &- &- &- &-\\
  HS102       &&3 &302 &303 &214 &8.85e$-$09    &&- &- &- &- &-\\
  HS103       &&2 &199 &200 &158 &9.65e-09    &&74 &237 &238 &75 &2.90e$-$09\\
  HS104       &&4 &22 &23 &21 &2.88e$-$09    &&29 &36 &37 &30 &1.51e$-$09\\
  HS105       &&3 &22 &23 &18 &1.69e$-$09    &&48 &52 &53 &49 &7.94e$-$09\\
  HS106       &&- &- &- &- &-    &&51 &85 &86 &52 &1.13e$-$09\\
  HS107       &&8 &37 &38 &30 &9.17e$-09$    &&13 &15 &16 &14 &2.80e$-$09\\
  HS108       &&3 &24 &25 &23 &5.24e$-$09    &&- &- &- &- &-\\
  HS109       &&- &- &- &- &-    &&- &- &- &- &-\\
  HS110      &&4 &10 &11 &10 &2.95e$-$09    &&9 &15 &16 &9 &2.58e$-$08 \\
  HS111      &&2 &54 &55 &44 &5.96e$-$09    &&51 &80 &81 &52 &3.74e$-$09\\
  HS112      &&4 &91 &92 &77 &9.94e$-$09    &&27 &37 &38 &28 &2.30e$-$09\\
  HS113      &&3 &21 &22 &22 &4.03e$-$09    &&16 &18 &19 &17 &2.18e$-$09\\
  HS114       &&- &- &- &- &-    &&34 &41 &42 &35 &2.09e$-$09\\
  HS116       &&- &- &- &- &-    &&- &- &- &- &-\\
  HS117      &&4 &62 &63 &60 &5.77e$-$09    &&20 &20 &21 &21 &8.48e$-$09\\
  HS118       &&14 &26 &27 &27 &3.16e$-$09    &&34 &37 &38 &35 &3.48e$-$10\\
  HS119       &&15 &29 &30 &30 &3.78e$-$09    &&30 &30 &31 &31 &3.38e$-$09\\
  \hline
   total      && 556    & 3137	&3233  & 2717 &	 && 1822 & 3049	& 3142	 & 2012	\\
   \hline
   average    &&  6    & 34 	&35  &   29  &	 && 20 & 33	& 34 	& 22 \\ 
   \hline
\end{longtable}
}
\end{center}

\begin{figure}[htbp]
  \centering
  \begin{minipage}{0.45\textwidth}
    \centering
    \includegraphics[width=\textwidth]{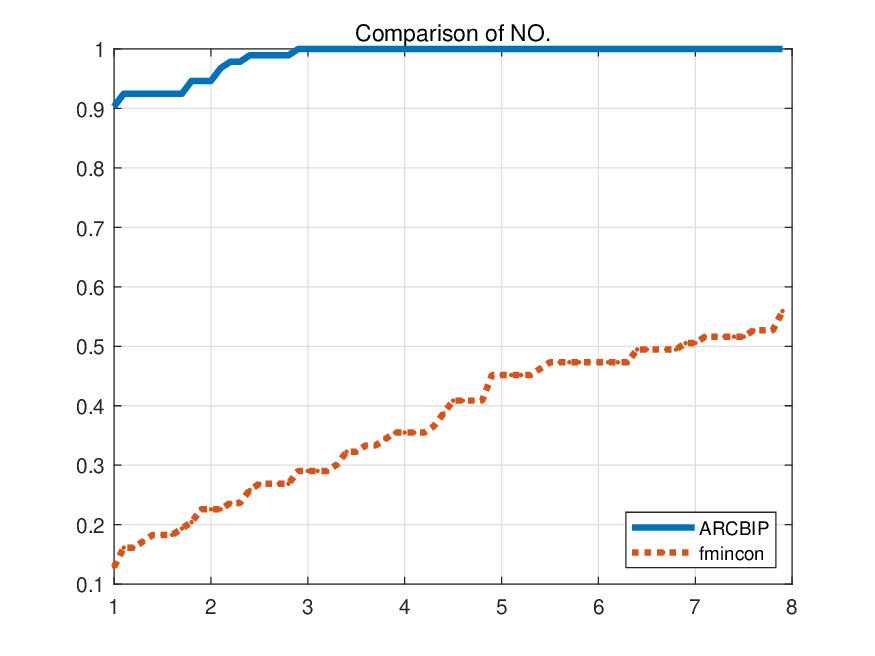}
  \end{minipage}
  \hfill 
  \begin{minipage}{0.45\textwidth}
    \centering
    \includegraphics[width=\textwidth]{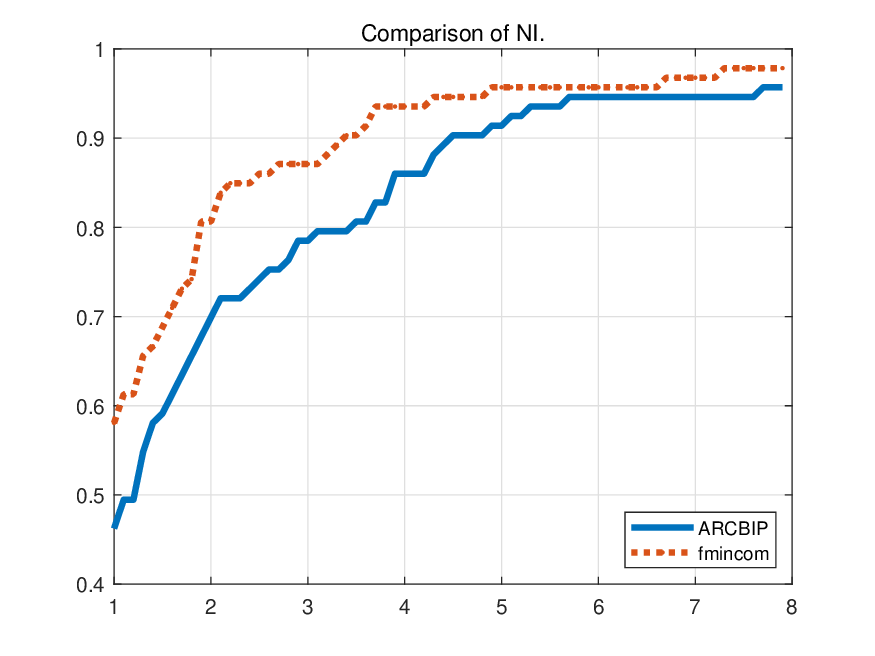}
  \end{minipage}
  
  \vspace{10pt} 
%
  \begin{minipage}{0.45\textwidth}
    \centering
    \includegraphics[width=\textwidth]{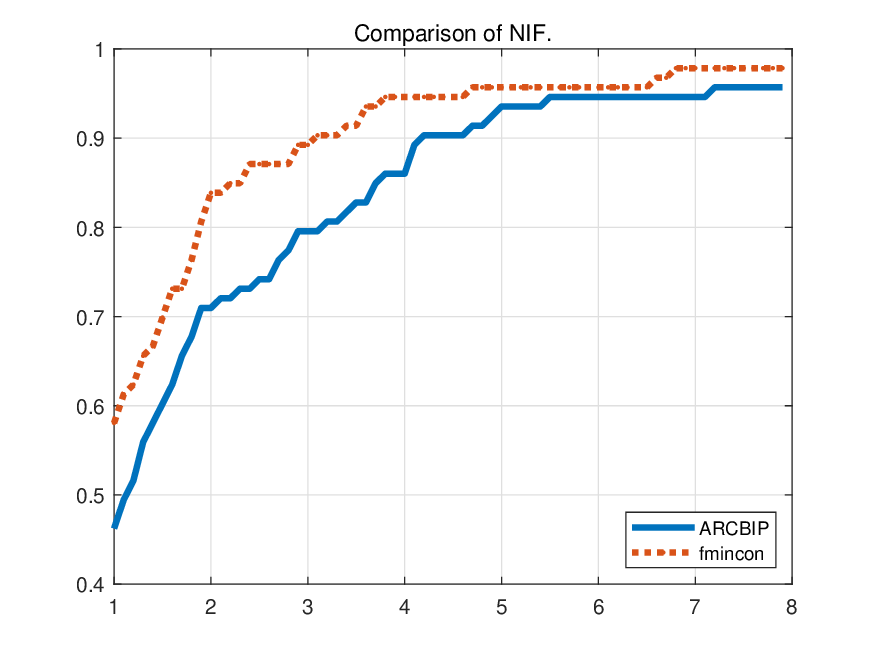}
  \end{minipage}
  \hfill
  \begin{minipage}{0.45\textwidth}
    \centering
    \includegraphics[width=\textwidth]{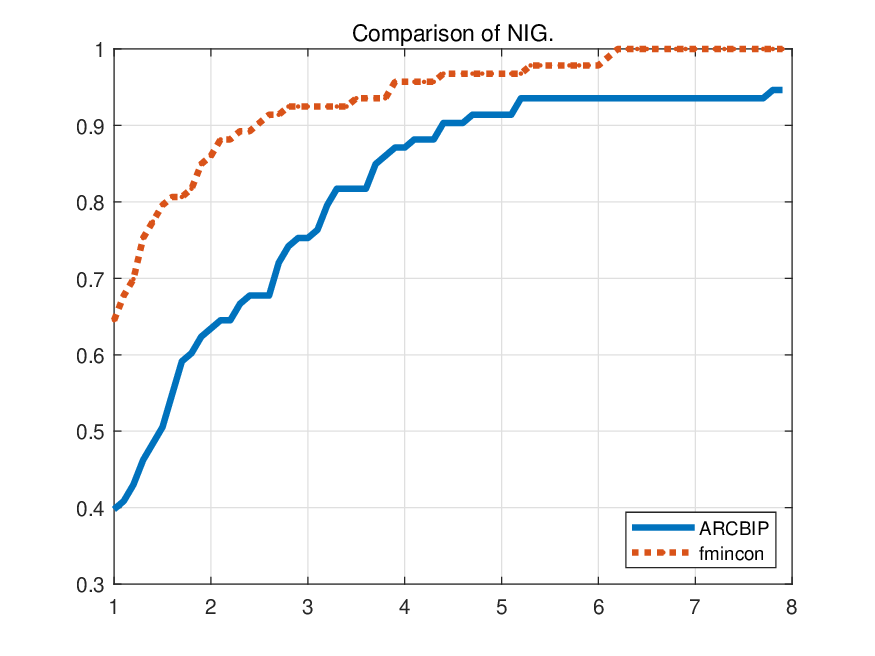}
  \end{minipage}
  
\caption{Performance profiles based on Table \ref{numercal table} on $NO$, $NI$, $NIF$, and $NIG$ (logarithmic scale for the performance ratio).}\label{Figure1}
\end{figure}

%

\begin{center}
	{
\setlength{\tabcolsep}{1pt}
		\begin{longtable}{cccccccccccccccccccccc}
			\caption{Comparison results II}\label{table4}
			\endfirsthead
			\multicolumn{17}{c}{Table \ref{table4} continued}\\
			\hline
			\multirow{2}{*}{Problem}  &&\multicolumn{3}{c}{ARCBIP }&&\multicolumn{3}{c}{Alg. 3.6 \cite{Liu2020}}&&\multicolumn{3}{c}{IPOPT}&&\multicolumn{3}{c}{fmincon} \\
			\cmidrule{3-5}
			\cmidrule{7-9}
            \cmidrule{11-13}
            \cmidrule{15-17}
			                     & &$NI$&$NIF$&$NIG$      &&$NI$&$NIF$&$NIG$        & &$NI$&$NIF$&$NIG$  &&$NI$&$NIF$&$NIG$& \\
			\hline
			\endhead
			\hline
			\multirow{2}{*}{Problem}  & &\multicolumn{3}{c}{ ARCBIP}&&\multicolumn{3}{c}{Alg. 3.6 \cite{Liu2020}}&&\multicolumn{3}{c}{IPOPT}&&\multicolumn{3}{c}{fmincon} \\
		   \cmidrule{3-5}
		   \cmidrule{7-9}
           \cmidrule{11-13}
           \cmidrule{15-17}
			                       & &$NI$&$NIF$&$NIG$      &&$NI$&$NIF$&$NIG$        & &$NI$&$NIF$&$NIG$  &&$NI$&$NIF$&$NIG$& \\
			\hline
  CB2      &&11 &12 &12 &&8 &9 &9 &&8 &9 &9 &&14 &15 &14  \\
  CB3      &&11 &12  &12 &&9 &10 &10 &&9 &10 &10 &&13 &14 &14 \\
  CHACONN1 &&11 &12  &12  &&8 &9 &9 &&8 &9 &9 &&13 &14 &14\\
  CHACONN2 &&11 &12 &12  &&9 &10 &10 &&8 &9 &9 &&13 &14 &14\\ 
  CONGIGMZ &&21 &22  &18  &&36 &69 &37 &&34 &42 &35  &&48 &49 &35\\ 
  DEMYMALO &&13 &14  &14  &&17 &25 &18 &&13 &14 &14 &&14 &15 &15\\ 
  DIPIGRI  &&12 &13 &13  &&24 &88 &25 &&12 &24 &13  &&36 &37 &18\\ 
  EXPFITA  &&17 &18  &18  &&30 &106 &31 &&18 &19 &19 &&46 &47 &47\\ 
  GIGOMEZ1 &&15 &16  &16  &&16 &43 &17 &&16 &17 &17 &&17 &18 &15\\
  GIGOMEZ2 &&15 &16 &16 &&8 &10 &9 &&9 &10 &10 &&13 &14 &14\\ 
  GIGOMEZ3 &&11 &12 &12 &&13 &28 &14  &&8 &9 &9 &&14 &15 &15\\ 
  GOFFIN   &&6 &7  &6 &&106 &297 &107 &&6 &7 &7 &&17 &18 &18\\  
  HAIFAS   &&16 &17  &16 &&12 &21 &13 &&12 &13 &13 &&166 &167 &117 \\ 
  HALDMADS &&25 &26 &25 &&15 &34 &16 &&46 &54 &41  &&38 &39 &34\\
  HS10     &&10 &11 &11 &&11 &12 &12 &&13 &14 &14 &&16 &17 &17\\ 
  HS11     &&8 &9 &9 &&7 &8 &8 &&8 &9 &9  &&10 &11 &11\\ 
  HS12     &&9 &10 &10 &&12 &36 &13 &&9 &10 &10 &&12 &13 &10\\
  HS14     &&13 &14 &14 &&12 &36 &13 &&7 &8 &8 &&9 &10 &9\\
  HS22     &&5 &6 &6 &&8 &10 &9 &&6 &7 &7 &&11 &12 &11\\ 
  HS29     &&8 &9 &9 &&35 &142 &36  &&8 &9 &9  &&17 &18 &15\\
  HS43     &&11 &12 &12 &&14 &18 &15 &&14 &18 &15 &&15 &16 &14\\ 
  HS100    &&15 &16 &16 &&24 &88 &25 &&9 &10 &10 &&36 &37 &18\\ 
  HS100MOD &&15 &16 &16 &&54 &263 &55 &&11 &29 &12 &&35 &36 &19\\
  HS113    &&21 &22 &22 &&54 &150 &55 &&11 &12 &12 &&18 &19 &17\\ 
  KIWCRESC &&12 &13  &13 &&13 &27 &14 &&9 &11 &10 &&13 &14 &14\\
  MADSEN   &&14 &15 &15 &&16 &19 &17  &&19 &20 &20  &&19 &20 &16\\
  MAKELA1  &&12 &13 &13 &&10 &22 &11 &&14 &15 &15  &&14 &15 &13\\
  MAKELA2  &&25 &26 &24 &&16 &41 &17 &&7 &8 &8 &&14 &15 &15\\
  MAKELA3  &&21 &22 &22 &&331 &365 &332  &&17 &19 &18 &&44 &45 &35\\
  MAKELA4  &&5 &6 &6 &&58 &354 &59 &&9 &11 &10  &&39 &40 &40\\  
  MIFFLIN1 &&11 &12 &12 &&7 &8 &9 &&6 &7 &7  &&12 &13 &12\\  
  MIFFLIN2 &&13 &14 &14 &&47 &227 &48 &&15 &16 &16 &&18 &19 &15\\  
  PENTAGON &&14 &15  &15 &&17 &26 &18  &&14 &15 &15 &&- &- &-\\
  POLAK1   &&11 &12 &12 &&165 &1569 &166 &&7 &8 &8 &&18 &19 &19\\
  POLAK3   &&22 &23 &21 &&22 &26 &23 &&- &- &- &&71 &72 &70\\
  POLAK5   &&30 &31 &29 &&7 &9 &8 &&31 &32 &32 &&- &- &-\\ 
  ROSENMMX &&35 &36 &36 &&31 &79 &32 &&16 &18 &17 &&30 &31 &25 \\ 
  S268     &&16 &17 &17 &&21 &106 &22 &&19 &20 &20 &&37 &38 &30\\
  SPIRAL   &&103 &104  &79 &&146 &319 &147 &&- &- &- &&160 &161 &109\\
  TFI1     &&31 &32  &30 &&34 &60 &35 &&81 &210 &82  &&60 &61 &43\\
  TFI3     &&26 &27 &27 &&147 &554 &148 &&16 &18 &17 &&23 &24 &24\\
  WOMFLET  &&9 &10  &10 &&23 &116 &24 &&9 &10 &10 &&75 &76 &37\\
  \hline
total   &&551  &589	&578	&&1461	&5069	&1500		&&547	&763	&579		&&1057	&1095	&863\\
\hline
average &&15	&16	&16		&&39	&134	&40		&&15	&20	&16		&&28	&29	&23\\ 
\hline
		\end{longtable}
	}
\end{center}

\begin{figure}[htbp]
  \centering
    \begin{minipage}{0.32\textwidth}
    \centering
    \includegraphics[width=\textwidth]{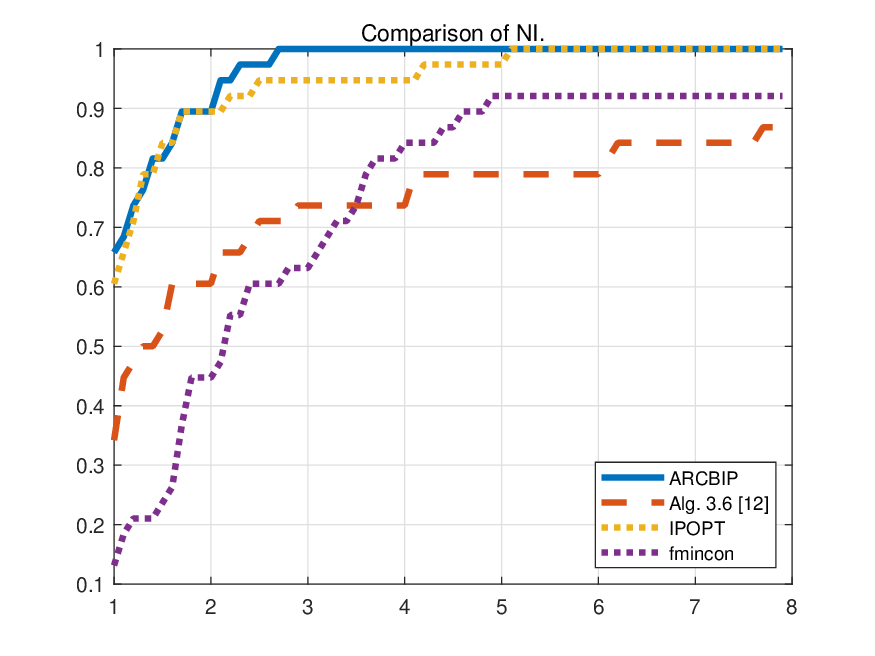}
  \end{minipage}
  \hfill
  \begin{minipage}{0.32\textwidth}
    \centering
    \includegraphics[width=\textwidth]{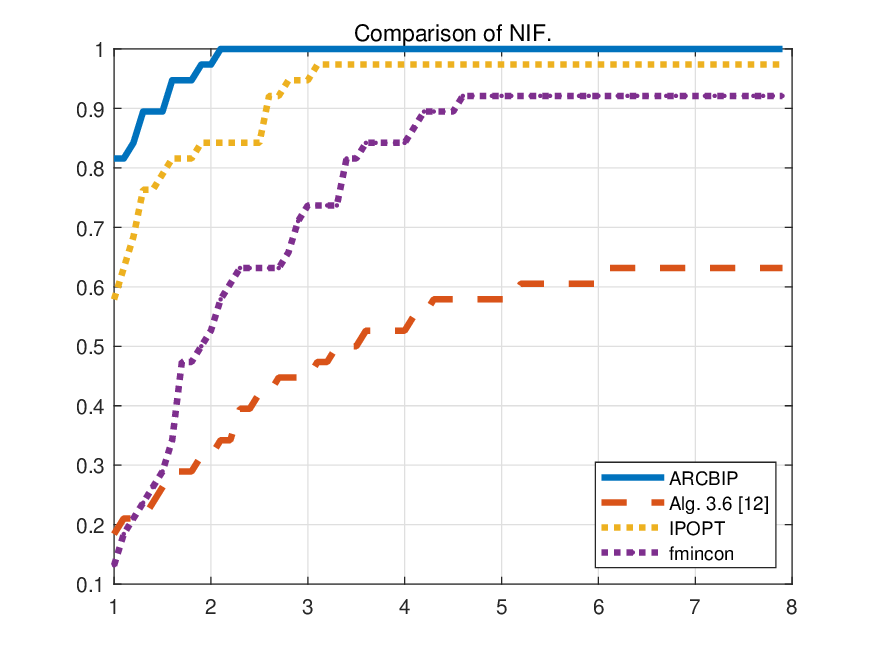}
  \end{minipage}
  \hfill
  \begin{minipage}{0.32\textwidth}
    \centering
    \includegraphics[width=\textwidth]{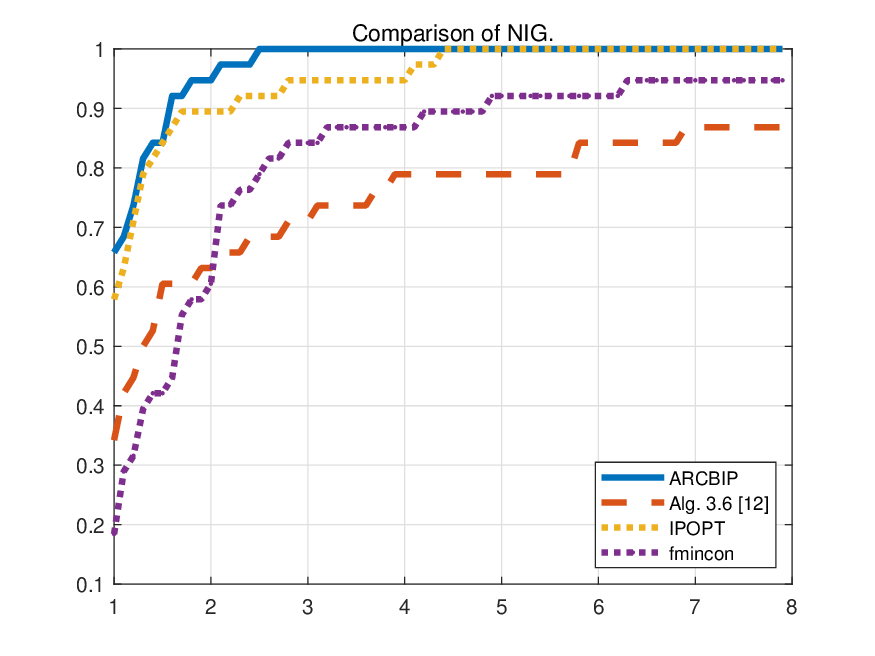}
  \end{minipage}
  \caption{ Performance profiles based on Table \ref{table4} on $NI$,  $NIF$,  and $NIG$ (logarithmic scale for the performance ratio).}\label{Figure2}
\end{figure}
\section{Conclusion}
\label{sec6}
We have introduced an implementable adaptive regularisation algorithm using cubics based on interior-point methods (ARCBIP) for solving nonlinear inequality constrained optimization.
The fraction to boundary rule is used to bound the slack variables away from zero, so that the algorithm does not stop too quickly on a solution with insufficient accuracy. In addition, ARCBIP is based on a primal version, while it also permits primal-dual steps to ensure that the primal and dual iterates to be interior-points. Since it is not easy to find an exact solution of the subproblem, we compute approximate solution and propose conditions that the approximate solution should satisfy. Unlike interior-point algorithms based on line search frameworks, which usually require that the Lagrangian Hessian or its approximation is uniformly positive definite on the null space of the Jacobian of active constraints to guarantee the descent property of the search directions, ARCBIP does not need positive definiteness of Lagrangian Hessian to ensure the global convergence. 
Preliminary numerical results and some comparison results are presented to demonstrate the performance of ARCBIP. 

One of our next goals is to analyze evaluation-complexity of ARCBIP. 
The other is to reduce the number of subproblems to be computed because two subproblems need to be solved in the inner loop of ARCBIP. Next we aim to design a method that only requires solving one subproblem. For example, instead of using the composite step method, we can try to develop an efficient algorithm to directly solve subproblem \eqref{cubicsubp}. 

%
%

\begin{acknowledgement}
The authors are very grateful to the editor and the anonymous referees, whose valuable suggestions and insightful comments helped to improve significantly the paper. The authors gratefully acknowledge the support of the Natural Science Foundation of Henan Province (252300421993), the National Natural Science Foundation of China (12071133), Key Scientific Research Project for Colleges and Universities in Henan Province (25B110005).
\end{acknowledgement}

\begin{dataa}
 There is no data generated or analysed in this paper.
\end{dataa}

\section*{Declarations}
\small  The authors declare that they have no conflict of interest.

\begin{appendices}
\section{Proofs of lemmas}\label{secA1}

\subsection{Proof of Lemma \ref{ykbound}}\label{le1}
\begin{proof}
  Assumption (AS2) implies that $ \{f_k \} $ is bounded below and $ \{g_k \} $ is bounded. Suppose that $\gamma_{ub}$ represents the upper bound of $ \{-f_k\} $ and $\| g_k \|$. We have logarithmic inequalities that
  \begin{equation}\label{lny}
    \sum_{i=1}^m \ln y_k^{(i)} \leq m \ln \|y_k\|_{\infty} \leq m \ln \|y_k\|.
  \end{equation}
   From \eqref{lny}, the facts that $\text{pred}_k$ is non-negative and that the sequence $\{\nu_k\}$ is monotonic non-decreasing, expanding the right side of the inequality in (\ref{tildephi}) and summing up, it can be inferred that
  \begin{equation}\label{uptildephi}
    \tilde{\phi}(x_k, y_k; \nu_k) \leq \tilde{\phi}( x_0 , y_0 ; \nu_0) +\left( \frac{1}{\nu_0}-\frac{1}{\nu_k} \right) (\gamma_{ub} + \mu m \max_{0 \leq i \leq k} \ln \| y_i \|). \nonumber
  \end{equation}
  Conversely, using the definition of $\tilde{\phi}$, we can also obtain that
    \begin{equation}\label{downtildephi}
    \tilde{\phi}( x_k, y_k; \nu_k) \geq - \frac{1}{\nu_k} ( \gamma_{ub} + \mu m \ln \| y_k \| ) + \| y_k \| - \| g_k \| \nonumber
  \end{equation}
  for any $k$. Now, we define the indices $l_i$ such that the condition $ \| y_{l_i} \| = \max_{k \leq l_i} \| y_k \| $ is satisfied. Then combining the above two inequalities for $k$ given by any such $l_i$ we have that
  \begin{equation}
  \begin{aligned}
    &~- \frac{1}{\nu_{l_i}} ( \gamma_{ub} + \mu m \ln \| y_{l_i} \| )  + \| y_{l_i} \| - \| g_{l_i} \|  
    \leq \tilde{\phi}( x_0 , y_0 ; \nu_0 ) +\left( \frac{1}{\nu_0} - \frac{1}{\nu_{l_i}} \right) ( \gamma_{ub} + \mu m \ln \| y_{l_i} \| ) \nonumber
  \end{aligned}  
  \end{equation}
  and thus
  \begin{equation}\label{upyli}
    \| y_{l_i} \| \leq \tilde{\phi}( x_0 , y_0 ; \nu_0 ) + \cfrac{1}{\nu_0} ( \gamma_{ub} + \mu m \ln \| y_{l_i} \| ) +\gamma_{ub} . \nonumber
  \end{equation}
  Since the ratio $ (\ln \| y \| / \| y \|) $ tends to $0$ when $ \| y \| \rightarrow \infty $, 
  $ \{ y_{l_i} \} $ must be bounded. By definition of the indices $ l_i $, we conclude that the whole sequence $ \{ y_k \} $ is bounded.
\end{proof}

\subsection{Proof of Lemma \ref{npredlow}}\label{le2}
\begin{proof}
If $u_k^c = 0$, it is obviously true from \eqref{ukc} that the right hand side of the inequality (\ref{npredboun}) is zero.  

If $u_k^c \neq 0$, by the change of variables (\ref{u}) and the normal Cauchy decrease condition problem \eqref{NCauchypoi}, the scalar $ \alpha_k^c $ is a solution of
\begin{subequations}\label{NCu}
  \begin{align}
   \underset{\alpha \geq 0}{\text{minimize}} \quad &~ \cfrac{1}{2} \| g_k+ y_k + \alpha (A_k^T u_x^c + Y_k u_y^c)\|^2 +\frac{1}{3} \widetilde{\sigma}_k \left\|\alpha u_k^c \right\|^3\\
   \text{subject to} \quad &~ \alpha  \leq - \frac{\xi \tau }{\left(u_y^c\right)^{(i)}} \quad \text{for all } i \text{ such that } \left(u_y^c\right)^{(i)} < 0
\end{align}
\end{subequations}
with $u_k^c := \begin{pmatrix}
                 u_x^c \\
                 u_y^c 
               \end{pmatrix}$.
The upper bounds of problem \eqref{NCu} are also satisfied if
  $\alpha  \leq  \frac{\xi \tau }{\left\|u_k^c\right\|}.$

From the definition (\ref{ukc}) of $u_k^c$, the normal Cauchy decrease condition (\ref{NCauchycon}), the Cauchy-Schwarz inequality and the fact that $ a^2 -b^2 \leq 2 a (a-b) $, it can be inferred that
  \begin{eqnarray}\label{negnpred}
    &     & - \| g_k + y_k \| \text{npred}_k(n_k) \nonumber \\
    &\leq & -\gamma_n \| g_k + y_k \| \text{npred}_k(\alpha_k^c n_k^c) \nonumber  \\
    &  =  & -\gamma_n \| g_k + y_k \| \left[\left\| g_k+y_k \right\| - \left\| g_k+y_k + \alpha_k^c \left(A_k^T n_x^c +n_y^c\right) \right\| \right] \nonumber  \\
    &     & +\frac{1}{3} (\alpha_k^c)^3 \gamma_n \widetilde{\sigma}_k \| g_k + y_k \| \left\| D_k n_k^c \right\|^3 \nonumber  \\
    &\leq & -\frac{\gamma_n}{2} \left[\left\| g_k+y_k \right\|^2- \left\| g_k+y_k + \alpha_k^c \left(A_k^T u_x^c + Y_k u_y^c \right) \right\|^2 \right] \nonumber  \\
    &     & + \frac{1}{3} (\alpha_k^c)^3 \gamma_n \widetilde{\sigma}_k \| g_k + y_k \| \left\| u_k^c \right\|^3 \nonumber  \\
    &  =  & -\frac{\gamma_n}{2}\left[ -(\alpha_k^c)^2 \left\| A_k^T u_x^c + Y_k u_y^c \right\|^2 -2 \alpha_k^c (g_k+y_k)^T \left(A_k^T u_x^c + Y_k u_y^c\right) \right] \nonumber  \\ 
    &     & +\frac{1}{3} (\alpha_k^c)^3 \gamma_n \widetilde{\sigma}_k \| g_k + y_k \| \left\| u_k^c \right\|^3 \nonumber  \\
    &  =  & \gamma_n \left( -\alpha_k^c \left\| u_k^c \right\|^2 + \frac{1}{2} (\alpha_k^c)^2 \left\| 
    \begin{pmatrix}
      A_k^T & Y_k 
    \end{pmatrix}
     u_k^c \right\|^2 + \frac{1}{3} (\alpha_k^c)^3 \widetilde{\sigma}_k \| g_k + y_k \| \left\| u_k^c \right\|^3 \right) \nonumber  \\
    &\leq & \gamma_n \alpha_k^c \left\| u_k^c \right\|^2 \left( -1 + \frac{1}{2} \alpha_k^c \left\| 
    \begin{pmatrix}
      A_k^T & Y_k 
    \end{pmatrix}
    \right\|^2 + \frac{1}{3} (\alpha_k^c)^2 \widetilde{\sigma}_k \| g_k + y_k \|  \left\| u_k^c \right\| \right).
  \end{eqnarray}
Hence we can obtain from \eqref{negnpred} that $ \text{npred}_k (n_k) \geq 0 $ provided
\begin{equation} 
-1 + \cfrac{1}{2} \alpha_k^c \left\| 
 \begin{pmatrix}
  A_k^T & Y_k 
 \end{pmatrix}
 \right\|^2 + \cfrac{1}{3} (\alpha_k^c)^2 \widetilde{\sigma}_k \| g_k + y_k \|  \left\| u_k^c \right\| \leq 0 \nonumber
 \end{equation}
 since $\alpha_k^c \geq 0$.
The above inequality is equivalent to $\alpha_k^c \in  \left[ 0, \bar{\alpha}_k \right]$, where
 \begin{equation}
   \bar{\alpha}_k := \frac{3\left[ -\cfrac{1}{2} \left\| 
   \begin{pmatrix}
      A_k^T & Y_k 
    \end{pmatrix}
    \right\|^2 + \sqrt{\cfrac{1}{4} \left\| 
    \begin{pmatrix}
      A_k^T & Y_k 
    \end{pmatrix}
    \right\|^4 +\cfrac{4}{3} \widetilde{\sigma}_k \| g_k + y_k \|  \left\| u_k^c \right\| } \right]}{2\widetilde{\sigma}_k \| g_k + y_k \|  \left\| u_k^c \right\|} .  \\ \nonumber
 \end{equation}
 Furthermore, we can express $\bar{\alpha}_k$ as
 \begin{equation}
   \bar{\alpha}_k = 2 \left[ \frac{1}{2} \left\| 
   \begin{pmatrix}
      A_k^T & Y_k 
    \end{pmatrix}
    \right\|^2 + \sqrt{\frac{1}{4} \left\| 
    \begin{pmatrix}
      A_k^T & Y_k 
    \end{pmatrix}
    \right\|^4 +\frac{4}{3} \widetilde{\sigma}_k \| g_k + y_k \|  \left\| u_k^c \right\| }  \right]^{-1}. \nonumber
 \end{equation}
 Let
 \begin{equation}\label{omegak}
   \omega_k := \left[ \sqrt{2} \max \left\{\left\| 
    \begin{pmatrix}
      A_k^T & Y_k 
    \end{pmatrix}
    \right\|^2 , 2 \sqrt{\widetilde{\sigma}_k \| g_k + y_k \|  \left\| u_k^c \right\|} \right\}\right ]^{-1}.
 \end{equation}
 From the inequalities 
 \begin{equation}
   \begin{aligned}
     &~ \sqrt{\frac{1}{4} \left\| 
     \begin{pmatrix}
       A_k^T & Y_k 
     \end{pmatrix}
     \right\|^4 +\frac{4}{3} \widetilde{\sigma}_k \| g_k + y_k \|  \left\| u_k^c \right\| } 
     \leq  \cfrac{1}{2} \left\| 
     \begin{pmatrix}
       A_k^T & Y_k 
     \end{pmatrix}
     \right\|^2 + \cfrac{2}{\sqrt{3}} \sqrt{\widetilde{\sigma}_k \| g_k + y_k \|  \left\| u_k^c \right\|} \\
     \leq &~ 2 \max \left\{ \frac{1}{2} \left\| 
    \begin{pmatrix}
      A_k^T & Y_k 
    \end{pmatrix}
    \right\|^2 , \frac{2}{\sqrt{3}} \sqrt{\widetilde{\sigma}_k \| g_k + y_k \|  \left\| u_k^c \right\|} \right\} \\
     \leq &~ \sqrt{2} \max \left\{ \left\| 
    \begin{pmatrix}
      A_k^T & Y_k 
    \end{pmatrix}
    \right\|^2 , 2 \sqrt{\widetilde{\sigma}_k \| g_k + y_k \|  \left\| u_k^c \right\|} \right\} \nonumber
    \end{aligned}
 \end{equation}
 obtained by $\sqrt{a^2 + b^2} \leq a + b (a,b \geq 0)$ and 
 \begin{equation}
    \frac{1}{2} \left\| 
    \begin{pmatrix}
      A_k^T & Y_k 
    \end{pmatrix}
    \right\|^2 \leq \sqrt{2} \max \left\{ \left\| 
    \begin{pmatrix}
      A_k^T & Y_k 
    \end{pmatrix}
    \right\|^2 , 2 \sqrt{\widetilde{\sigma}_k \| g_k + y_k \|  \left\| u_k^c \right\|} \right\}, \nonumber
 \end{equation}
 it follows that $ 0 < \omega_k \leq \bar{\alpha}_k $.
 Also due to the upper bounds of $\alpha$, and replacing $\alpha$ in (\ref{negnpred}) with $ \omega_k $, we have that
 \begin{eqnarray}\label{negnpred2}
    & & - \| g_k + y_k \|  \text{npred}_k(n_k) \nonumber\\ 
    &\leq & \gamma_n \left\| u_k^c \right\|^2 \left( -1 + \frac{1}{2} \omega_k \left\| 
    \begin{pmatrix}
      A_k^T & Y_k 
    \end{pmatrix}
    \right\|^2 + \frac{1}{3} \omega_k^2 \widetilde{\sigma}_k \| g_k + y_k \|  \left\| u_k^c \right\| \right) \nonumber \\
    & & \min \left\{ \left[ \sqrt{2} \max \left\{\left\| 
    \begin{pmatrix}
      A_k^T & Y_k 
    \end{pmatrix}
    \right\|^2 , 2 \sqrt{\widetilde{\sigma}_k \| g_k + y_k \|  \left\| u_k^c \right\|} \right\}\right ]^{-1} , \; \frac{\xi \tau }{\left\|u_k^c\right\|} \right\}.
 \end{eqnarray} 
It now follows from \eqref{omegak} that $\omega_k \left\| 
\begin{pmatrix}
  A_k^T & Y_k 
\end{pmatrix}
\right\|^2 \leq 1$ and $ \omega_k^2 \widetilde{\sigma}_k \| g_k + y_k \|  \left\| u_k^c \right\|  \leq 1  $. Hence the expression in the round brackets in \eqref{negnpred2}
is bounded above by $-1/6$. This and (\ref{negnpred2}) imply the inequality in (\ref{npredboun}).
\end{proof}

\subsection{Proof of Lemma \ref{treduplemma}}\label{le3}
\begin{proof}
  By problem \eqref{TCauchypoi} in the tangential Cauchy decrease condition, the scalar $ \beta_k^c $ is a solution of
  \begin{subequations}\label{TCau}
  \begin{align}
    \underset{\beta \geq 0}{\text{minimize}} \; &~ -\text{tpred}_k (\beta N_k p_k^c) = 
     \left(\nabla f_k + B_k n_x\right)^T \beta N_x p_k^c  - \mu \left( Y^{-1}_k e - Y_k^{-2} n_y \right)^T \beta N_y p_k^c \nonumber \\
    &~ + \frac{1}{2} \beta^2 (p_k^c)^T (N_x^T B_k N_x + \mu N_y^T Y^{-2}_k N_y) p_k^c +\frac{1}{3} \beta^3 \sigma_k \|D_k N_k p_k^c\| ^3  \nonumber  \\  
    \text{subject to} \; &~ \beta \leq \frac{-\tau + (Y^{-1}_k n_y)^{(i)}}{(Y^{-1}_k N_y p_k^c)^{(i)}}, \text{ for all } i \text{ such that } (Y^{-1}_k N_y p_k^c)^{(i)} < 0. \nonumber
  \end{align}
\end{subequations}
Since the normal subproblem ensures that $ (Y_k^{-1} n_y)^{(i)} \geq -\xi \tau $, it follows from the definition of the Euclidean norm that
\begin{equation}\label{betaup}
  \beta \leq \cfrac{( 1-\xi ) \tau}{\left\|
  \begin{pmatrix}
   N_x p_k^c \\
   Y_k^{-1} N_y p_k^c
\end{pmatrix} \right\|}.
\end{equation}

Assumptions (AS2)-(AS3) imply that
$\{B_k\}$ is bounded and $ \|N_k\| \leq \gamma_N $. For $\beta_k^c \geq 0$, using the definition (\ref{pkc}) of $ p_k^c $, the tangential Cauchy decrease condition (\ref{TCauchycon}) and the Cauchy-Schwarz inequality, we have that
  \begin{align}\label{negtpred}
    &~ -\text{tpred}_k(t_k) \nonumber \\
     \leq &~ - \gamma_t \text{tpred}_k(\beta_k^c N_k p_k^c) \nonumber \\
    = &~ \gamma_t \left[ \left(\nabla f_k + B_k n_x\right)^T \beta_k^c N_x p_k^c - \mu \left( Y^{-1}_k e - Y_k^{-2} n_y \right)^T \beta_k^c N_y p_k^c \right. \nonumber \\
    &~ \left. + \cfrac{1}{2} (\beta_k^c)^2 (p_k^c)^T (N_x^T B_k N_x + \mu N_y^T Y^{-2}_k N_y) p_k^c + \cfrac{1}{3} (\beta_k^c)^3 \sigma_k \|D_k N_k p_k^c\| ^3 \right]  \nonumber \\
    \leq &~ \gamma_t \beta_k^c \|p_k^c\|^2 \left( -1 + \cfrac{1}{2} \beta_k^c \| W_k^N \| + \cfrac{1}{3} (\beta_k^c)^2 \sigma_k \|D_k\| ^3 \|N_k \| ^3 \| p_k^c \| \right) \nonumber \\
    \leq &~ \gamma_t \beta_k^c \|p_k^c\|^2 \left( -1 + \cfrac{1}{2} \beta_k^c \| W_k^N \| + \cfrac{1}{3} (\beta_k^c)^2 \sigma_k 
    \gamma_N^3 \| D_k \|^3 \| p_k^c \| \right),
  \end{align}
where 
\begin{equation}\label{Wkn}
   W_k^N := N_x^T B_k N_x + \mu N_y^T Y^{-2}_k N_y .
  \end{equation}
Hence we can obtain from \eqref{negtpred} that $ \text{tpred}_k (n_k) \geq 0 $ provided
\begin{equation}
 -1 + \cfrac{1}{2} \beta_k^c \| W_k^N \| + \cfrac{1}{3} (\beta_k^c)^2 \sigma_k  \gamma_N ^3 \| D_k \|^3 \| p_k^c \|  \leq 0\nonumber
\end{equation}
because $\beta_k^c \geq 0$.
Then the above inequality is equivalent to $\beta_k^c \in [ 0 , \bar{\beta}_k ]$, where

\begin{equation}
  \begin{aligned}
    &~ \bar{\beta}_k := \cfrac{3\left( -\cfrac{1}{2} \| W_k^N \| +\sqrt{ \cfrac{1}{4} \| W_k^N \|^2 + \cfrac{4}{3} \sigma_k \gamma_N ^3 \| D_k \|^3 \| p_k^c \| } \right)}{2 \sigma_k \gamma_N ^3 \| D_k \|^3 \| p_k^c \|}.\nonumber
  \end{aligned}
\end{equation}
Furthermore, we can express $ \bar{\beta}_k $ as
\begin{equation}\label{betaba}
  \bar{\beta}_k = 2 \left( \frac{1}{2} \| W_k^N \| +\sqrt{ \cfrac{1}{4} \| W_k^N \|^2 + \cfrac{4}{3} \sigma_k \gamma_N ^3 \| D_k \|^3 \| p_k^c \| } \right)^{-1}. \nonumber
\end{equation}
Let
\begin{equation}\label{thetak}
  \theta_k := \left[ \sqrt{2} \max \left\{\left\| W_k^N
    \right\| , 2 \sqrt{\sigma_k \gamma_N ^3 \| D_k \|^3 \| p_k^c \|} \right\}\right ]^{-1}.
\end{equation}
From the inequalities 
\begin{equation}
  \begin{aligned}
    \sqrt{ \cfrac{1}{4} \| W_k^N \|^2 + \cfrac{4}{3} \sigma_k \gamma_N^3 \| D_k \|^3 \| p_k^c \| }
    \leq &~ \cfrac{1}{2} \| W_k^N \| + \cfrac{2}{\sqrt{3}} \sqrt{ \sigma_k \gamma_N ^3 \| D_k \|^3 \| p_k^c \|  }\\
    \leq &~  \sqrt{2} \max \left\{\left\| W_k^N \right\| , 2 \sqrt{\sigma_k \gamma_N^3 \| D_k \|^3 \| p_k^c \|} \right\}\nonumber
  \end{aligned}
\end{equation}
obtained by $\sqrt{a^2 + b^2} \leq a + b (a,b \geq 0)$ and 
\begin{equation}
  \frac{1}{2} \| W_k^N \| \leq \sqrt{2} \max \left\{\left\| W_k^N \right\| , 2 \sqrt{\sigma_k \gamma_N^3 \| D_k \|^3 \| p_k^c \|} \right\},\nonumber
\end{equation}
it follows that $ 0 \leq \theta_k \leq \bar{\beta}_k $.
Also due to the upper bounds (\ref{betaup}) of $\beta$, and replacing $\beta$ in (\ref{negtpred}) with $ \theta_k $, we have that
\begin{equation}\label{negtpred2}
  \begin{aligned}
    &~ -\text{tpred}_k(t_k) \\
    \leq &~ \gamma_t \|p_k^c\|^2 \left( -1 + \cfrac{1}{2} \theta_k \| W_k^N \| + \cfrac{1}{3} \theta_k^2 \sigma_k \gamma_N^3 \| D_k \|^3 \| p_k^c \| \right)\\
    &~ \min \left\{ \left[ \sqrt{2} \max \left(\left\| W_k^N
    \right\| , 2 \sqrt{\sigma_k \gamma_N^3 \| D_k \|^3 \| p_k^c \|} \right)\right]^{-1} , \cfrac{( 1-\xi ) \tau}{\left\| 
    \begin{pmatrix}
   N_x p_k^c \\
   Y_k^{-1} N_y p_k^c
\end{pmatrix} \right\|} \right\}.
  \end{aligned} 
\end{equation}
It now follows from \eqref{thetak} that $\theta_k \| W_k^N \| \leq 1$ and $ \theta_k^2 \sigma_k \gamma_N^3 \| D_k \|^3 \| p_k^c \| \leq 1 $. So the expression in the round brackets in \eqref{negtpred2}
 is bounded above by $-1/6$. This and \eqref{negtpred2} imply the inequality in \eqref{tpredboun}.
\end{proof}

\subsection{Proof of Lemma \ref{nkuplemma}}\label{le4}
\begin{proof}
  Assumption (AS2) implies that Lemma \ref{ykbound} holds. It can be inferred that (\ref{Dup}) holds. By the definition of the normal predicted reduction \eqref{npred}, the change of variables (\ref{u}) and the Cauchy-Schwarz inequality, we have that
  \begin{equation}
    \begin{aligned}
     -\text{npred}_k (n_k) 
=&~ \left\| g_k+y_k+A_k^T n_x +n_y \right\| - \left\| g_k+y_k \right\|  +\frac{1}{3} \widetilde{\sigma}_k \left\| D_k n_k \right\|^3 \\
     \geq &~ - \left\| A_k^T u_x + Y_k u_y \right\|  + \frac{1}{3} \widetilde{\sigma}_k \left\| u_k \right\|^3 
     \geq   -  \left\| 
     \begin{pmatrix}
      A_k^T & Y_k 
     \end{pmatrix}
     \right\| \left\| u_k \right\| + \frac{1}{3} \widetilde{\sigma}_k \left\| u_k \right\|^3. \\
      \nonumber
    \end{aligned}
  \end{equation}
   Hence, $ -\text{npred}_k ( n_k ) > 0 $ whenever 
     \begin{equation}
       \left\| u_k \right\| > \sqrt{3} \sqrt{ \cfrac{ \left\| 
      \begin{pmatrix}
      A_k^T & Y_k 
      \end{pmatrix}
      \right\|}{\tilde{\sigma}_k} } . \nonumber
     \end{equation}
     Substituting $ u $ back to $ n $, using the Cauchy-Schwarz inequality we obtain that
     \begin{equation}
       \left\| n_k \right\| > \frac{\sqrt{3}}{\|D_k\|} \sqrt{ \cfrac{ \left\| 
      \begin{pmatrix}
      A_k^T & Y_k 
      \end{pmatrix}
      \right\|}{\tilde{\sigma}_k} } . \nonumber
     \end{equation}
     But $ \text{npred}_k ( n_k ) \geq 0 $. Thus it yields that
     \begin{equation}
       \left\| n_k \right\| \leq \frac{\sqrt{3}}{\|D_k\|} \sqrt{ \cfrac{ \left\| 
      \begin{pmatrix}
      A_k^T & Y_k 
      \end{pmatrix}
      \right\|}{\tilde{\sigma}_k} } . \nonumber
     \end{equation}
The proof is complete.
\end{proof}

\subsection{Proof of Lemma \ref{tkuplemma}}\label{le5}
\begin{proof}
   We can deduce from \eqref{hatt} and \eqref{tN} that 
   \begin{equation}\label{tnp}
    t_k =
   \begin{pmatrix}
  t_x \\
  t_y
\end{pmatrix}=
\begin{pmatrix}
  N_x p_k \\
  N_y p_k
\end{pmatrix}=N_k p_k.
\end{equation}
    Due to \eqref{pkc}, \eqref{tpred}, 
    \eqref{tnp}, the assumption (AS3), for $ k \geq 0 $, it can be inferred that
  \begin{equation}
  \begin{aligned}
  &~ - \text{tpred}_k (t_k) \\
  = &~ - \text{tpred}_k ( N_k p_k ) \\
  = &~  \left(\nabla f_k + B_k n_x\right)^T N_x p_k - \mu \left( Y^{-1}_k e - Y_k^{-2} n_y \right)^T N_y p_k \\
  &~ + \frac{1}{2} p_k^T (N_x^T B_k N_x + \mu N_y^T Y^{-2}_k N_y)p_k + \frac{1}{3}\sigma_k \left\|
\begin{pmatrix}
  N_x p_k \\
  Y^{-1}_k N_y p_k 
\end{pmatrix}
\right\|^3 \\
  = &~ -(p_k^c)^T p_k + \frac{1}{2} p_k^T (N_x^T B_k N_x + \mu N_y^T Y^{-2}_k N_y)p_k + \frac{1}{3}\sigma_k \left\|
\begin{pmatrix}
  N_x p_k \\
  Y^{-1}_k N_y p_k 
\end{pmatrix}
\right\|^3 \\
  \geq &~ -\|p_k^c\| \|p_k\| - \frac{1}{2} \| p_k \|^2 \| W_k^N \|  + \frac{1}{3}\sigma_k  \gamma_N^{-3} \|D_k\|^3 \|p_k\|^3 \\
  = &~ \left( \frac{1}{9} \sigma_k  \gamma_N^{-3} \|D_k\|^3 \|p_k\|^3 -\|p_k^c\| \|p_k\| \right) + \left( \frac{2}{9} \sigma_k \gamma_N^{-3} \|D_k\|^3 \|p_k\|^3 - \frac{1}{2} \| p_k \|^2 \| W_k^N \|  \right) , \nonumber
  \end{aligned}
  \end{equation}
  where $W_k^N$ is defined by \eqref{Wkn} and $ \gamma_N^{-1} $ is given by \eqref{N_kbou}. But 
  \begin{equation}
  \cfrac{1}{9} \sigma_k  \gamma_N^{-3} \|D_k\|^3 \|p_k\|^3 -\|p_k^c\| \|p_k\| > 0 \nonumber
  \end{equation}
  if
$   \| p_k \| > 3 \gamma_N^{\frac{3}{2}} \sqrt{ \frac{\|p_k^c\|}{\sigma_k \|D_k\|^3}}, $ 
  while 
  \begin{equation}
   \cfrac{2}{9} \sigma_k \gamma_N^{-3} \|D_k\|^3 \|p_k\|^3 - \cfrac{1}{2} \| p_k \|^2 \| W_k^N \| >0 \nonumber
  \end{equation}
  if
$
  \| p_k \| > \frac{9 \gamma_N^3\| W_k^N \| }{4\sigma_k\|D_k\|^3}.  
$
    Hence, $ -\text{tpred}_k > 0 $ whenever 
  \begin{equation}
    \|p_k\| > \cfrac{3 \gamma_N^{\frac{3}{2}} }{\sigma_k} \max \left\{ \sqrt{\frac{ \sigma_k \|p_k^c\|}{\|D_k\|^3}} , \cfrac{3 \gamma_N^{\frac{3}{2}} \| W_k^N \|}{4\|D_k\|^3} \right\}. \nonumber
  \end{equation}
  But $ \text{tpred}_k \geq 0 $ due to (\ref{tpredgeq0}). Then
  \begin{equation}
    \|p_k\| \leq \cfrac{3 \gamma_N^{\frac{3}{2}} }{\sigma_k} \max \left\{ \sqrt{\frac{ \sigma_k \|p_k^c\|}{\|D_k\|^3}} , \cfrac{3 \gamma_N^{\frac{3}{2}} \| W_k^N \|}{4\|D_k\|^3} \right\}. \nonumber
  \end{equation}
  Due to assumption (AS3), the boundedness of $ \{B_k\} $ and $ \{y_k\} $ and Cauchy-Schwarz inequality, we obtain the result (\ref{tkbou}) from \eqref{tnp}.
\end{proof}

\subsection{Proof of Lemma \ref{dkuplemma}}\label{le6}
\begin{proof}
Based on the assumptions (AS2)-(AS3), we have that the upper bound (\ref{nkup}) and (\ref{tkbou}) of the step $ n_k $ and $t_k$, respectively. There also exist that $ \{ g_k \} $, $ \{ y_k \} $ and $ \{ A_k \} $ are bounded. 
Then from \eqref{sigmahat}, \eqref{nkup} and the fact that $ \xi \in (0,1) $, we have  that
    \begin{equation}\label{nupp}
     \| n_k \| \leq \frac{\sqrt{3}}{\|D_k\|}
     \sqrt{\frac{ \left\| 
     \begin{pmatrix}
       A_k^T & Y_k
     \end{pmatrix}
     \right\|}{\tilde{\sigma}_k} } \leq \cfrac{\gamma_n^{\prime }}{ \sqrt{ \tilde{\sigma}_k}\|D_k\|} = \cfrac{\gamma_n^{\prime \prime}}{ \sqrt{ \sigma_k }\|D_k\|},
    \end{equation}
    where $ \gamma_n^{\prime \prime} = \xi^{\frac{3}{2}} \gamma_n^{\prime } $.
    From \eqref{D} and \eqref{pkc}, we have that
    \begin{eqnarray*}
p_k^c  &=& -N_x^T \left(\nabla f_k + B_k n_x\right) + \mu N_y^T \left( Y^{-1}_k e - Y_k^{-2} n_y \right)\\
     &=& N_k^T \begin{pmatrix}
                     -(\nabla f_k+B_k n_x) \\
                     \mu(Y^{-1}_k e - Y_k^{-2} n_y) 
                   \end{pmatrix} = N_k^T \begin{pmatrix}
                 I & 0 \\
                 0 & Y_k^{-1} 
               \end{pmatrix}
               \begin{pmatrix}
                     -(\nabla f_k+B_k n_x) \\
                     \mu( e - Y_k^{-1} n_y) 
               \end{pmatrix} \\
    &=& N_k^T D_k \begin{pmatrix}
                     \nabla f_k \\
                     -\mu e
                   \end{pmatrix}
        - N_k^T D_k \begin{pmatrix}
                 I & 0 \\
                 0 & Y_k^{-1} 
               \end{pmatrix} 
               \begin{pmatrix}
                     B_k \\
                     \mu I
               \end{pmatrix}
        n_k.
    \end{eqnarray*}
    Then, assumptions (AS2)-(AS3) imply that 
    \begin{equation}
    \begin{aligned}\label{pkcup}
    \|p_k^c\| &\leq \|N_k\| \|D_k\| \left\|\begin{pmatrix}
                     \nabla f_k \\
                     -\mu e
                   \end{pmatrix}\right\| + \|N_k\| \|D_k\|^2 
                   \left\|
                   \begin{pmatrix}
                     B_k \\
                     \mu I
               \end{pmatrix}
                   \right\| \|n_k\| \\
    &\leq \gamma_{p_1} \|D_k\| + \gamma_{p_2}\|D_k\|^2 \|n_k\|,
    \end{aligned}
    \end{equation}
    where $\gamma_{p_1}$ and $\gamma_{p_2}$ are constants.
    
From $\sigma_0 > \hat{\sigma}_{\min}$ and the update rules of $\sigma_k$ in Algorithm \ref{Alg2}, we can get $\sigma_k \geq \hat{\sigma}_{\min}$ for all $k$. From \eqref{pkc}, \eqref{nupp} and \eqref{pkcup}, we have that
    \begin{equation}\label{tupp}
    \begin{aligned}
        \| t_k \|  &\leq 3 \gamma_N^{\frac{5}{2}} \sqrt{\cfrac{\|p_k^c\|}{\sigma_k\|D_k\|^3 }} \leq 3 \gamma_N^{\frac{5}{2}} \sqrt{\cfrac{\gamma_{p_1}}{\sigma_k\|D_k\|^2}+\cfrac{\gamma_{p_2}\|n_k\|}{\sigma_k\|D_k\|}} \\
        &\leq 3 \gamma_N^{\frac{5}{2}} \sqrt{\cfrac{\gamma_{p_1}}{\sigma_k\|D_k\|^2}+\cfrac{\gamma_{p_2} \gamma_n^{\prime \prime}}{\sigma_k^{\frac{3}{2}}\|D_k\|^2}} \leq \frac{3 \gamma_N^{\frac{5}{2}}}{\sqrt{\sigma_k}\|D_k\|} \sqrt{\gamma_{p_1}+\cfrac{\gamma_{p_2} \gamma_n^{\prime \prime}}{\hat{\sigma}_{\min}^{1/2}}} \leq \cfrac{\gamma_t^{\prime \prime}}{ \sqrt{\sigma_k}\|D_k\|} 
        \end{aligned}
    \end{equation}
   for all $ k $, where 
   $ \gamma_t^{\prime \prime} = 3 \gamma_N^{\frac{5}{2}} \sqrt{\gamma_{p_1}+\gamma_{p_2} \gamma_n^{\prime \prime}/\hat{\sigma}_{\min}^{1/2}} $. Thus, based on \eqref{dcomposite}, the above two inequalities \eqref{nupp} and \eqref{tupp} and the Cauchy-Schwarz inequality, there exists a value $ \gamma_d $ such that
    \begin{equation}
      \| d_k \| \leq \| n_k\| +\| t_k \| \leq \cfrac{\gamma_n^{\prime \prime}}{\sqrt{ \sigma_k }\|D_k\|} + \cfrac{\gamma_t^{\prime \prime}}{\sqrt{ \sigma_k }\|D_k\|} \leq\cfrac{\gamma_d^{\prime \prime}}{\sqrt{ \sigma_k}\|D_k\|} = \cfrac{\gamma_d}{\sqrt{\tilde{\sigma}_k}\|D_k\|}, \nonumber
    \end{equation}
   where $ \gamma_d = 2 \gamma_n^{\prime \prime} / \xi^{\frac{3}{2}} $ for all $ k >0$. 
\end{proof}

\subsection{Proof of Lemma \ref{p-al}}\label{le7}
\begin{proof}
 A Taylor expansion of $ f(x_k+d_x) $ and the Lipschitz continuity of $ \nabla f(x) $ give that 
  \begin{eqnarray}\label{fb}
    &&  f(x_k + d_x) - f(x_k) - \nabla f_k^T d_x \nonumber \\
    &\leq &~ \sup_{ \varrho \in [ x_k, x_k +d_x ]} ( \nabla f(\varrho) ^T d_x - \nabla f_k^T d_x ) \leq  \gamma_{fd} \| d_x \|^2.
 \end{eqnarray}
 Using the Lipschitz continuity of $ A $,
  we have that 
   \begin{eqnarray}\label{gb}
  && \left| \left\| g(x_k+d_x) + y_k + d_y \right\| - \left\| g_k + y_k + A_k^T d_x + d_y \right\| \right|\nonumber \\
  &\leq &~ \left\| g(x_k + d_x) - g_k -A_k^T d_x \right\| \leq  \sup_{ \xi \in [ x_k, x_k +d_x ]} \left\| A(\xi) - A(x_k) \right\| \| d_x \|
  \leq  \gamma_{gy} \| d_x \|^2 
  \end{eqnarray}
  for some positive constant $ \gamma_{gy} $. Similarly, for any scalars $\kappa$ and $\kappa^{\prime}$ satisfying $ \kappa > 0 $ and $ \kappa^{\prime} \geq -\tau \kappa $,
  \begin{equation}\label{lnineq}
    \begin{aligned}
      \left| \ln ( \kappa + \kappa^{\prime} ) - \ln \kappa - \cfrac{\kappa^{\prime}}{\kappa} \right|
      &\leq \sup_{ t \in [ \kappa, \kappa +\kappa^{\prime} ]} \left| \cfrac{\kappa^{\prime}}{t} -\cfrac{\kappa^{\prime}}{\kappa} \right| 
      \leq \cfrac{\kappa}{\kappa + \kappa^{\prime}} \left(\cfrac{\kappa^{\prime}}{\kappa}\right)^2 \leq \cfrac{1}{ 1- \tau } \left(\cfrac{\kappa^{\prime}}{\kappa}\right)^2.
    \end{aligned} 
  \end{equation}
  These three inequalities \eqref{fb}, \eqref{gb} and \eqref{lnineq}, the definitions \eqref{pred} and \eqref{ared}, 
  and the bound
  \begin{equation}
    d_x^T B_k d_x \leq \gamma_B \|d_x\|^2 \nonumber
  \end{equation}
  with a positive constant $\gamma_B$ obtained from assumption (AS2) yield that
  \begin{eqnarray*}
      && \left| \text{pred}_k(d_k) - \text{ared}_k(d_k) \right|  \\
      &= & \left| f(x_k + d_x) - f(x_k) - \nabla f_k^T d_x - \cfrac{1}{2} d_x^T B_k d_x - \cfrac{1}{3} \sigma_k \|D_k d_k\|^3 \right. \\
      &&   + \nu_k \left(  \left\| g(x_k+d_x) + y_k + d_y \right\|- \left\| g_k + y_k + A_k^T d_x + d_y \right\|  \right) \\
      && \left. - \mu \sum_{i=1}^m \left[ \ln (y_k+d_y)^{(i)} - \ln y_k^{(i)} - \cfrac{d_y^{(i)}}{y_k^{(i)}} + \cfrac{1}{2} \left( \cfrac{d_y^{(i)}}{y_k^{(i)}} \right)^2   \right]   \right| \\
      &\leq & \gamma_{fd} \|d_x\|^2 + \cfrac{1}{2} \gamma_B \|d_x\|^2 + \nu_k \gamma_{gy} \|d_x\|^2  + \mu \cfrac{1}{1-\tau} \| Y_k^{-1} d_y \|^2 + \cfrac{1}{2} \mu \| Y_k^{-1} d_y \|^2   \\
      &\leq & \left( \gamma_{fd} + \cfrac{1}{2} \gamma_B + \nu_k \gamma_{gy} \right) \|d_x\|^2 + \mu \left( \cfrac{1}{1-\tau} + \cfrac{1}{2} \right) \| Y_k^{-1} d_y \|^2 \\
      &\leq & \gamma_L \left( ( 1+ \nu_k ) \| d_x \|^2 + \| Y_k^{-1} d_y \|^2 \right),
  \end{eqnarray*}
  where $ \gamma_L := \max \left\{ \gamma_{fd} + \cfrac{1}{2} \gamma_B ,  \gamma_{gy} , \mu \left( \cfrac{1}{1-\tau} + \cfrac{1}{2} \right) \right\} $. 
\end{proof}

\subsection{Proof of Lemma \ref{feasib}}\label{le8}
\begin{proof}
  Assume that, for the sake of contradiction, there exists an infinite subsequence indexed by $i$ (where the outer iteration counter $k$ is fixed) with $\sigma_{k,i} < \sigma_{k+1,i} $, and corresponding steps $ d_{k,i} = n_{k,i} + t_{k,i} $ and penalty parameters $ \nu_{k,i} $, such that $ \text{ared}_{k,i} (d_{k,i}) < \eta_2 \text{pred}_{k,i} (d_{k,i}) $ for all $ i $. Since $ \eta_2 \in (0,1) $, this implies $ \left| \text{pred}_{k,i} (d_{k,i}) - \text{ared}_{k,i} (d_{k,i}) \right| > ( 1 - \eta_2 ) \text{pred}_{k,i} (d_{k,i}) $. Then using inequalities (\ref{nkup}) and (\ref{tkbou}), which imply that $d_x^{k,i} \rightarrow 0$ and $d_y^{k,i} \rightarrow 0$, and Lemma \ref{p-al}, we obtain 
  \begin{equation}\label{predki}
    \text{pred}_{k,i} (d_{k,i}) = ( 1 + \nu_{k,i} ) o (\| d_x^{k,i} \|) + o (\| d_y^{k,i} \|).
  \end{equation}
  We shall show that this equation leads to a contradiction, thus proving the lemma.
   To simplify the representation, we omit the arguments in $\text{npred}_{k,i} (n_{k,i})$, $\text{tpred}_{k,i} (t_{k,i})$, and $\text{pred}_{k,i} (d_{k,i})$.
  
  First, consider the case when $ g_k + y_k = 0 $. From (\ref{npred}) and (\ref{npred0}), we have $ \text{npred}_{k,i} = 0 $. Since $ g_k + y_k = 0 $, problem \eqref{norm3} has a solution in the range space of $
   \begin{pmatrix}
      A_k^T & Y_k^2 
   \end{pmatrix}^T
   $. The range space condition (\ref{ranspacon}) implies that $ n_{k,i} $ is of the form (\ref{ranspacon}) for some vector $ \omega_{k,i} $. Therefore, 
   \begin{equation}
   0 = \text{npred}_{k,i} = \left\| ( A_k^T A_k + Y_k^2) \omega_{k,i} ) \right\| + \cfrac{1}{3} \tilde{\sigma}_k \left\| 
   \begin{pmatrix}
      A_k \\
      Y_k
   \end{pmatrix}
   \omega_{k,i} \right\|^3, \nonumber
   \end{equation}
   which implies that $ \omega_{k,i} = 0 $ and $ n_{k,i} = 0 $ because the matrix $ A_k^T A_k + Y_k^2 $ is nonsingular. Given that $\text{npred}_{k,i}$ and $ n_{k,i}$ both vanish, from \eqref{tpredgeq0}, \eqref{repred2} and \eqref{chi}, we have $ \text{pred}_{k,i}  \geq 0 $. Thus, inequality (\ref{penapara}) holds independently of $\nu_{k,i}$, which implies that $ \{ \nu_{k,i} \}_{ i \geq 1 } $ is bounded. Consequently, (\ref{predki}) yields
   \begin{equation}\label{predki2}
    \text{pred}_{k,i} =  o (\| d_x^{k,i} \|) + o (\| d_y^{k,i} \|).
   \end{equation}
   On the other hand, from (\ref{pkc}) and $ n_{k,i} = 0 $, we have that 
   \begin{equation}
    p_k^c = -N_x^T \nabla f_k + \mu N_y^T Y^{-1}_k e. \nonumber
    \end{equation}
   This vector is nonzero; otherwise the KKT conditions of the barrier problem \eqref{barr} and the definition of $ N_k $ would imply that $ \begin{pmatrix}
  x_k \\
  y_k
\end{pmatrix}$ is a stationary point of the problem. 

By the assumption $\sigma_{k,i} < \sigma_{k+1,i} $ of proof by contradiction, $\sigma_{k,i}\to\infty$ as $i\to\infty$.
Then, for $\sigma_{k,i} $ is sufficiently large (when $i$ is large enough), from inequality \eqref{Dup} and (\ref{tkbou}), the boundedness of  $\{N_k\}$ and $\{\nabla f_k\}$, we deduce
   \begin{equation}
     \|t_{k,i}\| \leq 3 \gamma_N^{\frac{5}{2}}  \sqrt{ \frac{\|p_k^c\|}{\sigma_{k,i} \|D_k\|^3}}  \leq 3 \gamma_N^{\frac{5}{2}} \sqrt{\cfrac{\| -N_x^T \nabla f_k + \mu N_y^T Y^{-1}_k e \|}{\sigma_{k,i}\|D_k\|^3}} \leq \frac{\gamma_{t_0} }{\sqrt{\sigma_{k,i}}}, \nonumber
   \end{equation}
   where $ \gamma_{t_0} $ is a constant. Then, from inequality (\ref{tpredboun}), the boundedness of $ \{y_k\} $,
   the mean value inequality, and the fact that $ t_{k,i} = d_{k,i} $, for all $k$ sufficiently large, we obtain 
   \begin{equation}
     \begin{aligned}
       \text{pred}_{k,i}
       &= \text{tpred}_{k,i} \geq \cfrac{\gamma_t}{12\sqrt{2}} \sqrt{\cfrac{\|p_k^c\|^3}{\sigma_k
     \gamma_N^3 \|D_k\|^3}} \geq \cfrac{ \gamma_1^{\prime}}{\sqrt{\sigma_{k,i}}} \geq \cfrac{\gamma_1^{\prime}}{\gamma_{t_0} } \| d_{k,i} \|  \geq \cfrac{\gamma_2^{\prime}}{\sqrt{2}\gamma_{t_0} } \left(\| d_x^{k,i} \| + \| d_y^{k,i} \|\right), \nonumber
     \end{aligned}
   \end{equation}
   where $\gamma_1^{\prime}$ and $\gamma_2^{\prime}$ are constants. This contradicts (\ref{predki2}).
   
   Now consider the case $ g_k + y_k \neq 0 $ . Since the matrix $
   \begin{pmatrix}
      A_k^T & Y_k 
   \end{pmatrix}
   $ has full rank and  $\tilde{\sigma}_{k,i}$ can be sufficiently large, we can deduce from (\ref{npredboun}) that for large $i$
   \begin{equation}\label{npredboun2}
     \text{npred}_{k,i} \geq \cfrac{\gamma_3}{\sqrt{\tilde{\sigma}_{k,i}}}. \nonumber
   \end{equation}
 Then, from \eqref{penapara}, \eqref{Dup}, \eqref{dkup} and the fact that 
 \begin{equation}
 \| d_x^{k,i} \| + \| d_y^{k,i} \| \leq \frac{\sqrt{2} \gamma_d}{\|D_k\|} \sqrt{\cfrac{1}{\tilde{\sigma}_{k,i}}} \nonumber
 \end{equation}
 for all $k $ sufficiently large, we obtain that
   \begin{equation}
     \begin{aligned}
     \text{pred}_{k,i} &~ \geq \delta \nu_{k,i} \text{npred}_{k,i} \geq \delta \nu_{k,i} \cfrac{\gamma_3 }{\sqrt{\tilde{\sigma}_{k,i}}} \geq \frac{\sqrt{2}}{2} \delta \nu_{k,i} \gamma_3 \gamma_d^{-1} \gamma_D^{-1} \left(\| d_x^{k,i} \| + \| d_y^{k,i} \|\right),
     \end{aligned} \nonumber
   \end{equation}
   which contradicts \eqref{predki}. Therefore, the assumption is false, and the lemma holds.
\end{proof}

\subsection{Proof of Lemma \ref{lemma9}}\label{le9}
\begin{proof}
 For the step in $ x $, it yields from \eqref{Dup} and (\ref{dkup}) that
$
  \| x_{k+1} - x_k \| = \| d_x \| \leq \| d_k \| \leq \cfrac{\gamma_d^{\prime \prime} \gamma_D}{\sqrt{\sigma_k}}.
$
  Similarly, for the step in $ y $, it is evident that
 $
  \| y_{k+1} - y_k \| = \| d_y \| \leq \| d_k \| \leq \cfrac{\gamma_d^{\prime \prime}\gamma_D}{\sqrt{\sigma_k}}.
  $
 Hence, 
$
\left\|
    \begin{pmatrix}
      x_{k+1} \\
      y_{k+1}
    \end{pmatrix}-
    \begin{pmatrix}
      x_{k} \\
      y_{k}
    \end{pmatrix}
\right\| \leq \| x_{k+1} - x_k \|  + \| y_{k+1} - y_k \| 
\leq \cfrac{2 \gamma_d^{\prime \prime}\gamma_D}{\sqrt{\sigma_k}}.  
$
Denote $\gamma_{xy} =2 \gamma_d^{\prime \prime}\gamma_D$. Then the conclusion follows.
\end{proof}

\subsection{Proof of Lemma \ref{nandnpred}}\label{le10}
\begin{proof}
  From Lemma \ref{npredlow}, we have that
  \begin{equation}
  \begin{aligned}
    \| g_k + y_k \| \text{npred}_k (n_k)&~ \geq \frac{\gamma_n}{12} \left\| \binom{A_k}{Y_k} ( g_k + y_k ) \right\| \\
    &~ \min \left\{ \sqrt{\cfrac{ \left\|
     \begin{pmatrix}
       A_k \\
       Y_k 
     \end{pmatrix}
     ( g_k + y_k )\right\|}{\tilde{\sigma}_k \| g_k + y_k \| }} \; , \; \cfrac{ \left\|
     \begin{pmatrix}
       A_k \\
       Y_k 
     \end{pmatrix}
     ( g_k + y_k )\right\|}{\left\| 
     \begin{pmatrix}
      A_k^T & Y_k
\end{pmatrix}\right\|^2} \; , \; \xi \tau \right\}.
    \end{aligned} \nonumber
  \end{equation}
  
  If $ g_k + y_k = 0 $, then the above inequality implies $ \text{npred}_k (n_k) = 0 $. We obtain $ n_k = 0 $ as the same proof by Lemma \ref{feasib}, then (\ref{nnpred}) is clearly satisfied.
  
  We prove that in the case of $ g_k + y_k \neq 0 $ below.
  
  Using (\ref{sigmamin}), we have that
   \begin{equation}\label{npred111}
     \text{npred}_k (n_k) \geq \cfrac{ \gamma_n \gamma_{\rm{inf}}}{12}  \min\left\{ \sqrt{\cfrac{ \gamma_{\rm{inf}}}{\tilde{\sigma}_k}} \; , \; \cfrac{ \gamma_{\rm{inf}} \left\| g_k + y_k \right\|}{\gamma_{\rm{sup}}^2}  \; , \; \xi \tau \right\},
   \end{equation}
   where $ \gamma_{\rm{sup}} \geq \sup_k \| \begin{pmatrix}   
A_k^T & Y_k
\end{pmatrix} \| $.
   Without loss of generality, we assume that $\left\| g_k + y_k \right\| < \xi \tau \gamma_{\rm{sup}}^2 / \gamma_{\rm{inf}} $. Then, the minimum value in (\ref{npred111}) cannot occur at $ \xi \tau $, and we can express (\ref{npred111}) as
   \begin{equation}\label{npred112}
     \text{npred}_k (n_k) \geq \cfrac{\gamma_n \gamma_{\rm{inf}}}{12} \min \left\{ \sqrt{\cfrac{ \gamma_{\rm{inf}}}{\tilde{\sigma}_k}} \; , \; \cfrac{ \gamma_{\rm{inf}} \left\| g_k + y_k \right\|}{\gamma_{\rm{sup}}^2}  \right\}.
   \end{equation}
   Now we consider two cases. 
   
   Case 1. Suppose that 
$
    \| g_k + y_k \| \geq \cfrac{ \gamma_{\rm{sup}}^2 }{ \gamma_{\rm{inf}} }\sqrt{\cfrac{ \gamma_{\rm{inf}} }{\tilde{\sigma}_k }}. 
$
   Then, from proof of Lemma \ref{nkuplemma} 
    and conclusion of Lemma \ref{dkuplemma}, there always exist
   \begin{equation}
   \begin{aligned}
     \text{npred}_k (n_k) &~ \geq \cfrac{\gamma_n \gamma_{\rm{inf}} }{12} \min \left\{ \sqrt{\cfrac{ \gamma_{\rm{inf}} }{\tilde{\sigma}_k }}  \; , \;  \sqrt{\cfrac{ \gamma_{\rm{inf}} }{\tilde{\sigma}_k }} \right\} = \cfrac{ \gamma_n \gamma_{\rm{inf}}^{\frac{3}{2}} }{12 \sqrt{\tilde{\sigma}_k}} \geq  \cfrac{\gamma_n \gamma_{\rm{inf}}^{\frac{3}{2}} }{12 \gamma_d} \left\|
     \begin{pmatrix}
       n_x \\
       Y_k^{-1} n_y
     \end{pmatrix}
     \right\| . \nonumber
     \end{aligned}
   \end{equation}   
   From this inequality, we can immediately derive (\ref{nnpred}).
   
   Case 2. Suppose that
   \begin{equation}\label{gyup}
     \| g_k + y_k \| \leq \cfrac{ \gamma_{\rm{sup}}^2 }{ \gamma_{\rm{inf}} }\sqrt{\cfrac{ \gamma_{\rm{inf}} }{\tilde{\sigma}_k }}.
   \end{equation}
   Consider an arbitrary vector $ \bar{n} $ within the range of $\begin{pmatrix}   
A_k^T & Y_k^2
\end{pmatrix}^T$, which gives a lower objective value in the normal subproblem \eqref{norm3} than $ n=0 $. If $\| g_k + y_k \|$ is sufficiently small, then this vector satisfies the constraint of problem \eqref{norm3}. Since $ \bar{n} = \begin{pmatrix}   
A_k^T & Y_k^2
\end{pmatrix}^T \omega $ for some vector $\omega \in \mathbf{R}^m $, we have that
   \begin{equation}
     \| g_k + y_k \| \geq \left\| g_k + y_k + ( A_k^T \quad Y_k ) \binom{A_k}{Y_k} \omega \right\| + \cfrac{1}{3} \tilde{\sigma}_k \left\| \binom{A_k}{Y_k} \omega_k \right\|^3. \nonumber
   \end{equation}
   From $ \tilde{\sigma}_k > 0 $, the above inequality can be expressed as
   \begin{equation}
     \left\|( A_k^T \quad Y_k ) \binom{A_k}{Y_k} \omega \right\|^2 \leq - 2( g_k + y_k )^T ( A_k^T A_k +Y_k^2 ) \omega. \nonumber
   \end{equation}  
   Using the Cauchy-Schwarz inequality, it implies that 
$     \left\|( A_k^T \quad Y_k ) \binom{A_k}{Y_k} \omega \right\| \leq 2 \| g_k + y_k \|.$
   From (\ref{sigmamin}), it follows that
   \begin{equation}\label{nbargy}
     \left\|
\begin{pmatrix}
  \bar{n}_x \\
  Y_k^{-1} \bar{n}_y
\end{pmatrix}
\right\|
     = \left\| \binom{A_k}{Y_k} \omega \right\| \leq \cfrac{2}{\gamma_{\rm{inf}}} \| g_k + y_k \| ,
   \end{equation}  
   where $\bar{n}= \begin{pmatrix}
  \bar{n}_x \\
  \bar{n}_y
\end{pmatrix}$.
   Together with (\ref{gyup}) and (\ref{nbargy}) imply that $ \bar{n} $ is bounded by $ \tilde{\sigma}_k $. Furthermore, for each slack variable $ y^{(i)} $, (\ref{nbargy}) is represented as
   \begin{equation}
     (Y_k^{-1} \bar{n}_y )^{(i)}\geq - 
     \left\|
\begin{pmatrix}
  \bar{n}_x \\
  Y_k^{-1} \bar{n}_y
\end{pmatrix}
\right\|
      \geq - \cfrac{2}{\gamma_{\rm{inf}}} \| g_k + y_k \| \geq -\xi \tau , \nonumber
   \end{equation} 
   provided that $ \| g_k + y_k \| \leq ( \xi \tau \gamma_{\rm{inf}} )/2 $. Thus, $ \bar{n} $ is feasible for problem \eqref{norm3}. The vector  $ \bar{n} $ falls within the range of $ \begin{pmatrix}   
A_k^T & Y_k^2
\end{pmatrix}^T $, and it gives a zero value for the objective function of problem \eqref{norm3}. 
By the above argument, if $\|g_k+y_k\|$ is sufficiently small, $\bar{n}$ is feasible for problem \eqref{norm3}, and therefore $\bar{n}$ is a solution to \eqref{norm3}.
   Given that $ \bar{n} $ is the solution to problem \eqref{norm3} within the range of $ \begin{pmatrix}   
A_k^T & Y_k^2
\end{pmatrix}^T $, the range space condition (\ref{ranspacon}) implies that the normal step $ n_k $ must also fall within the range of $ \begin{pmatrix}   
A_k^T & Y_k^2
\end{pmatrix}^T $. 
From \eqref{nbargy}, we obtain that
   \begin{equation}\label{ngy}
     \left\|
     \begin{pmatrix}
       n_x \\
       Y_k^{-1} n_y
     \end{pmatrix}
     \right\| \leq \cfrac{2}{\gamma_{\rm{inf}}} \| g_k + y_k \| .
   \end{equation}
   From (\ref{npred112}) and (\ref{gyup}), we have that
   \begin{equation}
     \text{npred}_k (n_k) \geq \cfrac{\gamma_n \gamma_{\rm{inf}}^2}{12 \gamma_{\rm{sup}}^2 } \left\| g_k + y_k \right\| \geq \cfrac{\gamma_n \gamma_{\rm{inf}}^3}{24 \gamma_{\rm{sup}}^2 } \left\|
     \begin{pmatrix}
       n_x \\
       Y_k^{-1} n_y
     \end{pmatrix}
     \right\|, \nonumber
   \end{equation}
   which together with (\ref{ngy}) implies (\ref{nnpred}).
\end{proof}

\subsection{Proof of Lemma \ref{nuknubar}}\label{le11}
\begin{proof}
  In Step 4 of ARCBIP, $ \nu_k $ is chosen to be sufficiently large such that
  \begin{equation}\label{penapara2}
   \text{pred}_k(d_k) \geq \delta \nu_k \text{npred}_k (n_k),
  \end{equation}
  where $\text{pred}_k$ is shown in \eqref{repred2} and \eqref{chi}.
  By (\ref{npred}), we have that $ \text{npred}_k \leq \left\| g_k + y_k \right\| $ which can demonstrate the boundedness of $ \{ \text{npred}_k \} $.
  Therefore, from assumption (AS2), (\ref{chi}) and (\ref{nnpred}), there exists a constant $ \gamma_1^{\prime} > 0 $ such that
  \begin{equation}
    - \nabla f_k^T n_x - \frac{1}{2}n_x^T B_k n_x + \mu \left( e^T Y_k^{-1} n_y -\frac{1}{2} n_y^T Y_k^{-2} n_y  \right)  \geq - \gamma_1^{\prime} \text{npred}_k (n_k). \nonumber
  \end{equation}
  If $ \text{npred}_k > 0 $, then we can obtain  that
  \begin{equation}
  \cfrac{1}{2} \nu_k  \text{npred}_k (n_k) + \cfrac{1}{3} \sigma_k \|D_k t_k\|^3 + \cfrac{1}{3} \nu_k \tilde{\sigma}_k \|D_k n_k\|^3 - \cfrac{1}{3} \sigma_k \| D_k d_k \|^3 \geq 0 \nonumber
  \end{equation}
  from \eqref{nu}. If $ \text{npred}_k = 0 $, the result is the same as above. Based on (\ref{repred2}), the reduction in prediction satisfies
  \begin{equation}\label{repred4}
    \text{pred}_k (d_k) \geq \cfrac{1}{2} \nu_k  \text{npred}_k (n_k)+ \text{tpred}_k (t_k) - \gamma_1^{\prime} \text{npred}_k  (n_k).   
  \end{equation}
  
  Given that $ \text{npred}_k $ and $ \text{tpred}_k $ are nonnegative, if $ \nu_k \geq \gamma_1^{\prime} ( \frac{1}{2} - \delta ) $, then (\ref{penapara2}) holds.
  Therefore, if $ \nu_k $ exceeds $ \gamma_1^{\prime} ( \frac{1}{2} - \delta ) $, it will never increase. The above results indicate that after some iterations, as in $k_1$, $\nu_k$ remain unchanged at a certain constant $\bar{\nu}$.
  Currently, (\ref{penapara2}) and (\ref{repred4}) suggest that 
$
    \text{pred}_k (d_k) \geq  \text{tpred}_k (t_k) - \gamma_1^{\prime} \text{npred}_k (n_k) \geq \text{tpred}_k (t_k) - \cfrac{ \gamma_1^{\prime} }{ \delta \nu_k } \text{pred}_k (d_k).   
$
  The result (\ref{predgapttpred}) holds with $ 1/\gamma_{pt}^{\prime} = 1 + \gamma_1^{\prime} / ( \delta \bar{\nu} ) $.
\end{proof}

\subsection{Proof of Lemma \ref{yawfr0}}\label{le12}
\begin{proof}
By Theorem \ref{AYgy0}, $ g_k + y_k \rightarrow 0 $, and thus (\ref{nnpred}) eventually holds at all iterates. Based on Lemma \ref{nuknubar}, we obtain  that $ \nu_k= \bar{\nu} $ for all $ k \geq k_1 $. From ARCBIP we have that
$$
 \phi ( x_k , y_k ; \bar{\nu} ) \leq \phi ( x_{k_1} , y_{k_1} ; \bar{\nu} ) , \; \text{for} \; k \geq k_1.  
$$
Thus
$
  - \mu \sum^m_{i=1}\ln y^{(i)} \leq \phi ( x_{k_1} , y_{k_1} ; \bar{\nu} ) - f_k - \bar{\nu} \| g_k + y_k \|.  
$
Since we assume that $ \{f_k\} $ is bounded below, and $ \{y_k\} $ is bounded, this implies the existence of a vector $ \bar{y} > 0 $, such that $ y_k \geq \bar{y}$ , for $k \geq 1$. Thus, since $ g_k + y_k \rightarrow 0 $, we obtain that $ g_k < 0 $ for large $ k $, thereby $ \{y_k\} $ is bounded away from zero. 
\end{proof}


\end{appendices}

\bibliography{myrefsARCBIPr1}

\end{document}